\documentclass{article} % For LaTeX2e
\usepackage[dvipsnames]{xcolor}         % colors
\usepackage{iclr2024_conference,times}

\usepackage[utf8]{inputenc} % allow utf-8 input
\usepackage[T1]{fontenc}    % use 8-bit T1 fonts
\usepackage{hyperref}       % hyperlinks
\usepackage{url}            % simple URL typesetting
\usepackage{booktabs}       % professional-quality tables
\usepackage{amsfonts}       % blackboard math symbols
\usepackage{nicefrac}       % compact symbols for 1/2, etc.
\usepackage{microtype}      % microtypography
\usepackage{amsmath}
\usepackage{amsthm}
 \usepackage{algorithm}
\usepackage{algorithmicx}
\usepackage{algpseudocode}

\usepackage{subcaption}
\usepackage{multirow}
\usepackage{graphicx, wrapfig}
\usepackage{bm}
\usepackage{tikz}
\usetikzlibrary{arrows.meta, positioning}
\usetikzlibrary{decorations.pathmorphing, decorations.text}
\usepackage{enumitem}
\setitemize[1]{leftmargin=2em, itemsep=0pt, partopsep=0pt, parsep=\parskip, topsep=0pt}

\newtheorem{theorem}{Theorem}
\newtheorem*{theorem*}{Theorem}
\newtheorem{lemma}{Lemma}
\newtheorem{coro}{Corollary}
\newtheorem{rmk}{Remark}
\newtheorem{defi}{Definition}

\newtheorem{assump}{Assumption}
\newtheorem{example}{Example}

\DeclareMathOperator*{\argmax}{argmax}
 \DeclareMathOperator*{\argmin}{argmin}
 \newcommand{\bR}{\mathbb{R}}
 
 \newcommand{\bZ}{\mathbb{Z}}

 \newcommand{\bE}{\mathbb{{E}}}
 \newcommand{\cA}{\mathcal{A}}
  \newcommand{\cP}{\mathcal{P}}
 \newcommand{\cS}{\mathcal{S}}
 \newcommand{\cT}{\widetilde{\mathcal{T}}}
 \newcommand{\cM}{\mathcal{M}}
 \newcommand{\cV}{\bR^{|\cS|}}%{\mathcal{V}}
 
 \newcommand{\cL}{\mathcal{L}}
 \newcommand{\cF}{\mathcal{F}}
 \newcommand{\cD}{\mathcal{D}}
 \newcommand{\hcL}{\widehat{\mathcal{L}}}
 \newcommand{\cX}{\mathcal{X}}
 
\newcommand{\oV}{\overline{V}}
\newcommand{\oQ}{\overline{Q}}
\newcommand{\tV}{\widetilde{V}}
\newcommand{\tQ}{\widetilde{Q}}

 \newcommand{\KL}{\textup{KL}}
 \newcommand{\Etrans}{\bE_{s'\sim P_{s,a}}}
 \newcommand{\hEtrans}[1]{\bE_{s'\sim \widehat{P}^{#1}_{s,a}}}
 \newcommand{\Tvq}{\mathcal{T}_{V\to Q}}
 \newcommand{\Tq}{\mathcal{T}_Q}
 \newcommand{\Tz}{\mathcal{T}_Z}
 \newcommand{\Tzf}{\mathcal{T}_{Z,\mathcal{F}}}
 \newcommand{\hatTzf}{\widehat{\mathcal{T}}_{Z,\mathcal{F}}}
 \newcommand{\hP}[1]{\widehat{P}^{#1}}
 \newcommand{\Psa}{P_{s,a}}
 
 \newcommand{\sigPsa}{\sigma_{\Psa}}
 \newcommand{\hPsa}{\widehat{P}_{s,a}}
  
\newcommand{\proj}{\textup{Proj}}
\newcommand{\projA}{\proj_{\Delta^{|\cA|}}}
\newcommand{\projX}{\proj_\cX}
\newcommand{\<}{\langle}
\renewcommand{\>}{\rangle}
 \renewcommand{\Pr}{\textup{Pr}}
 \newcommand\mydots{\ifmmode\ldots\else\makebox[1em][c]{.\hfil.\hfil.}\thinspace\fi}
\newcommand{\Vmrm}{V_{\textup{MRM}}}
\newcommand{\Tmrm}{\cT_{\textup{MRM}}}

\allowdisplaybreaks[4]

\title{Soft Robust MDPs and Risk-Sensitive MDPs: Equivalence, Policy Gradient, and Sample Complexity}

\author{%
Runyu (Cathy) Zhang\\
Harvard University\\
\texttt{\small runyuzhang@fas.harvard.edu}
\hspace{-5pt}
\And
Yang Hu\\
Harvard University\\
\texttt{\small yanghu@g.harvard.edu}
\And
Na Li \\
Harvard University \\
\texttt{\small nali@seas.harvard.edu}
}

\begin{document}

%title option: Risk-sensitive and Robust MDP: Equivalence, Policy Gradient, and Sample-Complexity
\maketitle

\vspace{-5pt}
\begin{abstract}
\vspace{-5pt}

    {Robust Markov Decision Processes (MDPs) and risk-sensitive MDPs are both powerful tools for making decisions in the presence of uncertainties. Previous efforts have aimed to establish their connections, revealing equivalences in specific formulations. 
    This paper introduces a new formulation for risk-sensitive MDPs, which assesses risk in a slightly different manner compared to the classical Markov risk measure \cite{ruszczynski2010risk}, and establishes its equivalence with a class of soft robust MDP (RMDP) problems, including the standard RMDP as a special case.
       %which captures the risk of a MDP in a slightly different manner compared with the classical Markov risk measure \cite{ruszczynski2010risk}. We establish its equivalence with a class of soft robust MDP (RMDP) problems, which includes the standard RMDP as a special case. 
   Leveraging this equivalence, we further derive the policy gradient theorem for both problems, proving gradient domination and global convergence of the exact policy gradient method under the tabular setting with direct parameterization. This forms a sharp contrast to the Markov risk measure, known to be potentially non-gradient-dominant \cite{huang2021convergence}. We also propose a sample-based offline learning algorithm, namely the robust fitted-Z iteration (RFZI), for a specific soft RMDP problem with a KL-divergence regularization term (or equivalently the risk-sensitive MDP with an entropy risk measure). We showcase its streamlined design and less stringent assumptions due to the equivalence and analyze its sample complexity. }

%This paper focuses on reinforcement learning for the soft robust Markov decision process (MDP) problem, an extension of the robust MDP framework. We first introduce the risk-sensitive MDP and establish the equivalence between risk-sensitive MDP and soft robust MDP. This equivalence offers an alternative perspective for addressing the soft RMDP and enables the design of efficient learning algorithms. Given this equivalence, we further derive the policy gradient theorem for the soft robust MDP problem and prove the global convergence of the exact policy gradient method under the tabular setting with direct parameterization. We also propose a sample-based offline learning algorithm, namely the robust fitted-Z iteration (RFZI), for a specific soft robust MDP problem with a KL-divergence regularization term and analyze the sample complexity of the algorithm. 
\end{abstract}

\vspace{-10pt}
\section{Introduction}
\vspace{-5pt}

Making decisions amidst uncertainty presents a fundamental challenge cutting across diverse domains, including finance \citep{follmer2016stochastic,sriboonchita2009stochastic}, engineering \citep{khan2016dynamic,shahidehpour2003market}, and robotics \citep{thrun2000probabilistic} etc. Within these realms,  decisions carry consequences that depend not only on expected rewards but also on the level of uncertainty and associated risks. Addressing this challenge necessitates approaches such as robust, and risk-sensitive decision-making. These approaches explicitly incorporate uncertainty and aim to find policies that perform well across a spectrum of scenarios and adeptly strike a balance between expected gains and potential risks.

%In the static setting \Lina{it is unclear what "static'' means here because the previous paragraph didn't say anything static, dynamic.}, 

For robust decision-making in a dynamic environment, the robust Markov Decision Process (RMDP) is a popular framework. RMDPs model the environment as a Markov decision process, seeking policies that excel across various potential models. This involves solving a max-min problem, optimizing an objective function that considers the policy's worst-case performance across all models within a defined uncertainty set. The RMDP framework was introduced by \citep{iyengar2005robust,nilim2005robust}, spurring research into efficient planning algorithms when the model is given \citep{gonzalez2002minimax,xu2010distributionally,wiesemann2013robust,yu2015distributionally,mannor2016robust}.{There are also works focusing on the computational facets for these problems \cite{ho2022robust,behzadian2021fast,grand2022convex, derman2021twice} which leverage convex formulation and regularization techniques to tackle robustness.}
 In cases of unknown models \footnote{{By `unknown model' we refer to the setting where the nominal probability transition model is unknown. Both model-based and model-free methods belong to this setting, where model-based methods keep an empirical estimate of the nominal model whereas model-free algorithms don't require this empirical estimation step.}}, recent efforts have designed reinforcement learning (RL) algorithms with guarantees, but most are model-based for tabular cases, i.e., requiring an empirical estimation of the probability transition model \citep{lim2013reinforcement,panaganti2022sample,zhou2021finite,yang2022toward,Xu23,shi2022distributionally}, thereby impeding their applicability to large state spaces.  Some works focus on the model-free setting and employ linear function approximation for handling large state spaces \citep{tamar2014scaling,roy2017reinforcement,badrinath2021robust}. However, these approaches provide only asymptotic guarantees and rely on approximated robust dynamic programming, which inherently is computationally more expensive than standard dynamic programming. A recent contribution by \cite{panaganti2022robust} offers non-asymptotic sample complexity guarantees in the context of model-free robust RL. This achievement, however, introduces additional dual variables, thus adding additional computational complexity and imposing more stringent assumptions. 

An alternative approach for handling uncertainty is risk-sensitive decision-making, which intriguingly shares an elegant equivalence with robust decision-making. The concept of coherent risk measures was initially introduced and explored in 
\citep{artzner1999coherent,delbaen2002coherent,rockafellar2000optimization}, where the uncertainty is represented by a static random variable.  The connection to robustness was established by characterizing risk measures as the infimum of expected shortfall across a set of probability measures, known as the risk envelope.  The risk notion is further extended to convex risk measures which capture a broader class of risk evaluation functions \citep{follmer2002convex,ruszczynski2006optimization,follmer2002robust}. Subsequently, conditional and dynamic risk measures were introduced to generalize risk assessment from static random variables to stochastic processes \citep{artzner2007coherent,cheridito2006dynamic,follmer2006convex,detlefsen2005conditional,riedel2004dynamic,ruszczynski2006conditional,pichler2022risk}. 
In particular, \cite{ruszczynski2010risk} introduces the \textit{Markov risk measure} in the context of Markov Decision Processes (MDPs). However, the equivalence between the Markov risk measure and robust MDPs is not as straightforward as in static settings. Notably, \cite{ruszczynski2010risk,shen2013risk, BAUERLE2022953,bauerle2022distributionally,osogami2012robustness} established the equivalence between optimizing the Markov risk measure and solving a modified RMDP problem, where the uncertainty set dynamically changes with the implemented policy. This differs from the standard RMDPs, where the uncertainty sets are typically unrelated to the policy. Though \cite{osogami2012robustness} attains stronger equivalence results with RMDPs, it is only applicable to specific risk measures, such as Conditional Value at Risk (CVaR). Similar to RMDPs, optimizing Markov risk measures also faces many challenges. Firstly, building upon the equivalence with the modified RMDP with policy-dependent uncertainty set, \citet{huang2021convergence} highlights that, even in a tabular setting with direct parameterization, Markov risk measures may lack gradient-dominance – a stark contrast to the gradient domination observed in standard MDPs \citep{agarwal2020}. This implies that policy gradient algorithms may not ensure global optima, even in a straightforward, full-information environment.  Further, the sample complexity is also harder to obtain. While there is a series of efforts dedicated to optimizing the Markov risk measure within the realm of RL \citep{ccavucs2014computational,shen2014risk,kose2021risk}, these works primarily provide asymptotic convergence results. 

The challenges outlined above motivate us to investigate the potential of introducing an alternative risk formulation.
This new formulation seeks to capture risk in a way similar to Markov risk measures while achieving a stronger and broader equivalence with RMDPs. Moreover, we aim to enhance convergence properties, including the crucial aspect of gradient domination. These improvements are poised to support the development of learning algorithms for both RMDPs and risk-sensitive MDPs while maintaining provable guarantees.

% The aforementioned challenges motivate us to study whether we could propose another slightly different risk formulation that captures the risk in a similar fashion as Markov risk measures but provides a stronger and more general equivalence with RMDP, and a better convergence property (such as gradient domination) which all could be used to facilitate the design of learning algorithms for both RMDPs and risk-sensitive MDPs with provable guarantees.

% ---------------------------------------------------

\noindent\textbf{Our Contributions:}  
In this paper, we propose a new formulation for risk-sensitive MDP, whose definition incorporates the general concepts of convex risk measures. We first establish the equivalence of risk-sensitive MDP with a class of soft RMDP problems, which includes the standard RMDP as a special case. Leveraging this equivalence, we proceed to derive the policy gradient theorem for both the aforementioned class of soft RMDPs and risk-sensitive MDPs (Theorem \ref{theorem:policy-gradient}) and prove the global convergence of the exact policy gradient method under the tabular setting with direct parameterization. Our result, to the best of our knowledge, presents the first global convergence analysis with iteration complexity for a general class of risk-sensitive MDPs. %We also note that although our formulation of risk-sensitive MDP differs from the Markov risk measure, they both yield identical optimal values and optimal policies. Consequently, our results provide an alternative approach for the efficient optimization of the Markov risk measure, which can be seen as a contribution on its own merit. 

Based on the policy gradient theorem, we also highlight the difficulty of gradient estimation using samples compared with the standard MDP setting, motivating us to seek other types of sample-based learning methods. 
In the last part of this paper, {we mainly focus on the setting of offline learning with nonlinear function approximation which is a relatively less-studied scenario, and propose a sample-based offline learning algorithm, namely the robust fitted-Z iteration (RFZI), that resembles policy iteration rather than policy gradient.}%for a specific soft RMDP that is equivalent to a particular risk-sensitive MDP. 
Specifically, we focus on a setting where the regularization term for the RMDP is a KL-divergence term, which is equivalent to the risk-sensitive MDP with the entropy risk measure.  The algorithm utilizes the equivalence between the two problems, which enables simpler algorithm design. {Notably, our algorithm is model-free and does not rely on an empirically estimated probability transition model.}
 %\textcolor{blue}{for one specific soft RMDP/risk-sensitive MDP, where the regularization term for the RMDP is a KL-divergence term and the risk measure associated with the risk-sensitive MDP is the entropy risk measure. The algorithm utilizes the equivalence between the two problems, which enables simpler algorithm design.}
The sample complexity for RFZI is also provided.  Compared with \cite{panaganti2022robust} which considers offline robust RL with sample-complexity guarantees, our work considers a different uncertainty set, requires less computational and implementation complexity, and less stringent assumptions.

%We would like to emphasize that, although it takes a similar recursive form, the value function for our risk-sensitive MDP is \emph{different} from the Markov risk measure considered in most previous literature (e.g. \citep{ruszczynski2006conditional,huang2021convergence,tamar2015policy,tamar2017sequential}). This slight difference enables \emph{a stronger and more general equivalence with (soft) RMDP, a cleaner policy gradient theorem and better convergence property (gradient domination) as well as simpler offline RL algorithm design.}

Due to space limit, we defer a detailed literature review and numerical simulations to the appendix.

%In this paper, we study the iteration and sample complexity of reinforcement learning for a class of soft RMDP problems, which includes the standard RMDP as a special case. To gain a better understanding of the problem, we begin by introducing the risk-sensitive MDP, whose definition incorporates the concept of convex risk measures, and establishing the equivalence between these two problems (Theorem \ref{theorem:dual-representation}). We would like to emphasize that, although it takes a similar recursive form, the value function for our risk-sensitive MDP is \emph{different} from the Markov risk measure considered in most previous literature (e.g. \citep{ruszczynski2006conditional,huang2021convergence,tamar2015policy,tamar2017sequential}, detailed elaboration in Remark \ref{rmk:difference-with-Markov-risk-measure}). This slight difference enables \emph{a stronger and more general equivalence with (reguarized) RMDP, a cleaner policy gradient theorem and better convergence property (gradient domination) as well as simpler offline RL algorithm design.}

%\textcolor{blue}{TODO: In particular, \citet{derman2021twice} build the relationship of robust MDP with a certain risk-related convex programming; In addition to advances in the learning for RMDPs, there are also works focusing on the computational facets for these problems \cite{ho2022robust,behzadian2021fast,grand2022convex, derman2021twice} which leverage convex formulation and regularization techniques to tackle robustness. }
\vspace{-8pt}
\section{Problem Settings and Preliminaries}
\vspace{-8pt}
\paragraph{Markov Decision Processes (MDPs).} A finite Markov decision process (MDP) is defined by a tuple $\cM = (\cS, \cA, P, r, \gamma, \rho)$, where $\cS$ is a finite set of states, $\cA$ is a finite set of actions available to the agent, and $P$ is the transition probability function such that $P(s'|s,a)$ describes the probability of transitioning from one state $s$ to another $s'$ given a particular action $a$. For the sake of notation simplicity, we use $\Psa$ to denote the probability distribution $P(\cdot|s,a)$ over the state space $\cS$.  $r: \cS \times \cA \rightarrow [0,1]$ is a reward function, $\gamma\in [0,1)$ is a discounting factor, and $\rho$ specifies the initial probability distribution over the state space $\cS$.

A stochastic policy $\pi: \cS\to \Delta^{|\cA|}$  specifies a strategy where the agent chooses its action based on the current state in
a stochastic fashion; more specifically, the probability of choosing action $a$ at state $s$ is given by $\Pr(a|s) = \pi(a|s)$. A deterministic policy is a special case of the stochastic policy where for every state $s$ there is an action $a_s$ such that $\pi(a_s|s) = 1$. For notation simplicity, we slightly overload the notation and use $\pi(s)$ to denote the action $a_s$ for deterministic policies. For a given stationary policy $\pi$ and a set of transition probability distributions $\{\Psa\}_{s\in\cS, a\in\cA}$, we denote the discounted state visitation distribution by 
\begin{align*}
  \textstyle  d^{\pi,P}(s):=(1-\gamma)\sum_{t=0}^{+\infty}\gamma^t \Pr^{\pi,P}(s_t = s ~|~ s_0\sim\rho).
\end{align*}

\vspace{-3pt}
\paragraph{Robust MDPs (RMDPs) and Soft Robust MDPs.} Unlike the standard MDP which considers a fixed transition model $\{\Psa\}$, the robust MDP considers a set $\cP$ of transition probability distributions and aims to solve the sup-inf problem \cite{iyengar2005robust}
\vspace{-2pt}
 \begin{equation}\label{eq:robust MDP}
  \textstyle  \sup_{\pi} \inf_{\{\widehat P_t\in\cP\}_{t\ge0}}\bE_{s_t,a_t\sim \pi,\widehat P, s_0\sim \rho} \sum_{t=0}^{+\infty} \gamma^t \left(r(s_t,a_t) \right)
  \footnote{{For the sake of generality, we allow the transition probability to be non-stationary and the policy to be non-Markovian and stochastic. However, in later sections we will show that the sup-inf solution can be obtained by a stationary deterministic Markov policy and a stationary transition probability (Theorem \ref{theorem:dual-representation}).}}
\end{equation}
 where the objective is to find the best action sequence that maximizes a worst-case objective over all possible models in the uncertainty set $\cP$. Many papers \cite{iyengar2005robust,nilim2005robust,yang2022toward,panaganti2022sample,badrinath2021robust} consider the uncertainty set under the $(s,a)$-rectangularity condition  $\cP = \otimes_{s\in\cS,a\in\cA}\cP_{s,a}$, where $\cP_{s,a} = \{\hPsa: \ell(\hPsa,\Psa)\le \epsilon\} $, and $\ell$ is a penalty function that captures the deviation of $\hPsa$ from a nominal model $\Psa$. Some popular penalty functions are KL divergence, total variation distance, etc.
 
 In this paper, we generalize the above robust MDP problem to a wider range of problems which we call the \textit{soft robust MDP}\footnote{{We adopt the term from robust optimization literature, the concept of regularizing the adversaries actions is referred as soft-robustness \cite{ben2010soft} (or comprehensive robustness \cite{ben2006extending} and globalized robustness \cite{ben2017globalized}).}}. %Let $\Omega$ denote all the possible state-transition probabilities.
 The objective of the soft robust MDP solves the following sup-inf problem:
 \vspace{-5pt}
\begin{equation}\label{eq:soft MDP}
 \textstyle   \sup_{\pi} \inf_{\{\widehat P_t\}_{t\ge0}}\bE_{s_t,a_t\sim \pi,\widehat P, s_0\sim \rho} \sum_{t=0}^{+\infty} \gamma^t \left(r(s_t,a_t) + \gamma D(\widehat{P}_{t;s_t,a_t}, P_{s_t, a_t}) \right).
\end{equation}
\vspace{-2pt}
Note that here $\inf_{\{\widehat P_t\}_{t\ge0}}$ is with respect to all the possible state-transition probability distributions. When the penalty function $D$ is chosen as the indicator function%\Lina{Should we have $d$ or $D$?}{Let's use $D$, $d$ is in conflict with the state visitation distribution}
\vspace{-2pt}
\begin{equation*}
   \textstyle D(\hPsa,\Psa) = \left\lbrace \begin{array}{@{}cc}
        0 & \ell(\hPsa, \Psa) \le \epsilon \\
        +\infty & \textup{otherwise}
    \end{array} \right.,
\end{equation*}
\vspace{-2pt}
it recovers the robust MDP problem \eqref{eq:robust MDP}. When $D$ is set as non-indicator functions, for example, $D(\hPsa, \Psa) = \KL(\hPsa||\Psa)$, Problem \eqref{eq:soft MDP} is a robust MDP with a soft penalty term $D$ on the deviation of $\hPsa$ from $\Psa$ rather than a hard constraint on $\hPsa$. 
%is given by a regularization term rather than a hard constraint.

%\cathy{Cathy Book Mark}
Similar to the robust MDP problem, we can define the optimal value function as %\yang{Continue on the naming issues. I do think we should not overload the $V$ and $Q$ notations, since these value functions are not a part of the MDPs by definition. I would suggest explicitly write subscripts like $V^{\star}_r, V^{\pi}_r, Q^{\star}_r, Q^{\pi}_r$ or adding accents like $\overline{V}^{\star}$.} \yang{This probably also makes the statement of Theorem 2 cleaner. Now the notation is overloaded so much that we have to say ``(3) and (7) obtains the same $V$'', and the notation does confuse people if they are not equal.} \Lina{I actually agree with what Yang says. I think by using different notation of $V$ for robust MDP and risk-sensitve MDP, it is easier to read. You can do $\overline{V}$ for robust MDP and $\widetilde{V}$ for risk-senstive MDP}
\vspace{-3pt}
\begin{equation}\label{eq:min-max-value-function-star}
   \textstyle \oV^\star(s)\!:=\! \sup_{\pi}\! \inf_{\{\widehat P_t\}_{t\ge0}}\!\!\bE_{s_t\sim \widehat P} \left[\sum_{t=0}^{+\infty} \gamma^t \left(r(s_t,a_t) + \gamma D(\widehat{P}_{t;s_t,a_t}, P_{s_t, a_t}) \right)\Big|s_0 = s\right].
\end{equation}
\vspace{-3pt}
Additionally, given a stationary policy $\pi$, the value function $\oV^\pi$ under policy $\pi$ is defined as follows:
\vspace{-5pt}
\begin{equation}\label{eq:min-max-value-function-pi}
   \textstyle \oV^\pi(s):= \inf_{\{\widehat P_t\}_{t\ge0}}\bE_{s_t,a_t\sim \pi,\widehat P} \left[\sum_{t=0}^{+\infty} \gamma^t \left(r(s_t,a_t) + \gamma D(\widehat{P}_{t;s_t,a_t}, P_{s_t, a_t}) \right)\Big|s_0 = s\right].
\end{equation}
We also define the corresponding Q-functions as% \Lina{place holder for lina: we may no need to introduce $Q$ here}{The Q function is needed for policy gradient and RFQI so I think it might be better to introduce it in the preliminary section}\Lina{question: do you need $Q$ in the robust MDP form, or do you need $Q$ in the risk-sensitive MDP form? For RFQI, I think you used the $Q$ in the risk-sensitive MDP form. Maybe we don't need to introduce $Q$ here but we should introduce $Q$ in the risk-sensitive MDP. or we should add definitions of risk-sensitive $Q$ in the risk-sensitive MDP. Then in the equivalence theorem, we should say the equivalence of the $Q$ too.}
\begin{align*}
    \textstyle \oQ^\star(s,a)& \textstyle \!:=\! \sup_{\{a_t\}_{t\ge 1}} \!\inf_{\{\widehat P_t\}_{t\ge0}}\!\bE_{s_t\sim\widehat P} \left[\sum_{t=0}^{+\infty} \gamma^t \!\left(r(s_t,a_t) \!+\! \gamma D(\widehat{P}_{t;s_t,a_t},\! P_{s_t, a_t}) \right)\Big|s_0 \!= \!s, a_0 \!=\! a\right]\\
    \textstyle \oQ^\pi(s,a)&\textstyle\!:=\! \inf_{\{\widehat P_t\}_{t\ge0}}\bE_{s_t,a_t\sim \pi, t\ge 1,\widehat P} \left[\sum_{t=0}^{+\infty} \gamma^t \left(r(s_t,a_t) \!+\! \gamma D(\widehat{P}_{t;s_t,a_t}, P_{s_t, a_t}) \right)\Big|s_0 \!=\! s, a_0 \!=\! a\right].
\end{align*}
\begin{rmk}[Soft Robust MDP.]
The soft robust MDP problem is useful, especially when the uncertainty set is not explicitly given. {In this case, it is more desirable to consider all possible probability transition models $\{\hP{}_t\}_{t\ge0}$ while treating the deviation from the nominal model as a soft penalty term $D$ rather than constraining it to be within a specified uncertainty set.} %This is not uncommon because the boundary of the uncertainty set is typically unknown and in many cases, there is no hard boundary of the uncertainty set.
\end{rmk}

In this paper, we establish a connection between the soft robust MDP and another class of MDPs, namely risk-sensitive MDPs. To define the risk-sensitive MDP, we will first introduce the notation of convex risk measures.

%might not perfectly represent the real environment and robustness is desirable to account for uncertain modeling errors. Thus, solving the problem is of great importance. In this paper,  we will build an equivalence connection with another type of MDP problem--risk-sensitive MDP. Then we will develop learning algorithms to solve risk-sensitive MDP and hence the soft robust MDP due to the equivalence.
%\Lina{Try to use the following sentences to add transition/motivation sentences: (soft) robust MDP is useful when .... \cite{}. However, the difficulty lies in how to solve the problem. In this paper, we will build an equivalence connection with another type of MDP problem--risk-sensitive MDP. Then we will develop learning algorithms to solve risk-sensitive MDP and hence the soft robust MDP due to the equivalence.} 
\vspace{-5pt}
\paragraph{Convex Risk Measures \cite{follmer2002convex}.} %To quantify risks in decision-making problems, especially when there is randomness associated with the decision outcomes, people have studied different risk measures \cite{artzner1999coherent,jorion2007value,hamel2011set}. Here we consider a convex risk measure \cite{follmer2002convex}. 
Consider a finite set $\cS$, let $\bR^{|\cS|}$ denote the set of real-valued functions over $\cS$. 
A convex risk measure $\sigma:\bR^{|\cS|}\to \bR$ is a
function that satisfies the following properties:
%Given a probability distribution $\mu\in \Delta^{|\cS|}$ on a finite space $\cS$, the convex measure of risk  $\sigmu: \bR^{|\cS|}\to \bR$ is a function over the set of real-valued functions defined over $\cS$ (without loss of generality, we use $\bR^{|\cS|}$ to denote the set of all real-valued functions over $\cS$) that satisfies the following properties:%\Lina{this definition is unclear. What is $\cV$? not introduced yet. For this intro, not only mathematically, but also a brief sentence of the intuitive meaning. Also it says that the risk measure depends on $\mu$ but the three properties doesn't show that. how $\mu$ does show up in the definition of convex risk measure? Also need to briefly explain $\bR^{|\cS|}$ to make the job easier for readers }

\begin{enumerate}
\item Monotonicity: for any $V', V\in \cV$,  if $V'\le V$, then $\sigma(V)\le \sigma(V').$
    \item Translation invariance: for any $V\in \cV, m \in \bR$, $\sigma(V + m) = \sigma(V) - m$ .
    \item Convexity: for any $V', V\in \cV, \lambda\in[0,1]$, $\sigma(\lambda V + (1-\lambda) V') \le \lambda\sigma(V) + (1-\lambda) \sigma(V').$
\end{enumerate}

\vspace{-5pt}
Using standard duality theory, 
{it is shown in classical results \cite{follmer2002convex} that} convex risk measures satisfy the following dual representation theorem:
\begin{theorem}[Dual Representation Theorem \cite{follmer2002convex}]\label{thm:convex-risk-dual-representation}
The function $\sigma:\bR^{|\cS|}\to \bR$ is a convex risk measure if and only if there exists a ``penalty function'' $D(\cdot): \Delta^{|\cS|}\to \bR$ such that 
\begin{equation}\label{eq:dual-representation-d-to-V}
   \textstyle \sigma(V) = \sup_{\widehat\mu\in\Delta^{|\cS|}} \left(-\bE_{\widehat\mu} V - D(\widehat\mu)\right).
\end{equation}
Further, the penalty function $D$ can be chosen to satisfy the condition $D(\widehat\mu)\ge -\sigma(0)$ for any $\widehat\mu \in \Delta^{|\cS|}$ and it can be taken to be convex and lower-semicontinuous. In specific, it can be written in the following form:
\vspace{-5pt}
\begin{align}\label{eq:dual-representation-V-to-d}
   \textstyle  D(\widehat\mu) =  \sup_V\left( -\sigma(V) - \bE_{s\sim \widehat\mu}V(s)\right)
\end{align}
\end{theorem}
{Note that $\sigma$ and $D$ serve as the Fenchel conjugate of each other.} In most cases, the convex risk measure $\sigma(V)$ can be interpreted as the risk associated with a random variable that takes on values $V(s)$ where $s$ is drawn from some distribution $s\!\sim\!\mu$. Consequently, most commonly used risk measures are typically associated with an underlying probability distribution $\mu\!\in\!\Delta^{|\cS|}$ (e.g., Examples \ref{example:entropy-risk-measure}). This paper focuses on this type of risk measures and thus we use $\sigma(\mu,\cdot)$ to denote the risk measure, where the additional variable $\mu$ indicates the associated probability distribution.  %The larger $\sigma(\mu,V)$ is, the larger the risk is. 
Correspondingly, we denote the penalty term $D(\hat{\mu})$ of $\sigma(\mu,\cdot)$ in the dual representation theorem as $D(\hat{\mu},\mu)$.\footnote{Please note that the symbol $D$ serves a dual purpose, representing both the regularization term in \eqref{eq:soft MDP} and the penalty function for a risk measure in \eqref{eq:dual-representation-d-to-V} and \eqref{eq:dual-representation-V-to-d}. This intentional notation overlap will become clear in the following sections, which reveal the connection between these two terms.}

Here we provide an example of convex risk measure and its dual form. %\Lina{I went to \cite{follmer2002convex}, but I couldn't find exactly the same theorem}
\begin{example}[Entropy risk measure \cite{follmer2002convex}]\label{example:entropy-risk-measure}
For a given $\beta>0$, the entropy risk measure takes the form: 
\begin{equation*}
\textstyle \sigma(\mu, V) =  \beta^{-1}\log\bE_{s\sim \mu}e^{-\beta V(s)}.
\end{equation*}
Its corresponding penalty function $D$ in the dual representation theorem is the KL divergence
\begin{equation*}
  \textstyle  D(\widehat\mu, \mu) = \beta^{-1}\KL(\widehat\mu||\mu) = \beta^{-1}\sum_{s\in\cS} \widehat\mu(s)\log\left({\widehat\mu(s)}/{\mu(s)}\right).
\end{equation*}
\end{example}
\paragraph{Risk-Sensitive MDPs.} 
%\Lina{so far, $\pi$ is not introduced, $V^\pi$ is not intuitively introduced. Do you mean $V^\pi$ is introduced in equation \eqref{eq:markov-risk-measure-pi}}

%\Lina{question: can the risk-sensitive MDP be introduced/defined as something like $\min_\pi \sigma_{P_{s,a}}(\sum_ r(s,a))$, something like this? Right now, there is no clear definition of risk-sensitive MDP. The definition is a fixed point. But what is the risk point? I tried to look at the log-exponential Markov risk measure to get an understanding. But it doesn't show this either.} {This is quite tricky, roughly speaking, in the original paper, they first define certain properties of conditional risk measures and then show that under certain assumptions, this gives the Markov risk measure given by the fixed point equation. I have a sentence in the following paragraph trying to explain this ``The Markov risk measure is closely related to conditional risk measure and dynamic risk measure, and intuitively $-V^\pi$ captures the risk of policy $\pi$, i.e., the larger $V^\pi$ is, the smaller the risk. (Due to space limit, we refer readers to \citep{ruszczynski2010risk} for more detailed introduction on Markov risk measures.)'' }

Convex risk measures capture the risk associated with random variables. It would be desirable if the notion could be adapted to the MDP to capture the risk of a given policy under the Markov process. 
%\Lina{after reading a few literature and my own thinking plus my understanding of what is needed for this paper, I attempted my edition here. The major motivation is due to the difficulty/problem of original risk-sensitive MDP recursive formulation. You can see the red comments below. Hope my attempt does an ok job. It is kind of going back to your original way of only using fix point equation to introduce the risk-sensitive mdp}
{
Given an MDP $\cM$, a class of convex risk measures $\{\sigma(\Psa,\cdot)\}_{s\in\cS, a\in\cA}$, and a policy $\pi(\cdot|s)$, the risk-sensitive value function $\tV^\pi$ for the infinite discounted MDP is given as
\begin{equation}\label{eq:markov-risk-measure-pi}
 \textstyle  \tV^\pi(s) = \sum_a \pi(a|s)\left(r(s,a) - \gamma \sigma(P_{s,a}, \tV^\pi) \right), \forall s\in\cS.
\end{equation}
% To get an intuitive understanding of this fixed point equation, in an MDP with a finite time-horizon $t\in\{0,\ldots, T\}$, we use the following recursion for illustration: for all $s\in \mathcal{S}$,
% \begin{equation}\label{eq:markov-risk-measure-pi-finite}
%  \textstyle  \tV^\pi_t(s) := \sum_a \pi(a|s)\left(r(s,a) - \gamma \sigma(P_{s,a}, \tV^\pi_{t+1}) \right), t=0,1,\ldots, T, \quad \tV_{T+1}^\pi(s)=0.  
% \end{equation}
% In this recursion, each $\tV_t^\pi(\cdot)$ is the value function taking account of the risk induced by the randomness of the state transition $P_{s,a}$ from $t$ to $T$. Similar to the Bellman equation in the risk-neural case (i.e., $\sigma(P_{s,a}, V)\!:=\!-\mathbb{E}_{s'\sim P_{s,a}} V(s')$),  when $T\!\rightarrow\! \infty$, the recursion \eqref{eq:markov-risk-measure-pi-finite} becomes the fix point equation \eqref{eq:markov-risk-measure-pi}. %(Detailed verifications are deferred to the Appendix.)

With the definition of risk-sensitive $\tV^\pi$, the risk-sensitive MDP problem is to find the policy that maximizes $\max_\pi \tV^\pi$. We denote the optimal value by $\tV^\star$, which is the fix-point solution of the following equation,
\vspace{-5pt}
\begin{equation}\label{eq:markov-risk-measure-star}
 \textstyle  \tV^\star(s) := \max_a \left(r(s,a) - \gamma \sigma(P_{s,a}, \tV^\star) \right), \forall s\in\cS.
\end{equation}
{It is worth noting that the fixed point operators for \eqref{eq:markov-risk-measure-pi},\eqref{eq:markov-risk-measure-star} are contractive} (proof deferred to Appendix \ref{apdx:proof-of-Bellman-eq-existence}), which immediately implies the following lemma which verifies that the fixed point equations for $\tV^\pi$ \eqref{eq:markov-risk-measure-pi} and $\tV^\star$ \eqref{eq:markov-risk-measure-star} are well-defined .
\begin{lemma}\label{lemma:Bellman-eq-existence-uniqueness}
The solution to \eqref{eq:markov-risk-measure-pi} exists and is unique. Same argument holds for \eqref{eq:markov-risk-measure-star}.
\end{lemma}
}

\begin{rmk}
    {\color{black} We would like to emphasize that when the policy $\pi$ is stochastic, our definition of the value functions $\tV^\pi$ are different from the Markov risk measures defined in \citep{ruszczynski2010risk,huang2021convergence,tamar2015policy,tamar2017sequential} \footnote{{
Due to this difference, the value function $V^\pi$ can no longer be written as $\rho(\sum_{t=0}^{+\infty}\gamma^t r(s_t,a_t))$ where $\rho$ is a time-consistent dynamic risk measure. This makes our definition different from the usual interpretation of the dynamic risk measures.}}. However, the two quantities are equivalent when $\pi$ is deterministic. {Additionally, when further assuming that the risk measure $\sigma$ is mixture quasiconcave (c.f. \cite{Delage2016DicesionMakingUU}), the optimal policy for the Markov risk measure is also deterministic and thus the risk-sensitive MDP and the Markov risk measure obtain the same optimal value $\tV^\star$} 
%\footnote{{ We would like to note that the equivalence of optimal value only holds for convex risk measures with monotonicity, and might fail if other risk metrics such as the mean moment of order $p$ \cite{Delage2016DicesionMakingUU}, where the monotonicity property does not hold and the optimal policy might not be deterministic. Similarly, it is expected that the equivalence might also fail if we consider adding policy regularization into the value function.}} 
 \footnote{{
    We would like to note that the equivalence of optimal value might fail if $\sigma$ is not mixture semiconcave (e.g. mean (semi)-deviation, mean (semi)-moment measures \cite{Delage2016DicesionMakingUU}) or if policy regularization is added into the value function because the optimal policy might no longer be deterministic.}}   
    (see Appendix \ref{apdx:relationship-to-Markov-risk-measures} for more details).} \vspace{-5pt}
\end{rmk}    
%\Lina{Runyu, for general audience, they won't understand the remark. First of all, I don't think "Markov risk measures" was introduced; secondly, you need to briefly explain why}

    %{\color{blue} Cathy TODO: Put some of it to an independent section in the appendix}
    
We also define the Q-function of the risk sensitive MDP as:
\begin{align*}
    \tQ^\star(s,a):= r(s,a) - \gamma \sigma(P_{s,a}, \tV^\star), \quad\tQ^\pi(s,a):= r(s,a) - \gamma \sigma(P_{s,a}, \tV^\pi).\vspace{-5pt}
\end{align*}
% The Markov risk measure for the MDP is further defined as:
% \begin{equation}\label{eq:markov-risk-measure-star}
%  \textstyle  V^\star(s) = \max_a \left(r(s,a) - \gamma \sigPsa(V^\star) \right).
% \end{equation}
%{bookmark: proof to be finished}
% Remark on the properties of Markov risk measures and why

%\paragraph{Other notations:}
\textbf{Other notations: }{For any function $f:\cS\times\cA\to \bR$, state-action distribution $\mu\in\Delta(\cS\times\cA)$ the $\mu$-weighted 2-norm of $f$ is defined as $\|f\|_{2,\mu} = \left(\bE_{s,a\sim\mu}f(s, a)^2\right)^{1/2}$}. 
%where $\Pr^{\pi,P}(s_t = s~|~s_0\sim\rho)$ is the state visitation probability that $s_t = s$ by executing policy $\pi$ for the MDP with transition model $P$. %\textcolor{red}{Yang: I think the coefficient should be $1-\gamma$.}

%{TODO: find another variable for the penalty function $d$, in conflict with the state distribution...}

%We would like to remark that we intentionally overload the notation $V^\star$ for both the value function for soft RMDP \eqref{eq:min-max-value-function-star} and risk-sensitive MDP \eqref{eq:markov-risk-measure-star}. The notation $V^\pi$ also denote both \eqref{eq:min-max-value-function-pi} and \eqref{eq:markov-risk-measure-pi}. This is because in the later section (Section \ref{section:equivalence}) we are going to show that variables that share the same notation are indeed equivalent under certain conditions.

\vspace{-10pt}
\section{Equivalence of Soft RMDPs and Risk-Sensitive MDPs}\label{section:equivalence}
\vspace{-5pt}
\begin{theorem}[Equivalence of Soft RMDPs and Risk-Sensitive MDPs]\label{theorem:dual-representation}
For a given MDP $\cM$, a penalty function $D$, a class of convex risk measures $\{\sigma(P_{s,a},\cdot)\}$, and a stationary policy $\pi$, if the penalty function $D$ satisfies
\vspace{-5pt}
\begin{equation}\label{eq:equ-D-sigma}
   \textstyle D(\hPsa, \Psa) = \sup_V \left(-\sigma(P_{s,a},V) - \bE_{s'\sim \hPsa}V(s')\right) ,
\end{equation}
then the value functions and Q-functions of the soft RMDP and the risk-sensitive MDP are always the same. That is, ~$\oV^\star\! = \!\tV^\star \!=: \!V^\star, ~~\oV^\pi \!= \!\tV^\pi \!=:\! V^\pi,~~\oQ^\star \!=\! \tQ^\star \!=:\! Q^\star, ~~\oQ^\pi \!=\! \tQ^\pi \!=:\! Q^\pi.$ 
\footnote{{The equivalence $\oV^\pi = \tV^\pi, \oQ^\pi = \tQ^\pi$ easily extends to the setting with policy regularization, since adding regularization only requires changing the reward function $r(s,a)$ to be $r^\pi(s,a) = r(s,a) + \mathcal{R}(\pi(\cdot|s))$, where $\mathcal{R}$ is the policy regularizer, in which case the proof of Theorem 2 can still carry through naturally.}}
%\eqref{eq:min-max-value-function-star} and \eqref{eq:markov-risk-measure-star} obtains the same $V^\star$; \eqref{eq:min-max-value-function-pi} and \eqref{eq:markov-risk-measure-pi} obtain the same $V^\pi$. 
% \vspace{-5pt}
% \begin{align*}
%     \oV^\star\! = \!\tV^\star \!=: \!V^\star, ~~~~\oV^\pi \!= \!\tV^\pi \!=:\! V^\pi,~~~~\oQ^\star \!=\! \tQ^\star \!=:\! Q^\star, ~~~~\oQ^\pi \!=\! \tQ^\pi \!=:\! Q^\pi. 
% \end{align*}

Further, for every initial state $s_0$, the sup-inf solution of the policy and transition probabilities for $V^\star(s_0)$ defined in \eqref{eq:min-max-value-function-star} is given by:
\begin{align}
  \textstyle  \pi^\star(s) &\textstyle= \argmax_a \left(r(s,a) - \gamma \sigma(P_{s,a}, V^\star)\right),\label{eq:pi-star}\\
   \textstyle \hP\star_{t; s,a} &\textstyle= \hP\star_{s,a} = \argmin_{\widehat{P}} D(\widehat{P}, \Psa) + \bE_{s'\sim \widehat{P}}V^\star(s').\notag%\label{eq:hP-star}
\end{align}
%\Lina{notation: the paper is considered infitinite horizon, but in the first equation, there is $0:h-t-1$}
where \eqref{eq:pi-star} means that the optimal action sequence $\{a_t\}_{t\ge 1}$ can be achieved by implementing the deterministic policy $a_t = \pi^\star(s_t)$. %\Lina{I changed $\pi$ to $\pi^\star$}

Similarly, for any initial state $s_0$, the minimum solution of the transition probabilities for $V^\pi(s_0)$ defined in \eqref{eq:min-max-value-function-pi} is given by
\begin{align}
    \qquad~~~~~\textstyle\hP\pi_{t; s,a} = \hP\pi_{s,a} = \argmin_{\widehat{P}} D(\widehat{P}, \Psa) + \bE_{s'\sim \widehat{P}}V^\pi(s').\label{eq:hP-pi}
\end{align}
\end{theorem}
%\textcolor{blue}{bookmark: the signs in the proof needs to be fixed! Also change max min to sup inf.}
Since Theorem \ref{theorem:dual-representation} has established the equivalence of risk-sensitive MDPs and soft RMDPs, \textit{from now on we use $V^\star,V^\pi, Q^\star, Q^\pi$ to denote the value functions and Q-functions for both settings and assume by default that the penalty function $D$ and the risk measure $\sigma$ satisfy relationship ~\eqref{eq:equ-D-sigma}.
}%\Lina{I kind of prefer saying: the penalty function $D$ and the convex risk measure $\sigma$ satisfy relationship ~\eqref{eq:equ-D-sigma}. Using the equation in the theorem is more specific than (5) and (6).}

\begin{rmk}\label{rmk:different-equivalence}
    As a comparison to the equivalence result for the Markov risk measures \cite{ruszczynski2006conditional,ruszczynski2010risk,tamar2015policy}, their uncertainty set for the robust problem generally depends on the policy $\pi$ (see e.g. Assumption 2.2 in \cite{tamar2015policy}), while in our setting, the penalization function $D$ is independent of the policy and matches with the most standard formulation of RMDPs. %Notably, when $\sigma$ is chosen as a coherent risk measure (c.f. \cite{artzner1999coherent}), the risk-sensitive MDP is equivalent to the standard RMDP without the regularization term.}
\end{rmk}

%The proof of Theorem \ref{theorem:dual-representation} is mainly based on the dual representation theorem. However, we would like to remark that it is not a trivial generalization. Due to the non-concavity of $V^\pi$, the dual representation theorem cannot be directly applied. Our proof is established by defining a new Bellman operator for the risk-sensitive MDPs and applying the dual representation theorem on the Bellman operator. The proof starts with the finite horizon case and uses the contraction property of the Bellman operator to generalize it to the infinite horizon discounted case. The detailed proof is deferred to the Appendix.

%\textcolor{blue}{TODO: read \cite{ruszczynski2010risk} more carefully and see how they prove their `dual representation'. Remark on the difference in the proof.}

\vspace{-5pt}
\section{Policy Gradient for Soft RMDPs}
\vspace{-5pt}
%Policy gradient methods are widely recognized as one of the most prominent RL algorithms. Given their popularity, it is reasonable to study whether it is possible to extend these methods to the soft RMDP setting. 
In this section, we present the policy gradient theorem for a differentiable policy $\pi_\theta$ parameterized by $\theta$, which provides an analytical method for computing the gradient in soft RMDPs. Additionally, we prove the global convergence of the exact policy gradient ascent algorithm for the direct parameterization case. For simplicity, in this section we use the abbreviations $V^\theta, Q^\theta, \hP\theta, V^{(t)}, Q^{(t)},\hP{(t)}$ to denote $V^{\pi_\theta},Q^{\pi_\theta},\hP{\pi_\theta}, V^{\pi_{\theta^{(t)}}},Q^{\pi_{\theta^{(t)}}},\hP{\pi_{\theta^{(t)}}}$, respectively.
\begin{theorem}[Policy gradient theorem]\label{theorem:policy-gradient}
Suppose that $\pi_\theta$ is differentiable with respect to $\theta$ and that $\sigma(\Psa,\cdot):\bR^{|\cS|}\to \bR$ is a differentiable function, then $V^\theta(s)$ is also a differentiable function with respect to $\theta$ and the gradient is given by
\begin{equation*}\label{eq:policy-gradient}
  \textstyle  \nabla_\theta V^\theta(s) = \bE_{a_t\sim \pi_\theta(\cdot|s_t), s_{t+1}\sim \hP\theta_{s_t,a_t}} 
  \left[\sum_{t=0}^{+\infty}\gamma^t Q^\theta(s_t,a_t)\nabla_\theta \log \pi_\theta(a_t|s_t)\Big| s_0 = s\right],
\end{equation*}
where $\hP\theta$ is defined in \eqref{eq:hP-pi}.
\end{theorem}
%The policy gradient for soft RMDPs takes a similar form as compared to standard MDPs \citep{sutton1998introduction}, the main difference is that the trajectory is sampled from the probability transition model $\hP\theta$ rather than the nominal model $P$.

%\textcolor{blue}{[runyu:I changed the location of this remark (previously at the very end of this section.) and edit the content a bit.]}\Lina{I think it is better to leave the remark at the end of the section. Firstly it affected the flow. If I see the discussion of difficulties of sample-based gradient estimation, I would like to expect that the following to be sample based methods. Secondly, if it already says sample based method is difficult, I lost interest in reading the convergence of police gradient theorem. Nevertheless, it would be better to move this at the end of the section, which is a very natural flow to motivate the next section. Please change it back. And you can start the remark as "Though Theorem 4 establishes the global convergence of exact policy gradient, sample-based gradient estimation is challenging..."  }\textcolor{blue}{[runyu: I see your point. I changed the location because I feel like the discussion after Theorem 3 and the remark contain repetitive sentences about the trajectory being sampled from a different probability transition. Now I merge them and put it back as before, and refer to the remark to avoid repetition. Do you think this also works?]}

\textcolor{black}{We leave the discussion of this result to the end of this section in Remark \ref{rmk:difficulty-gradient-estimation}.} Theorem \ref{theorem:policy-gradient} immediately implies the following corollary on the policy gradient under direct parameterization \textcolor{blue}{(c.f. \cite{agarwal2020, xiao2022convergence})}, where the parameter $\theta_{s,a}$ directly represents the probability of choosing action $a$ at state $s$, i.e., $\theta_{s,a} = \pi_\theta(a|s)$.
% \begin{equation}\label{eq:direct-parameterization}
%   \textstyle theta_{s,a} = \pi_\theta(a|s) .
% \end{equation}
\begin{coro}[Policy gradient for direct parameterization]
Under direct parameterization,
\vspace{-3pt}
\begin{equation}\label{eq:policy-gradient-direct-param}
  \frac{\partial \bE_{s_0\sim\rho}V^\theta(s_0)}{\partial \theta_{s,a}} = \frac{1}{1-\gamma}d^{\pi_\theta,\hP\theta}(s)Q^\theta(s,a).
\end{equation}
\end{coro}
\vspace{-3pt}
Note that the policy gradient theorem only holds for the case where $\sigma(\Psa,\cdot)$ is differentiable; nevertheless, we can generalize \eqref{eq:policy-gradient-direct-param} to the non-differentiable case by defining the variable $G(\theta)\in \bR^{|\cS|\times|\cA|}$ as follows:
\begin{equation*}
   [G(\theta)]_{s,a} :=  \frac{1}{1-\gamma}d^{\pi_\theta,\hP\theta}(s)Q^\theta(s,a).
\end{equation*}
For both differentiable and non-differentiable cases, we could perform the following (`quasi'-)gradient ascent algorithm:
\begin{align}\label{eq:gradient-ascent}
  \textstyle  \theta^{(t+1)} = \proj_{\cX}(\theta^{(t)} + \eta G(\theta^{(t)})),
\end{align}
where $\cX = \mathop{\otimes}_{s\in \cS}\Delta^{|\cA|}$ denotes the feasible region of $\theta$. For the standard MDP case, it is known that the value function satisfies the gradient domination property under direct parameterization \cite{agarwal2020}, which enables global convergence of the policy gradient algorithm. Similar properties also hold for the soft RMDP/risk-sensitive MDP setting which is shown in the following lemma:
\begin{lemma}[Gradient domination under direct parameterization]\label{lemma:gradient-domination}
    \begin{align*}
       \textstyle \bE_{s_0\sim\rho}V^\star(s_0) - V^\theta(s) \le \left\|\frac{d^{\pi^\star,\hP\theta}}{d^{\pi_\theta,\hP\theta}}\right\|_\infty \max_{\overline\pi} \left<\overline\pi - \pi_\theta, G(\theta)\right>,
    \end{align*}
    where $\left\|\frac{d^{\pi^\star,\hP\theta}}{d^{\pi_\theta,\hP\theta}}\right\|_\infty: = \max_s \frac{d^{\pi^\star,\hP\theta}(s)}{d^{\pi_\theta,\hP\theta}(s)}.$
\end{lemma}

% \begin{rmk}\label{rmk:gradient-domination}
%     The gradient domination property suggests that as long as the term $\left\|\frac{d^{\pi^\star,\hP\theta}}{d^{\pi_\theta,\hP\theta}}\right\|_\infty$ is not infinity, all the first order stationary points are global optimal solutions. {\color{blue} We also want to point out that our results does not violate the negative result for Markov risk measures in \cite{huang2021convergence} because the risk measures are different as discussed in Remark \ref{rmk:difference-with-Markov-risk-measure}.}
% \end{rmk}
The gradient domination property suggests that as long as the term $\left\|\frac{d^{\pi^\star,\hP\theta}}{d^{\pi_\theta,\hP\theta}}\right\|_\infty$ is not infinite, all the first order stationary points are global optimal solutions. Based on this observation, we further derive the convergence rate for the policy gradient algorithm. Before that, we introduce the following sufficient exploration assumption:
\begin{assump}[Sufficient Exploration]\label{assump:sufficient-exploration}For any policy $\pi$,
it holds that $d^{\pi,\hP\pi}(s) > 0$, where $\hP\pi$ is defined as in  \eqref{eq:hP-pi}.
We define the distributional shift factor $M$ to be a constant that satisfies
%\begin{equation}
    $M \ge \frac{1}{d^{\pi,\hP\pi}(s)}$ for all state $s$ and policy $\pi$.
%\end{equation}
\end{assump}
Note that when we start with a initial distribution where $\rho(s)>0$ for every state $s$, the term $M$ can be upper bounded by $\frac{1}{(1-\gamma)\min_s\rho(s)}$. If Assumption \ref{assump:sufficient-exploration} is satisfied, it can be concluded that  $\left\|\frac{d^{\pi^\star,\hP\theta}}{d^{\pi_\theta,\hP\theta}}\right\|_\infty\le M$. Thus we could use gradient domination to derive the global convergence rate.
\begin{theorem}[Convergence rate for exact policy gradient under direct parameterization]\label{theorem:gradient-ascent} Under Assumption \ref{assump:sufficient-exploration}, by setting $\eta = \frac{(1-\gamma)^3}{2|\cA|M}$, running \eqref{eq:gradient-ascent} guarantees that
    \begin{align*}
   \textstyle \sum_{k=1}^K \left(\bE_{s_0\sim\rho}V^\star(s_0) - V^{(k)}(s_0)\right)^2 \le \frac{16|\cA| M^4}{(1-\gamma)^4}.
\end{align*}
Therefore, by setting $K \!\ge\! \frac{16|\cA| M^4}{(1-\gamma)^4 \epsilon^2}$, it is guaranteed that
~$\min_{1\le k\le K} \bE_{s_0\sim\rho}(V^\star(s_0) \!-\! V^{(k)}(s_0)) \!\le\! \epsilon.$
\end{theorem}
%\cathy{It would be nice if you could take a look at the following paragraphs}
If we apply the same proof technique to standard MDPs, the convergence rate is $O\left(\frac{|\cA|M^2}{(1-\gamma)^4}\right)$. The dependency on the distributional shift factor $M$ is worse for soft RMDPs, which is caused by the choice of a smaller stepsize $\eta$ (see Remark \ref{rmk:proof-dependency-on-M} in the Appendix for more details). %\Lina{I find "makes the sufficient ascent value worse'' is confusing, not sure about its meaning. Can you revise it? You can make it longer and I will revise.} 
It is an interesting open question whether this worse dependency is fundamental or just a proof artifact.

\begin{rmk}[Difficulties of Sample-based Gradient Estimation]\label{rmk:difficulty-gradient-estimation}
{Though Theorem \ref{theorem:gradient-ascent} establishes the global convergence of exact policy gradient, it is hard to generalize the result to sample-based settings. Note that the policy gradient in Theorem \ref{theorem:policy-gradient} takes a similar form as compared to standard MDPs \citep{sutton1998introduction}, however, there's a primary distinction that the expectation is taken over trajectories sampled from the probability transition model $\hP\theta$ instead of the nominal model $P$. Consequently, when confined to samples exclusively from the nominal model, estimating this expectation becomes exceptionally challenging, particularly in the context of non-generative models.}
\end{rmk}

\vspace{-5pt}
\section{Offline Reinforcement Learning of the KL-soft RMDP}
\vspace{-5pt}

{ Since the previous section considers learning with full information and studies iteration complexity, the major motivation for this section is to examine sample-based learning for risk sensitivity MDPs and soft robust MDPs. As discussed in Remark \ref{rmk:difficulty-gradient-estimation}, developing sample-based policy gradient learning methods might be difficult, therefore, we seek an alternative sample-based method that resembles policy iteration rather than policy gradient. Specifically, we mainly focus on the setting of offline learning with nonlinear function approximation which is a relatively less-studied scenario.} Moreover, due to the challenge in developing a method for soft MDPs with general $D$ functions (or equivalently for risk sensitive MDPs with general risk functions $\sigma$), in this section, we look into a particular and important case of soft RMDP where the regularization term is the KL-divergence, i.e.,
\vspace{-3pt}
\begin{equation}\label{eq:KL-soft-MDP}
\textstyle    \max_{\pi} \min_{\widehat P_t}\bE_{s_t,a_t\sim \pi,\widehat P_t, s_0\sim \rho} \sum_{t=0}^{+\infty} \gamma^t \left(r(s_t,a_t) + \gamma \beta^{-1} \KL(\widehat P_{t;s_t,a_t}|| P_{s_t,a_t})\right).
\end{equation}
The hyperparameter $\beta$ represents the penalty strength of the deviation of $\widehat P$ from $P$, the smaller $\beta$ is, the larger the penalty strength. From Example \ref{example:entropy-risk-measure} and Theorem \ref{theorem:dual-representation}, the KL-soft RMDP is equivalent to the risk-sensitive MDP problem with the risk measures $\sigma(\Psa,\cdot)$ chosen as the entropy risk measure \begin{equation*}\sigma(\Psa,V) = \beta^{-1}\log\bE_{s'\sim \Psa} e^{-\beta V(s')}.
\end{equation*}
%we look into a particular and important example of the Markov risk measure where the risk measures $\sigPsa$ are chosen as the entropy risk measure in Example \ref{example:entropy-risk-measure}, i.e., 
% \begin{equation*}
% \textstyle    \sigPsa(V) = \beta^{-1}\log\bE_{s'\sim \Psa} e^{-\beta V(s')}.
% \end{equation*}
% Classical results cite?? suggests that the corresponding distance function $d(\cdot,\cdot)$ is the KL-divergence, i.e.,
% \begin{equation*}
%   \textstyle  d(\hPsa, \Psa) = \min_V -\beta^{-1}\log\bE_{s'\sim \Psa} e^{-\beta V(s')} - \bE_{s'\sim \hPsa}V(s') = \beta^{-1}\KL(\hPsa||\Psa).
% \end{equation*}
% Thus, from Example \ref{example:entropy-risk-measure} and Theorem \ref{theorem:dual-representation}, the log-exponential Markov risk measure is equivalent to solving the following problem
% \begin{equation}\label{eq:KL-soft-MDP}
% \textstyle    \max_{\pi} \min_{\widehat P_t}\bE_{s_t,a_t\sim \pi,\widehat P_t, s_0\sim \rho} \sum_{t=0}^{+\infty} \gamma^t \left(r(s_t,a_t) + \gamma \beta^{-1} \KL(\widehat P_{t;s_t,a_t}|| P_{s_t,a_t})\right)
% \end{equation}
In this case, the Bellman equations for the value functions $V^\pi, V^\star, Q^\pi, Q^\star$ are given by:%\Lina{can use the risk-sensitve MDP, soft RMDP more clearly}
\begin{small}
\begin{align*}
   \textstyle V^\pi\!(s) &\textstyle\!=\! \sum_a \!\pi(a|s)Q(s,a),  \hspace{10pt}Q^\pi(s,a) \!=\! r(s,a) \!-\! \gamma\beta^{\!-\!1}\!\log \Etrans e^{-\beta V^\pi\!(s')},\\
\textstyle V^\star\!(s) &\textstyle\!=\! \max_a Q(s,a),\hspace{26pt} Q^\star(s,a) =  r(s,a) - \gamma\beta^{-1}\log \Etrans e^{-\beta V^*(s')}.
\end{align*}
\end{small}
% Further, with some simple calculations, $\hP\pi,\hP\star$ in Theorem \ref{theorem:dual-representation} are given by
% \begin{align*}
%   \textstyle  \hP\pi(s'|s,a) \propto P(s'|s,a) \exp(-\beta V^\pi(s')), ~~
%     \hP\star(s'|s,a) \propto P(s'|s,a) \exp(-\beta V^\star(s')).
% \end{align*}
For notational simplicity, we define the Bellman operator on the Q-functions $\Tq: \bR^{|\cS|\times|\cA|} \to \bR^{|\cS|\times|\cA|}$ as:
\begin{align}\label{eq:Tq}
  \textstyle  [\Tq Q](s,a) := r(s,a) - \gamma\beta^{-1}\log \bE_{s'\sim P(\cdot|s,a)} e^{-\beta\max_{a'}Q(s',a')}.
\end{align}
It is not hard to verify from the above arguments that the optimal Q function $Q^\star$ satisfies $$Q^\star = \Tq Q^\star.$$
% \begin{equation}Q^\star = \Tq Q^\star. \label{eq:Q*utility}
% \end{equation}
% \begin{rmk}[Motivation of KL-soft RMDP]\label{rmk: why-soft RMDP}
%     {
%     We would like to clarify that our focus on the KL-soft RMDP is mainly driven by practical considerations, especially the computations. %Specifically, this choice offers ease of implementation and imposes fewer stringent assumptions. 
%     In typical RMDP learning, each iteration involves solving a challenging max-min problem, often requiring the introduction of a dual variable \cite{panaganti2022robust}, which complicates computations and entails additional assumptions. In contrast, as demonstrated in subsequent sections, learning in the context of KL-soft RMDP does not require introducing a dual variable, significantly simplifying algorithm design and relaxing the algorithm assumptions. Furthermore, as evidenced by numerical simulations (see Appendix \ref{sec:numerics} Figure \ref{fig:simulation-PG-full-robust}), the optimal policy in the KL-soft RMDP exhibits robustness in the presence of model uncertainty. Hence, our motivation for investigating KL-soft RMDP is to offer an alternative, potentially more accessible, and analytically tractable approach to RMDP while maintaining similar robust outcomes.}
% \end{rmk}
\vspace{-6pt}
\paragraph{Offline robust reinforcement learning.} The remainder of the paper focuses on finding the optimal robust policy $\pi^\star$ for the soft robust MDP problem \eqref{eq:KL-soft-MDP}. Specifically, we explore offline robust reinforcement learning algorithms which use a pre-collected dataset $\cD$ to learn $\pi^\star$. The dataset is typically generated under the nominal model $\{\Psa\}_{s\in\cS,a\in\cA}$, such that $\cD = \{s_i, a_i, r_i, s_i'\}_{i=1}^N$, where the state-action pairs $(s_i, a_i) \sim \mu$ are drawn from a specific data-generating distribution $\mu$. %It is important to note that the learning algorithm is unaware of the nominal model, and has no data generated by any other transition models $\{\widehat{P}_{s,a}\}$. As a result, the learning algorithm must rely solely on $\cD$ to learn the behavior with respect to other transition models.

%Similar to the non-robust offline RL, we need to make certain assumptions for the data-generating distribution $\mu$. We first define the concept of robustly admissible distribution as follows:
\begin{defi}[Robustly Admissible Distributions]
A distribution $\nu \!\in \!\Delta^{|\cS|\times|\cA|}$  is robustly admissible if there exists $h\!\ge\! 0$ and a policy $\pi$ and transition probability $\hP{}\!\in\!\{P'\!:\!\KL(P'_{t;s,a}||P_{s,a})\!\le\!\beta\}$ (both can be non-stationary) such that $\nu(s,a) = \Pr(s_h,a_h|s_0\sim\rho, \pi,\hP{})$.
\end{defi}
\vspace{-3pt}
%We would like to note that the notion of robustly admissible distribution is adapted from the notion of admissible distribution defined in \cite{chen2019information}. 
\begin{assump}[Concentrability]\label{assump:concentrability}
The data-generating distribution $\mu$ satisfies concentrability if there exists a constant $C$ such that for any $\nu$ that is robustly admissible, $\max_{s,a}\frac{\nu(s,a)}{\mu(s,a)}\le C.$
\end{assump}
\begin{rmk}
The notion of robustly admissible distribution and concentrability are adapted from the the corresponding notions defined for the standard MDP setting \cite{chen2019information}, where they also demonstrate the necessity of this assumption for standard RL with function approximation. It would be an interesting open question whether Assumption \ref{assump:concentrability} is also necessary for robust RL settings. Recent works for standard offline RL also show that by considering variations of the RL algorithms (e.g. exploring pessimism \cite{xie2021bellman} or the primal-dual formulation \cite{zhan2022offline}), the concentrability assumption can be weakened to single-policy concentrability. Another interesting future direction is to study whether applying similar approaches for the soft RMDP would result in the same improvement.

%Assumption \ref{assump:concentrability} is admittedly stronger than the concentrability defined in \cite{chen2019information} because it need to take other transition probability distributions into consideration 

%The assumption is admittedly very strong. But what is actually needed is a set much smaller than $\{P':\KL(P'_{s,a}||P_{s,a})\le \beta\}$, what is used in the proof is $\hP{}\in \{\hP{\star},\hP{\pi_k},\hP{K}\}$ (defined later in proofs). \textcolor{blue}{\cite{chen2019information} shows the necessity of ... in the standard MDP setting. It would be an interesting open question whether this assumption is ... in the robust learning settings.} (Also, would like to emphasize that the concentrability assumption in \cite{panaganti2022robust} might not be enough to give their claimed results...)
\end{rmk}

\vspace{-6pt}
\subsection{Robust Fitted-Z Iteration (RFZI) }\label{section:RFZI}
\vspace{-5pt}
%Remark \ref{rmk:difficulty-gradient-estimation} highlights the challenge in estimating gradients for sample-based learning, particularly in offline settings. Thus, as a preliminary exploration, we consider an alternative algorithmic approach, \Lina{because I moved the motivation and transition earlier. we can just go head to propose the method.}

The offline robust MDP learning method we propose is Robust fitted-Z iteration (RFZI). The main idea is to utilize the fix point equation $Q^\star = \Tq Q^\star$ with the Bellman operator \eqref{eq:Tq} from the corresponding equivalent risk-sensitive MDP. However, $\Tq$ involves a term $\log\bE_{s'\sim\Psa}$ which is hard to approximate with empirical estimation. Thus, instead of directly solving $Q^*$ using $Q^\star=\Tq Q^\star$, we introduce an auxiliary variable, Z-function and solve a fix point equation for $Z$, which play an important role in our algorithm design and theoretical analysis.
%which tackles the problem from a different perspective. 
\vspace{-10pt}
\paragraph{The Z-functions.}
%We first define a class of auxiliary variables, which we name as the Z-functions, which play an important role in our algorithm design and analysis. 
For a given Q-function $Q:\cS\times\cA\to \mathbb{R}$,  we define its corresponding Z-function as below:
\vspace{-3pt}
\begin{align*}
  \textstyle   Z(s,a):= \Etrans e^{-\beta\max_{a'} Q(s',a')}.
\end{align*}
One can establish the relationship between the Z-function and the Q-function by
\begin{equation*}
 \textstyle   [\Tq Q](s,a) = r(s,a) - \gamma\beta^{-1}\log Z(s,a).
\end{equation*}
Further, we also define the Z-Bellman operator on Z-functions as:
\begin{align*}
 \textstyle   [\Tz Z](s,a) := \Etrans e^{-\beta\max_{a'} (r(s',a') - \gamma\beta^{-1}\log Z(s',a'))}.
\end{align*}
Then $\Tq [\Tq Q]](s,a) = r(s,a) - \gamma\beta^{-1}\log [\Tz Z](s,a).$
% \vspace{-4pt}
% \begin{equation*}
%  \textstyle   [\Tq [\Tq Q]](s,a) = r(s,a) - \gamma\beta^{-1}\log [\Tz Z](s,a).
% \end{equation*}
Thus, instead of solving $Q^\star = \cT_Q Q^\star$, an alternative approach is to solve $Z^\star = \cT_Z Z^\star$ and recover $Q^\star$ by $Q^\star = r \!-\! \gamma\beta^{-1}\!\log Z^\star$. This is the key intuition of our RFZI algorithm. Note that compare with $\Tq$, $\Tz$ eliminates the $\log$ dependency on the expectation term $\bE_{s'\sim\Psa}$, which makes it easier for empirical estimation. %\Lina{it will be helpful to explain why considering $Z^*=T_z Z^*$ not $Q^*=T_Q Q^*$.} \textcolor{blue}{[Runyu: I also put some of the explaination at the beginning of 5.1]}
\vspace{-6pt}
\paragraph{Function approximation and projected Z-Bellman operator.} 
%As mentioned in the previous paragraph, we seek to solve $Z^\star = \cT_Z Z^\star$. 
Given that $\cT_Z$ is a contraction mapping, the solution $Z^\star$ can be obtained by running $Z_{k+1} = \cT_Z Z_k$, $\lim_{k\to+\infty}Z_k = Z^\star.$
However, when the problem considered is of large state space, it is computationally very expensive to compute the Bellman operator $\cT_Z$ exactly. Thus function approximation might be needed to solve the problem approximately. Given a function class $\cF$, we define the projected Z-Bellman operator as:
\begin{align*}
   \textstyle [\Tzf Z](s,a) := \argmin_{Z'\in \cF} \|Z' - \Tz Z\|_{2,\mu}^2.
\end{align*}
One can verify that $\Tzf Z$ is also the minimizer of the following loss function $\cL$, 
\begin{align*}
  \textstyle  \cL(Z',Z)&\textstyle:= \bE_{s,a\sim\mu} \Etrans \left( Z'(s,a) - \exp{\left(-\beta\max_{a'} (r(s',a') - \gamma\beta^{-1}\log Z(s',a'))\right)}\right)^2,
\end{align*}
i.e., $\Tzf Z = \argmin_{Z'\in \cF} \cL(Z', Z)$.

% (A formal proof: 
% \begin{align*}
%     \cL(Z',Z)&= \bE_{s,a\sim\mu} \Etrans \left( Z'(s,a) - e^{-\beta\gamma\max_{a'} (r(s',a') - \beta^{-1}\log Z(s',a'))}\right)^2\\
%     &= \bE_{s,a\sim\mu}\left( Z'(s,a) - \Etrans e^{-\beta\gamma\max_{a'} (r(s',a') - \beta^{-1}\log Z(s',a'))}\right)^2 \\
%     + \bE_{s,a\sim\mu}\Etrans &\left( \Etrans e^{-\beta\gamma\max_{a'} (r(s',a') - \beta^{-1}\log Z(s',a'))} - e^{-\beta\gamma\max_{a'} (r(s',a') - \beta^{-1}\log Z(s',a'))}\right)^2\\
%     & = \bE_{s,a\sim\mu}\left( Z'(s,a) - [\Tz Z](s,a)\right)^2 + C(Z),
% \end{align*}
%where $C(Z')$ is a function that is irrelevant to $Z'$
%)

We make the following assumptions on the expressive power of the function class $\cF$
\begin{assump}[Approximate Completeness]\label{assump:approx-completeness}
$ \sup_{Z\in \cF}\inf_{Z'\in\cF}\|Z' - \Tz Z\|_{2,\mu} \le \epsilon_c$.
\end{assump}
\begin{assump}[Positivity]\label{assump:positivity}
$e^{-\frac{\beta}{1-\gamma}}\le Z\le 1, ~~\forall Z\in \cF.$ 
\end{assump}
\vspace{-5pt}
\paragraph{Approximate the projected Z-Bellman operator with empirical loss minimization.} The computation of the loss function $\cL$ requires knowledge of the empirical model $\Psa$ which the algorithm doesn't have access to. Thus, we introduce the following empirical loss to further approximate the loss function $\cL$. Given an offline data set $\{(s_i, a_i, s_i')\}_{i=1}^N$ generated from the distribution $(s_i, a_i)\!\sim \!\mu, ~s_i'\!\sim\! P_{s_i, a_i}$, we can define the empirical loss $\hcL$ as:
\begin{align*}
   \textstyle
   \hcL(Z',Z)&\textstyle:= \frac{1}{N}\sum_{i=1}^N  \left( Z'(s_i,a_i) - \exp{\left(-\beta\max_{a'} (r(s_i',a') - \gamma\beta^{-1}\log Z(s_i',a'))\right)}\right)^2
\end{align*}
Given the empirical loss $\hcL$ the empirical projected Bellman operator is defined as $\hatTzf Z:= \argmin_{Z'\in \cF} \hcL(Z', Z)$. Our robust fitted Z iteration (RFZI) is essentially updating the Z-functions iteratively by $ Z_{k+1} = \hatTzf Z_k.$
% \begin{align*}
%   \textstyle  Z_{k+1} = \hatTzf Z_k.
% \end{align*}
The detailed algorithm is displayed in Algorithm \ref{alg:RFQI}.
\begin{algorithm}
\caption{Robust Fitted Z Iteration (RFZI)}
\label{alg:RFQI}
\begin{algorithmic}[1]
\State \textbf{Input:} Offline dataset $\cD = (s_i, a_i, r_i, s_i')_{i=1}^N$, function class $\cF$.
\State \textbf{Initialize:} $Z_0 = 1\in \cF$
\For{$k = 0, \dots, K-1$}
\State Update $Z_{k+1} = \argmin_{Z\in \cF} \hcL(Z, Z_k)$.
\EndFor 
\State \textbf{Output:} $\pi_K = \argmax_a r(s,a) - \gamma\beta^{-1}\log Z_K(s,a)$
\end{algorithmic}
\end{algorithm}
\vspace{-5pt}
\subsection{Sample complexity}
%\vspace{-5pt}
This section provides the theoretical guarantee for the convergence of the RFZI algorithm. Due to space limit, we defer the proof sketches as well as detailed proofs to Appendix \ref{apdx:RFQI}.
%\textcolor{blue}{bookmark: The proof needs to be revised, now in the proof, $\beta$ is actually $\gamma\beta$}
\begin{theorem}[Sample complexity for RFZI]\label{theorem:sample-complexity} Suppose Assumption \ref{assump:concentrability},\ref{assump:approx-completeness} and \ref{assump:positivity} hold, then for any $\delta\in (0,1)$, with probability at least $1-\delta$, the policy $\pi_K$ obtained from RFQI algorithm (Algorithm \ref{alg:RFQI}) satisfies:
\vspace{-5pt}
\begin{equation*}
    \bE_{s_0\sim\rho} V^\star(s_0) - V^{\pi_K}(s_0)\le \frac{2\gamma^K}{(1\!-\!\gamma)^2} + \gamma \beta^{-1}e^{\frac{\beta}{1-\gamma}}\frac{2C}{(1\!-\!\gamma)^2}\left(4\sqrt{\frac{2\log(|\cF|)}{N}} \!+\! 5\sqrt{\frac{2\log(8/\delta)}{N}} \!+\!\epsilon_c\right).
\end{equation*}
\end{theorem}
\vspace{-5pt}
%\cathy{It would be nice if you could take a look at the following paragraphs} %\Lina{Your second and third terms in the paragraph is not aligned with the inequality. }
The performance gap in Theorem \ref{theorem:sample-complexity} consists of three parts. The first part $\frac{2\gamma^K}{(1-\gamma)^2}$ captures the effect of $\gamma$-contraction of the Bellman operators. The second term, which is the term with $\gamma \beta^{-1}e^{\frac{\beta}{1-\gamma}}\frac{2C}{(1-\gamma)^2}\epsilon_c$, is related to the approximation error caused by using function approximation. The third term $\gamma \beta^{-1}e^{\frac{\beta}{1-\gamma}}\frac{2C}{(1-\gamma)^2}\left(4\sqrt{\frac{2\log(|\cF|)}{N}} + 5\sqrt{\frac{2\log(8/\delta)}{N}}\right)$ is caused by the error of replacing the projected Z-Bellman operator with its empirical version. %\Lina{can we say anything compared to the literature results? it can be compared to other RMDP or even just to offline learning in standard MDP.}
\begin{rmk}[Comparison and Discussions]
    Under similar Bellman completeness and concentrability assumptions, the sample complexity for risk-neutral offline RL \cite{chen2019information} is $\widetilde O\!\left(\!C\frac{\log|\cF|}{(1-\gamma)^4\epsilon^2}\right)$, while our result gives $O\!\left(\!C^2 \!\left(\!\beta^{-1}e^{\frac{\beta}{1-\gamma}}\!\right)^2\!\frac{\log|\cF|}{(1-\gamma)^4\epsilon^2}\!\right)$ (assuming $\epsilon_c \!=\! 0$). As a consequence of robustness, our bound has a worse dependency on the concentrability factor $C$ and an additional factor $\left(\beta^{-1}e^{\frac{\beta}{1-\gamma}}\right)^2$. Note that the term $\beta^{-1}e^{\frac{\beta}{1-\gamma}}$ first decreases and then increases with $\beta$ as it goes from $0$ to $+\infty$, suggesting that the hyperparameter $\beta$ also affects the learning difficulty of the problem. The choice of $\beta$ should not be either too large or too small, ideally on the same scale with $1\!-\!\gamma$. {It is still unclear to us whether the exponential dependency $e^{\frac{\beta}{1-\gamma}}$ is a proof artifact or intrinsic in our setting, however, there are results under similar settings that suggest this exponential dependency on parameter $\beta$ and the effective length $\frac{1}{1-\gamma}$ is fundamental (e.g. Theorem 3 in \cite{fei2020risk}).} 
    
   {We also compare our performance bound with the RFQI algorithm \cite{panaganti2022robust} which considers a similar offline learning setting and obtains sample complexity $O\left(\frac{\log(|\mathcal{F}||\mathcal{G}|)}{(\beta\epsilon)^2(1-\gamma)^6}\right)$, where $\beta$ in their setting is the radius of the uncertainty set. Note that both results share the same dependency on $\epsilon$ and the concentrability constant $C$. However, the bound in \cite{panaganti2022robust} includes an additional term on the size of the dual variable space $\log |\mathcal{G}|$, whereas we have the exponential dependence term $e^{\frac{2\beta}{1-\gamma}}$.}
\end{rmk}
% \begin{rmk}[The choice of hyperparameter $\beta$] 
% Note that the term $\beta^{-1}e^{\frac{\beta}{1-\gamma}}$ first decrease and then increase with $\beta$ as it goes from $0$ to $+\infty$, which suggests that the hyperparameter $\beta$ also affects the learning difficulty of the problem. The choice of $\beta$ should not be either too large or too small, ideally on the same scale with $1\!-\!\gamma$.
% %Theorem \ref{theorem:sample-complexity} suggests that
%     %TBD: discussion for the choice of $\beta$ - cannot be either too large or too small. Relatively on the scale of $\sim O(1-\gamma)$
% \end{rmk}

\vspace{-5pt}
\section{Conclusions and Discussions}
\vspace{-5pt}
This paper proposes a new formulation of risk-sensitive MDP and establishes its equivalence with the soft robust MDP. This equivalence enables us to develop the policy gradient theorem and prove the global convergence of the exact policy gradient method under direct parameterization. Additionally, for the KL-soft robust MDP (or equivalently the risk-sensitive MDP with entropy risk measure) scenario, we propose a sample-based offline learning algorithm, namely the robust fitted-Z iteration (RFZI), and analyze its sample complexity. 

Our work admittedly has its limitations. Currently, our policy gradient result is limited to the exact gradient case, and further research is needed to extend it to approximate gradients. The RFZI algorithm is specifically designed for KL-soft problems and may be more suitable for small action spaces. Our future work will focus on developing practical algorithms that can handle large or even continuous state and action spaces, as well as generalizing the approach to accommodate different penalty functions.

\section*{Acknowledgement}
This work was funded by NSF AI institute: 2112085,  NSF ECCS: 2328241.
%Our work admittedly has its limitations. Currently, our policy gradient result are for the exact gradient case, and the RFZI algorithm is specifically designed for KL-soft problems and may be more suitable for small action spaces. Our future work will focus on developing practical sample-based algorithms that can handle large or even continuous state and action spaces, as well as generalizing the approach to accommodate different penalty functions.

%For the policy gradient method, our result is currently only restricted to the exact gradient case. For the offline learning algorithm RFQI, it can only handle the KL-soft scenario and is more suitable for small action spaces. Our future work will focus on more practical algorithm design for large or even continuous state and action spaces and generalizing to other penalty functions. 

%\cathy{TODO: Remember to discuss the limitation: only can deal with small action spaces. Only focusing on KL divergence. For policy gradient, currently only exact policy gradient. Next step: more practical algorithm design for large action spaces. Potentially implementing sample based PG}
\bibliographystyle{plainnat}
\bibliography{bib.bib}

\begin{thebibliography}{105}
\providecommand{\natexlab}[1]{#1}
\providecommand{\url}[1]{\texttt{#1}}
\expandafter\ifx\csname urlstyle\endcsname\relax
  \providecommand{\doi}[1]{doi: #1}\else
  \providecommand{\doi}{doi: \begingroup \urlstyle{rm}\Url}\fi

\bibitem[Agarwal et~al.(2020)Agarwal, Kakade, Lee, and Mahajan]{agarwal2020}
Alekh Agarwal, Sham~M. Kakade, Jason~D. Lee, and Gaurav Mahajan.
\newblock On the theory of policy gradient methods: Optimality, approximation,
  and distribution shift, 2020.

\bibitem[Artzner et~al.(1999)Artzner, Delbaen, Eber, and
  Heath]{artzner1999coherent}
Philippe Artzner, Freddy Delbaen, Jean-Marc Eber, and David Heath.
\newblock Coherent measures of risk.
\newblock \emph{Mathematical finance}, 9\penalty0 (3):\penalty0 203--228, 1999.

\bibitem[Artzner et~al.(2007)Artzner, Delbaen, Eber, Heath, and
  Ku]{artzner2007coherent}
Philippe Artzner, Freddy Delbaen, Jean-Marc Eber, David Heath, and Hyejin Ku.
\newblock Coherent multiperiod risk adjusted values and {B}ellman’s
  principle.
\newblock \emph{Annals of Operations Research}, 152:\penalty0 5--22, 2007.

\bibitem[Badrinath and Kalathil(2021)]{badrinath2021robust}
Kishan~Panaganti Badrinath and Dileep Kalathil.
\newblock Robust reinforcement learning using least squares policy iteration
  with provable performance guarantees.
\newblock In \emph{International Conference on Machine Learning}, pages
  511--520. PMLR, 2021.

\bibitem[B{\"a}uerle and Glauner(2022)]{bauerle2022distributionally}
Nicole B{\"a}uerle and Alexander Glauner.
\newblock Distributionally robust {M}arkov decision processes and their
  connection to risk measures.
\newblock \emph{Mathematics of Operations Research}, 47\penalty0 (3):\penalty0
  1757--1780, 2022.

\bibitem[Behzadian et~al.(2021)Behzadian, Petrik, and Ho]{behzadian2021fast}
Bahram Behzadian, Marek Petrik, and Chin~Pang Ho.
\newblock Fast algorithms for $ l_{\infty} $-constrained s-rectangular robust
  {MDP}s.
\newblock \emph{Advances in Neural Information Processing Systems},
  34:\penalty0 25982--25992, 2021.

\bibitem[Bellemare et~al.(2017)Bellemare, Dabney, and
  Munos]{bellemare2017distributional}
Marc~G Bellemare, Will Dabney, and R{\'e}mi Munos.
\newblock A distributional perspective on reinforcement learning.
\newblock In \emph{International conference on machine learning}, pages
  449--458. PMLR, 2017.

\bibitem[Ben-Tal et~al.(2006)Ben-Tal, Boyd, and Nemirovski]{ben2006extending}
Aharon Ben-Tal, Stephen Boyd, and Arkadi Nemirovski.
\newblock Extending scope of robust optimization: Comprehensive robust
  counterparts of uncertain problems.
\newblock \emph{Mathematical Programming}, 107\penalty0 (1-2):\penalty0 63--89,
  2006.

\bibitem[Ben-Tal et~al.(2010)Ben-Tal, Bertsimas, and Brown]{ben2010soft}
Aharon Ben-Tal, Dimitris Bertsimas, and David~B Brown.
\newblock A soft robust model for optimization under ambiguity.
\newblock \emph{Operations research}, 58\penalty0 (4-part-2):\penalty0
  1220--1234, 2010.

\bibitem[Ben-Tal et~al.(2017)Ben-Tal, Brekelmans, Den~Hertog, and
  Vial]{ben2017globalized}
Aharon Ben-Tal, Ruud Brekelmans, Dick Den~Hertog, and Jean-Philippe Vial.
\newblock Globalized robust optimization for nonlinear uncertain inequalities.
\newblock \emph{INFORMS Journal on Computing}, 29\penalty0 (2):\penalty0
  350--366, 2017.

\bibitem[Bäuerle and Glauner(2022)]{BAUERLE2022953}
Nicole Bäuerle and Alexander Glauner.
\newblock {M}arkov decision processes with recursive risk measures.
\newblock \emph{European Journal of Operational Research}, 296\penalty0
  (3):\penalty0 953--966, 2022.
\newblock ISSN 0377-2217.

\bibitem[{\c{C}}avu{\c{s}} and Ruszczy{\'n}ski(2014)]{ccavucs2014computational}
{\"O}zlem {\c{C}}avu{\c{s}} and Andrzej Ruszczy{\'n}ski.
\newblock Computational methods for risk-averse undiscounted transient {M}arkov
  models.
\newblock \emph{Operations Research}, 62\penalty0 (2):\penalty0 401--417, 2014.

\bibitem[Chen and Jiang(2019)]{chen2019information}
Jinglin Chen and Nan Jiang.
\newblock Information-theoretic considerations in batch reinforcement learning.
\newblock In \emph{International Conference on Machine Learning}, pages
  1042--1051. PMLR, 2019.

\bibitem[Cheridito et~al.(2006)Cheridito, Delbaen, and
  Kupper]{cheridito2006dynamic}
Patrick Cheridito, Freddy Delbaen, and Michael Kupper.
\newblock Dynamic monetary risk measures for bounded discrete-time processes.
\newblock \emph{Electronic Journal of Probability}, 2006.

\bibitem[Chow et~al.(2017)Chow, Ghavamzadeh, Janson, and Pavone]{chow2017risk}
Yinlam Chow, Mohammad Ghavamzadeh, Lucas Janson, and Marco Pavone.
\newblock Risk-constrained reinforcement learning with percentile risk
  criteria.
\newblock \emph{The Journal of Machine Learning Research}, 18\penalty0
  (1):\penalty0 6070--6120, 2017.

\bibitem[Coache and Jaimungal(2023)]{coache2023reinforcement}
Anthony Coache and Sebastian Jaimungal.
\newblock Reinforcement learning with dynamic convex risk measures.
\newblock \emph{Mathematical Finance}, 2023.

\bibitem[Delage et~al.(2016)Delage, Kuhn, and
  Wiesemann]{Delage2016DicesionMakingUU}
Erick Delage, Daniel Kuhn, and Wolfram Wiesemann.
\newblock "dice"-sion-making under uncertainty: When can a random decision
  reduce risk?
\newblock \emph{Manag. Sci.}, 65:\penalty0 3282--3301, 2016.
\newblock URL \url{https://api.semanticscholar.org/CorpusID:6113617}.

\bibitem[Delbaen(2002)]{delbaen2002coherent}
Freddy Delbaen.
\newblock Coherent risk measures on general probability spaces.
\newblock \emph{Advances in finance and stochastics: essays in honour of Dieter
  Sondermann}, pages 1--37, 2002.

\bibitem[Derman and Mannor(2020)]{derman2020distributional}
Esther Derman and Shie Mannor.
\newblock Distributional robustness and regularization in reinforcement
  learning.
\newblock \emph{ICML Workshop on Theoretical Foundations of Reinforcement
  Learning}, 2020.

\bibitem[Derman et~al.(2018)Derman, Mankowitz, Mann, and
  Mannor]{derman2018soft}
Esther Derman, Daniel~Jaymin Mankowitz, Timothy~A. Mann, and Shie Mannor.
\newblock Soft-robust actor-critic policy-gradient.
\newblock In \emph{Conference on Uncertainty in Artificial Intelligence}, 2018.

\bibitem[Derman et~al.(2021)Derman, Geist, and Mannor]{derman2021twice}
Esther Derman, Matthieu Geist, and Shie Mannor.
\newblock Twice regularized {MDP}s and the equivalence between robustness and
  regularization.
\newblock \emph{Advances in Neural Information Processing Systems},
  34:\penalty0 22274--22287, 2021.

\bibitem[Detlefsen and Scandolo(2005)]{detlefsen2005conditional}
Kai Detlefsen and Giacomo Scandolo.
\newblock Conditional and dynamic convex risk measures.
\newblock \emph{Finance and stochastics}, 9:\penalty0 539--561, 2005.

\bibitem[Ding et~al.(2023)Ding, Shi, Chi, and Zhao]{ding2023seeing}
Wenhao Ding, Laixi Shi, Yuejie Chi, and Ding Zhao.
\newblock Seeing is not believing: Robust reinforcement learning against
  spurious correlation.
\newblock \emph{Advances in Neural Information Processing Systems}, 2023.

\bibitem[Du et~al.(2020)Du, Kakade, Wang, and Yang]{du2019good}
Simon~S. Du, Sham~M. Kakade, Ruosong Wang, and Lin~F. Yang.
\newblock Is a good representation sufficient for sample efficient
  reinforcement learning?
\newblock In \emph{International Conference on Learning Representations}, 2020.

\bibitem[Eysenbach and Levine(2022)]{eysenbach2021maximum}
Benjamin Eysenbach and Sergey Levine.
\newblock Maximum entropy {RL} (provably) solves some robust {RL} problems.
\newblock In \emph{International Conference on Learning Representations}, 2022.

\bibitem[Fei et~al.(2020)Fei, Yang, Chen, Wang, and Xie]{fei2020risk}
Yingjie Fei, Zhuoran Yang, Yudong Chen, Zhaoran Wang, and Qiaomin Xie.
\newblock Risk-sensitive reinforcement learning: Near-optimal risk-sample
  tradeoff in regret.
\newblock \emph{Advances in Neural Information Processing Systems},
  33:\penalty0 22384--22395, 2020.

\bibitem[Fei et~al.(2021{\natexlab{a}})Fei, Yang, Chen, and
  Wang]{fei2021exponential}
Yingjie Fei, Zhuoran Yang, Yudong Chen, and Zhaoran Wang.
\newblock Exponential {B}ellman equation and improved regret bounds for
  risk-sensitive reinforcement learning.
\newblock \emph{Advances in Neural Information Processing Systems},
  34:\penalty0 20436--20446, 2021{\natexlab{a}}.

\bibitem[Fei et~al.(2021{\natexlab{b}})Fei, Yang, and Wang]{fei2021risk}
Yingjie Fei, Zhuoran Yang, and Zhaoran Wang.
\newblock Risk-sensitive reinforcement learning with function approximation: A
  debiasing approach.
\newblock In \emph{International Conference on Machine Learning}, pages
  3198--3207. PMLR, 2021{\natexlab{b}}.

\bibitem[F{\"o}llmer and Penner(2006)]{follmer2006convex}
Hans F{\"o}llmer and Irina Penner.
\newblock Convex risk measures and the dynamics of their penalty functions.
\newblock \emph{Statistics \& Risk Modeling}, 24\penalty0 (1):\penalty0 61--96,
  2006.

\bibitem[F{\"o}llmer and Schied(2002{\natexlab{a}})]{follmer2002convex}
Hans F{\"o}llmer and Alexander Schied.
\newblock Convex measures of risk and trading constraints.
\newblock \emph{Finance and stochastics}, 6:\penalty0 429--447,
  2002{\natexlab{a}}.

\bibitem[F{\"o}llmer and Schied(2002{\natexlab{b}})]{follmer2002robust}
Hans F{\"o}llmer and Alexander Schied.
\newblock Robust preferences and convex measures of risk.
\newblock \emph{Advances in Finance and Stochastics: Essays in Honour of Dieter
  Sondermann}, pages 39--56, 2002{\natexlab{b}}.

\bibitem[F{\"o}llmer and Schied(2016)]{follmer2016stochastic}
Hans F{\"o}llmer and Alexander Schied.
\newblock Stochastic finance.
\newblock In \emph{Stochastic Finance}. de Gruyter, 2016.

\bibitem[Geibel and Wysotzki(2005)]{geibel2005risk}
Peter Geibel and Fritz Wysotzki.
\newblock Risk-sensitive reinforcement learning applied to control under
  constraints.
\newblock \emph{Journal of Artificial Intelligence Research}, 24:\penalty0
  81--108, 2005.

\bibitem[Gonz{\'a}lez-Trejo et~al.(2002)Gonz{\'a}lez-Trejo,
  Hern{\'a}ndez-Lerma, and Hoyos-Reyes]{gonzalez2002minimax}
JI~Gonz{\'a}lez-Trejo, On{\'e}simo Hern{\'a}ndez-Lerma, and Luis~F Hoyos-Reyes.
\newblock Minimax control of discrete-time stochastic systems.
\newblock \emph{SIAM Journal on Control and Optimization}, 41\penalty0
  (5):\penalty0 1626--1659, 2002.

\bibitem[Grand-Cl{\'e}ment and Petrik(2022)]{grand2022convex}
Julien Grand-Cl{\'e}ment and Marek Petrik.
\newblock On the convex formulations of robust {M}arkov decision processes.
\newblock \emph{arXiv preprint arXiv:2209.10187}, 2022.

\bibitem[HasanzadeZonuzy et~al.(2021)HasanzadeZonuzy, Bura, Kalathil, and
  Shakkottai]{hasanzadezonuzy2021learning}
Aria HasanzadeZonuzy, Archana Bura, Dileep Kalathil, and Srinivas Shakkottai.
\newblock Learning with safety constraints: Sample complexity of reinforcement
  learning for constrained {MDP}s.
\newblock In \emph{Proceedings of the AAAI Conference on Artificial
  Intelligence}, 2021.

\bibitem[Ho et~al.(2022)Ho, Petrik, and Wiesemann]{ho2022robust}
Chin~Pang Ho, Marek Petrik, and Wolfram Wiesemann.
\newblock Robust $\phi$-divergence {MDP}s.
\newblock \emph{Advances in Neural Information Processing Systems},
  35:\penalty0 32680--32693, 2022.

\bibitem[Howard and Matheson(1972)]{howard1972risk}
Ronald~A Howard and James~E Matheson.
\newblock Risk-sensitive {M}arkov decision processes.
\newblock \emph{Management science}, 18\penalty0 (7):\penalty0 356--369, 1972.

\bibitem[Huang et~al.(2021)Huang, Leqi, Lipton, and
  Azizzadenesheli]{huang2021convergence}
Audrey Huang, Liu Leqi, Zachary~C Lipton, and Kamyar Azizzadenesheli.
\newblock On the convergence and optimality of policy gradient for {M}arkov
  coherent risk.
\newblock \emph{arXiv preprint arXiv:2103.02827}, 2021.

\bibitem[Istratescu(1981)]{istratescu1981fixed}
Vasile~I Istratescu.
\newblock \emph{Fixed point theory: an introduction}.
\newblock Springer, 1981.

\bibitem[Iyengar(2005)]{iyengar2005robust}
Garud~N Iyengar.
\newblock Robust dynamic programming.
\newblock \emph{Mathematics of Operations Research}, 30\penalty0 (2):\penalty0
  257--280, 2005.

\bibitem[Jaimungal et~al.(2022)Jaimungal, Pesenti, Wang, and
  Tatsat]{jaimungal2022robust}
Sebastian Jaimungal, Silvana~M Pesenti, Ye~Sheng Wang, and Hariom Tatsat.
\newblock Robust risk-aware reinforcement learning.
\newblock \emph{SIAM Journal on Financial Mathematics}, 13\penalty0
  (1):\penalty0 213--226, 2022.

\bibitem[Kakade and Langford(2002)]{kakade2002approximate}
Sham~M. Kakade and John Langford.
\newblock Approximately optimal approximate reinforcement learning.
\newblock In Claude Sammut and Achim~G. Hoffmann, editors, \emph{Machine
  Learning, Proceedings of the Nineteenth International Conference}, pages
  267--274, 2002.

\bibitem[Khalil et~al.(1996)Khalil, Doyle, and Glover]{khalil1996robust}
IS~Khalil, JC~Doyle, and K~Glover.
\newblock \emph{Robust and optimal control}.
\newblock Prentice hall, 1996.

\bibitem[Khan et~al.(2016)Khan, Hashemi, Paltrinieri, Amyotte, Cozzani, and
  Reniers]{khan2016dynamic}
Faisal Khan, Seyed~Javad Hashemi, Nicola Paltrinieri, Paul Amyotte, Valerio
  Cozzani, and Genserik Reniers.
\newblock Dynamic risk management: a contemporary approach to process safety
  management.
\newblock \emph{Current opinion in chemical engineering}, 14:\penalty0 9--17,
  2016.

\bibitem[K{\"o}se and Ruszczy{\'n}ski(2021)]{kose2021risk}
{\"U}mit K{\"o}se and Andrzej Ruszczy{\'n}ski.
\newblock Risk-averse learning by temporal difference methods with {M}arkov
  risk measures.
\newblock \emph{The Journal of Machine Learning Research}, 22\penalty0
  (1):\penalty0 1800--1833, 2021.

\bibitem[Kumar et~al.(2023)Kumar, Derman, Geist, Levy, and
  Mannor]{kumar2023policy}
Navdeep Kumar, Esther Derman, Matthieu Geist, Kfir~Yehuda Levy, and Shie
  Mannor.
\newblock Policy gradient for rectangular robust markov decision processes.
\newblock In \emph{Thirty-seventh Conference on Neural Information Processing
  Systems}, 2023.

\bibitem[La and Ghavamzadeh(2013)]{la2013actor}
Prashanth La and Mohammad Ghavamzadeh.
\newblock Actor-critic algorithms for risk-sensitive {MDP}s.
\newblock \emph{Advances in neural information processing systems}, 26, 2013.

\bibitem[Li et~al.(2022)Li, Shi, Chen, Chi, and Wei]{li2022settling}
Gen Li, Laixi Shi, Yuxin Chen, Yuejie Chi, and Yuting Wei.
\newblock Settling the sample complexity of model-based offline reinforcement
  learning.
\newblock \emph{arXiv preprint arXiv:2204.05275}, 2022.

\bibitem[Li et~al.(2023)Li, Sutter, and Kuhn]{li2023policy}
Mengmeng Li, Tobias Sutter, and Daniel Kuhn.
\newblock Policy gradient algorithms for robust {MDP}s with non-rectangular
  uncertainty sets.
\newblock \emph{arXiv preprint arXiv:2305.19004}, 2023.

\bibitem[Lim et~al.(2013)Lim, Xu, and Mannor]{lim2013reinforcement}
Shiau~Hong Lim, Huan Xu, and Shie Mannor.
\newblock Reinforcement learning in robust {M}arkov decision processes.
\newblock \emph{Advances in Neural Information Processing Systems}, 26, 2013.

\bibitem[Loomis and Sternberg(1968)]{loomis1968advanced}
Lynn~Harold Loomis and Shlomo Sternberg.
\newblock \emph{Advanced calculus}.
\newblock World Scientific, 1968.

\bibitem[Ma et~al.(2020)Ma, Xia, Zhou, Yang, and Zhao]{ma2020dsac}
Xiaoteng Ma, Li~Xia, Zhengyuan Zhou, Jun Yang, and Qianchuan Zhao.
\newblock Dsac: Distributional soft actor critic for risk-sensitive
  reinforcement learning.
\newblock \emph{arXiv preprint arXiv:2004.14547}, 2020.

\bibitem[Ma et~al.(2022)Ma, Liang, Xia, Zhang, Blanchet, Liu, Zhao, and
  Zhou]{ma2022distributionally}
Xiaoteng Ma, Zhipeng Liang, Li~Xia, Jiheng Zhang, Jose Blanchet, Mingwen Liu,
  Qianchuan Zhao, and Zhengyuan Zhou.
\newblock Distributionally robust offline reinforcement learning with linear
  function approximation.
\newblock \emph{arXiv preprint arXiv:2209.06620}, 2022.

\bibitem[Mankowitz et~al.(2020)Mankowitz, Levine, Jeong, Abdolmaleki,
  Springenberg, Shi, Kay, Hester, Mann, and Riedmiller]{mankowitz2019robust}
Daniel~J. Mankowitz, Nir Levine, Rae Jeong, Abbas Abdolmaleki, Jost~Tobias
  Springenberg, Yuanyuan Shi, Jackie Kay, Todd Hester, Timothy Mann, and Martin
  Riedmiller.
\newblock Robust reinforcement learning for continuous control with model
  misspecification.
\newblock In \emph{International Conference on Learning Representations}, 2020.

\bibitem[Mannor et~al.(2016)Mannor, Mebel, and Xu]{mannor2016robust}
Shie Mannor, Ofir Mebel, and Huan Xu.
\newblock Robust {MDP}s with k-rectangular uncertainty.
\newblock \emph{Mathematics of Operations Research}, 41\penalty0 (4):\penalty0
  1484--1509, 2016.

\bibitem[Nilim and El~Ghaoui(2005)]{nilim2005robust}
Arnab Nilim and Laurent El~Ghaoui.
\newblock Robust control of {M}arkov decision processes with uncertain
  transition matrices.
\newblock \emph{Operations Research}, 53\penalty0 (5):\penalty0 780--798, 2005.

\bibitem[Noorani and Baras(2021{\natexlab{a}})]{noorani2021risk}
Erfaun Noorani and John~S Baras.
\newblock Risk-sensitive reinforce: A monte carlo policy gradient algorithm for
  exponential performance criteria.
\newblock In \emph{2021 60th IEEE Conference on Decision and Control (CDC)},
  pages 1522--1527. IEEE, 2021{\natexlab{a}}.

\bibitem[Noorani and Baras(2021{\natexlab{b}})]{noorani2021riskrobust}
Erfaun Noorani and John~S Baras.
\newblock Risk-sensitive reinforcement learning and robust learning for
  control.
\newblock In \emph{2021 60th IEEE Conference on Decision and Control (CDC)},
  pages 2976--2981. IEEE, 2021{\natexlab{b}}.

\bibitem[Noorani and Baras(2022)]{noorani2022regularization}
Erfaun Noorani and John~S Baras.
\newblock From regularization to risk-sensitivity--and back again.
\newblock \emph{IFAC-PapersOnLine}, 55\penalty0 (15):\penalty0 33--38, 2022.

\bibitem[Noorani et~al.(2022)Noorani, Mavridis, and Baras]{noorani2022risk}
Erfaun Noorani, Christos Mavridis, and John Baras.
\newblock Risk-sensitive reinforcement learning with exponential criteria.
\newblock \emph{arXiv preprint arXiv:2212.09010}, 2022.

\bibitem[Osogami(2012)]{osogami2012robustness}
Takayuki Osogami.
\newblock Robustness and risk-sensitivity in {M}arkov decision processes.
\newblock \emph{Advances in Neural Information Processing Systems}, 25, 2012.

\bibitem[Panaganti and Kalathil(2022)]{panaganti2022sample}
Kishan Panaganti and Dileep Kalathil.
\newblock Sample complexity of robust reinforcement learning with a generative
  model.
\newblock In \emph{International Conference on Artificial Intelligence and
  Statistics}, pages 9582--9602. PMLR, 2022.

\bibitem[Panaganti et~al.(2022)Panaganti, Xu, Kalathil, and
  Ghavamzadeh]{panaganti2022robust}
Kishan Panaganti, Zaiyan Xu, Dileep Kalathil, and Mohammad Ghavamzadeh.
\newblock Robust reinforcement learning using offline data.
\newblock In Alice~H. Oh, Alekh Agarwal, Danielle Belgrave, and Kyunghyun Cho,
  editors, \emph{Advances in Neural Information Processing Systems}, 2022.

\bibitem[Pichler et~al.(2022)Pichler, Liu, and Shapiro]{pichler2022risk}
Alois Pichler, Rui~Peng Liu, and Alexander Shapiro.
\newblock Risk-averse stochastic programming: Time consistency and optimal
  stopping.
\newblock \emph{Operations Research}, 70\penalty0 (4):\penalty0 2439--2455,
  2022.

\bibitem[Pinto et~al.(2017)Pinto, Davidson, Sukthankar, and
  Gupta]{pinto2017robust}
Lerrel Pinto, James Davidson, Rahul Sukthankar, and Abhinav Gupta.
\newblock Robust adversarial reinforcement learning.
\newblock In \emph{International Conference on Machine Learning}, pages
  2817--2826. PMLR, 2017.

\bibitem[Ren et~al.(2023)Ren, Ren, Li, and Dai]{ren2023stochastic}
Tongzheng Ren, Zhaolin Ren, Na~Li, and Bo~Dai.
\newblock Stochastic nonlinear control via finite-dimensional spectral dynamic
  embedding.
\newblock \emph{arXiv preprint arXiv:2304.03907}, 2023.

\bibitem[Riedel(2004)]{riedel2004dynamic}
Frank Riedel.
\newblock Dynamic coherent risk measures.
\newblock \emph{Stochastic processes and their applications}, 112\penalty0
  (2):\penalty0 185--200, 2004.

\bibitem[Rockafellar et~al.(2000)Rockafellar, Uryasev,
  et~al.]{rockafellar2000optimization}
R~Tyrrell Rockafellar, Stanislav Uryasev, et~al.
\newblock Optimization of conditional value-at-risk.
\newblock \emph{Journal of risk}, 2:\penalty0 21--42, 2000.

\bibitem[Roy et~al.(2017)Roy, Xu, and Pokutta]{roy2017reinforcement}
Aurko Roy, Huan Xu, and Sebastian Pokutta.
\newblock Reinforcement learning under model mismatch.
\newblock \emph{Advances in neural information processing systems}, 30, 2017.

\bibitem[Ruszczy{\'n}ski(2010)]{ruszczynski2010risk}
Andrzej Ruszczy{\'n}ski.
\newblock Risk-averse dynamic programming for {M}arkov decision processes.
\newblock \emph{Mathematical programming}, 125:\penalty0 235--261, 2010.

\bibitem[Ruszczy{\'n}ski and
  Shapiro(2006{\natexlab{a}})]{ruszczynski2006conditional}
Andrzej Ruszczy{\'n}ski and Alexander Shapiro.
\newblock Conditional risk mappings.
\newblock \emph{Mathematics of operations research}, 31\penalty0 (3):\penalty0
  544--561, 2006{\natexlab{a}}.

\bibitem[Ruszczy{\'n}ski and
  Shapiro(2006{\natexlab{b}})]{ruszczynski2006optimization}
Andrzej Ruszczy{\'n}ski and Alexander Shapiro.
\newblock Optimization of convex risk functions.
\newblock \emph{Mathematics of operations research}, 31\penalty0 (3):\penalty0
  433--452, 2006{\natexlab{b}}.

\bibitem[Shahidehpour et~al.(2003)Shahidehpour, Yamin, and
  Li]{shahidehpour2003market}
Mohammad Shahidehpour, Hatim Yamin, and Zuyi Li.
\newblock \emph{Market operations in electric power systems: forecasting,
  scheduling, and risk management}.
\newblock John Wiley \& Sons, 2003.

\bibitem[Shen et~al.(2013)Shen, Stannat, and Obermayer]{shen2013risk}
Yun Shen, Wilhelm Stannat, and Klaus Obermayer.
\newblock Risk-sensitive {M}arkov control processes.
\newblock \emph{SIAM Journal on Control and Optimization}, 51\penalty0
  (5):\penalty0 3652--3672, 2013.

\bibitem[Shen et~al.(2014)Shen, Tobia, Sommer, and Obermayer]{shen2014risk}
Yun Shen, Michael~J Tobia, Tobias Sommer, and Klaus Obermayer.
\newblock Risk-sensitive reinforcement learning.
\newblock \emph{Neural computation}, 26\penalty0 (7):\penalty0 1298--1328,
  2014.

\bibitem[Shi and Chi(2022)]{shi2022distributionally}
Laixi Shi and Yuejie Chi.
\newblock Distributionally robust model-based offline reinforcement learning
  with near-optimal sample complexity.
\newblock \emph{arXiv preprint arXiv:2208.05767}, 2022.

\bibitem[Shi et~al.(2023)Shi, Li, Wei, Chen, Geist, and Chi]{shi2023curious}
Laixi Shi, Gen Li, Yuting Wei, Yuxin Chen, Matthieu Geist, and Yuejie Chi.
\newblock The curious price of distributional robustness in reinforcement
  learning with a generative model.
\newblock \emph{Advances in Neural Information Processing Systems}, 2023.

\bibitem[Singh et~al.(2020)Singh, Zhang, and Chen]{singh2020improving}
Rahul Singh, Qinsheng Zhang, and Yongxin Chen.
\newblock Improving robustness via risk averse distributional reinforcement
  learning.
\newblock In \emph{Learning for Dynamics and Control}, pages 958--968. PMLR,
  2020.

\bibitem[Sriboonchita et~al.(2009)Sriboonchita, Wong, Dhompongsa, and
  Nguyen]{sriboonchita2009stochastic}
Songsak Sriboonchita, Wing-Keung Wong, Sompong Dhompongsa, and Hung~T Nguyen.
\newblock \emph{Stochastic dominance and applications to finance, risk and
  economics}.
\newblock CRC Press, 2009.

\bibitem[Stanko and Macek(2019)]{stanko2019risk}
Silvestr Stanko and Karel Macek.
\newblock Risk-averse distributional reinforcement learning: A cvar
  optimization approach.
\newblock In \emph{IJCCI}, pages 412--423, 2019.

\bibitem[Sun et~al.(2021)Sun, Liu, Zhao, Yao, Jui, and Kong]{sun2021exploring}
Ke~Sun, Yi~Liu, Yingnan Zhao, Hengshuai Yao, Shangling Jui, and Linglong Kong.
\newblock Exploring the training robustness of distributional reinforcement
  learning against noisy state observations.
\newblock In \emph{ECML/PKDD}, 2021.

\bibitem[Sutton et~al.(1998)Sutton, Barto, et~al.]{sutton1998introduction}
Richard~S Sutton, Andrew~G Barto, et~al.
\newblock \emph{Introduction to reinforcement learning}, volume 135.
\newblock MIT press Cambridge, 1998.

\bibitem[Tamar et~al.(2012)Tamar, Di~Castro, and Mannor]{tamar2012policy}
Aviv Tamar, Dotan Di~Castro, and Shie Mannor.
\newblock Policy gradients with variance related risk criteria.
\newblock In \emph{Proceedings of the twenty-ninth international conference on
  machine learning}, pages 387--396, 2012.

\bibitem[Tamar et~al.(2014)Tamar, Mannor, and Xu]{tamar2014scaling}
Aviv Tamar, Shie Mannor, and Huan Xu.
\newblock Scaling up robust {MDP}s using function approximation.
\newblock In \emph{International conference on machine learning}, pages
  181--189. PMLR, 2014.

\bibitem[Tamar et~al.(2015)Tamar, Chow, Ghavamzadeh, and
  Mannor]{tamar2015policy}
Aviv Tamar, Yinlam Chow, Mohammad Ghavamzadeh, and Shie Mannor.
\newblock Policy gradient for coherent risk measures.
\newblock \emph{Advances in neural information processing systems}, 28, 2015.

\bibitem[Tamar et~al.(2017)Tamar, Chow, Ghavamzadeh, and
  Mannor]{tamar2017sequential}
Aviv Tamar, Yinlam Chow, Mohammad Ghavamzadeh, and Shie Mannor.
\newblock Sequential decision making with coherent risk.
\newblock \emph{IEEE Transactions on Automatic Control}, 62\penalty0
  (7):\penalty0 3323--3338, 2017.
\newblock \doi{10.1109/TAC.2016.2644871}.

\bibitem[Thrun(2000)]{thrun2000probabilistic}
Sebastian Thrun.
\newblock Probabilistic algorithms in robotics.
\newblock \emph{Ai Magazine}, 21\penalty0 (4):\penalty0 93--93, 2000.

\bibitem[Wang et~al.(2022{\natexlab{a}})Wang, Ho, and
  Petrik]{wang2022convergence}
Qiuhao Wang, Chin~Pang Ho, and Marek Petrik.
\newblock Policy gradient in robust mdps with global convergence guarantee.
\newblock In \emph{International Conference on Machine Learning},
  2022{\natexlab{a}}.

\bibitem[Wang and Zou(2021)]{wang2021online}
Yue Wang and Shaofeng Zou.
\newblock Online robust reinforcement learning with model uncertainty.
\newblock \emph{Advances in Neural Information Processing Systems},
  34:\penalty0 7193--7206, 2021.

\bibitem[Wang and Zou(2022)]{wang2022policy}
Yue Wang and Shaofeng Zou.
\newblock Policy gradient method for robust reinforcement learning.
\newblock In \emph{International Conference on Machine Learning}, pages
  23484--23526. PMLR, 2022.

\bibitem[Wang et~al.(2022{\natexlab{b}})Wang, Miao, and Zou]{wang2022robust}
Yue Wang, Fei Miao, and Shaofeng Zou.
\newblock Robust constrained reinforcement learning.
\newblock \emph{arXiv preprint arXiv:2209.06866}, 2022{\natexlab{b}}.

\bibitem[Wiesemann et~al.(2013)Wiesemann, Kuhn, and
  Rustem]{wiesemann2013robust}
Wolfram Wiesemann, Daniel Kuhn, and Ber{\c{c}} Rustem.
\newblock Robust {M}arkov decision processes.
\newblock \emph{Mathematics of Operations Research}, 38\penalty0 (1):\penalty0
  153--183, 2013.

\bibitem[Xiao(2022)]{xiao2022convergence}
Lin Xiao.
\newblock On the convergence rates of policy gradient methods.
\newblock \emph{The Journal of Machine Learning Research}, 23\penalty0
  (1):\penalty0 12887--12922, 2022.

\bibitem[Xie et~al.(2021)Xie, Cheng, Jiang, Mineiro, and
  Agarwal]{xie2021bellman}
Tengyang Xie, Ching-An Cheng, Nan Jiang, Paul Mineiro, and Alekh Agarwal.
\newblock {B}ellman-consistent pessimism for offline reinforcement learning.
\newblock \emph{Advances in neural information processing systems},
  34:\penalty0 6683--6694, 2021.

\bibitem[Xu and Mannor(2010)]{xu2010distributionally}
Huan Xu and Shie Mannor.
\newblock Distributionally robust {M}arkov decision processes.
\newblock \emph{Advances in Neural Information Processing Systems}, 23, 2010.

\bibitem[Xu et~al.(2022)Xu, Liu, Huang, Ding, Cen, Li, and
  Zhao]{xu2022trustworthy}
Mengdi Xu, Zuxin Liu, Peide Huang, Wenhao Ding, Zhepeng Cen, Bo~Li, and Ding
  Zhao.
\newblock Trustworthy reinforcement learning against intrinsic vulnerabilities:
  Robustness, safety, and generalizability.
\newblock \emph{arXiv preprint arXiv:2209.08025}, 2022.

\bibitem[Xu et~al.(2023)Xu, Panaganti, and Kalathil]{Xu23}
Zaiyan Xu, Kishan Panaganti, and Dileep Kalathil.
\newblock Improved sample complexity bounds for distributionally robust
  reinforcement learning.
\newblock In Francisco Ruiz, Jennifer Dy, and Jan-Willem van~de Meent, editors,
  \emph{Proceedings of The 26th International Conference on Artificial
  Intelligence and Statistics}, volume 206 of \emph{Proceedings of Machine
  Learning Research}, pages 9728--9754. PMLR, 25--27 Apr 2023.

\bibitem[Yang et~al.(2022)Yang, Zhang, and Zhang]{yang2022toward}
Wenhao Yang, Liangyu Zhang, and Zhihua Zhang.
\newblock Toward theoretical understandings of robust {M}arkov decision
  processes: Sample complexity and asymptotics.
\newblock \emph{The Annals of Statistics}, 50\penalty0 (6):\penalty0
  3223--3248, 2022.

\bibitem[Yu and Xu(2015)]{yu2015distributionally}
Pengqian Yu and Huan Xu.
\newblock Distributionally robust counterpart in {M}arkov decision processes.
\newblock \emph{IEEE Transactions on Automatic Control}, 61\penalty0
  (9):\penalty0 2538--2543, 2015.

\bibitem[Yu and Ying(2023)]{yu2023global}
Xian Yu and Lei Ying.
\newblock On the global convergence of risk-averse policy gradient methods with
  expected conditional risk measures.
\newblock In \emph{Proceedings of the 40th International Conference on Machine
  Learning}, pages 40425--40451, 2023.

\bibitem[Zhan et~al.(2022)Zhan, Huang, Huang, Jiang, and Lee]{zhan2022offline}
Wenhao Zhan, Baihe Huang, Audrey Huang, Nan Jiang, and Jason Lee.
\newblock Offline reinforcement learning with realizability and single-policy
  concentrability.
\newblock In \emph{Conference on Learning Theory}, pages 2730--2775. PMLR,
  2022.

\bibitem[Zhang et~al.(2020{\natexlab{a}})Zhang, Chen, Xiao, Li, Liu, Boning,
  and Hsieh]{zhang2020robust}
Huan Zhang, Hongge Chen, Chaowei Xiao, Bo~Li, Mingyan Liu, Duane Boning, and
  Cho-Jui Hsieh.
\newblock Robust deep reinforcement learning against adversarial perturbations
  on state observations.
\newblock \emph{Advances in Neural Information Processing Systems},
  33:\penalty0 21024--21037, 2020{\natexlab{a}}.

\bibitem[Zhang et~al.(2020{\natexlab{b}})Zhang, Hu, and Basar]{zhang2021policy}
Kaiqing Zhang, Bin Hu, and Tamer Basar.
\newblock Policy optimization for $\mathcal{H}_2$ linear control with
  $\mathcal{H}_\infty$ robustness guarantee: Implicit regularization and global
  convergence.
\newblock In \emph{Learning for Dynamics and Control}, pages 179--190. PMLR,
  2020{\natexlab{b}}.

\bibitem[Zhou et~al.(2021)Zhou, Zhou, Bai, Qiu, Blanchet, and
  Glynn]{zhou2021finite}
Zhengqing Zhou, Zhengyuan Zhou, Qinxun Bai, Linhai Qiu, Jose Blanchet, and
  Peter Glynn.
\newblock Finite-sample regret bound for distributionally robust offline
  tabular reinforcement learning.
\newblock In \emph{International Conference on Artificial Intelligence and
  Statistics}, pages 3331--3339. PMLR, 2021.

\end{thebibliography}
%%%%%%%%%%%%%%%%%%%%%%%%%%%%%%%%%%%%%%%%%%%%%%%%%%%%%%%%%%%%
\appendix
\newpage
\section{More Related Works}

\paragraph{Relationship between risk and robustness:}  %The relationship between robustness and risk for MDPs is well-studied. \cite{ruszczynski2010risk,shen2013risk, BAUERLE2022953,bauerle2022distributionally} build the equivalence between minimizing the Markov risk measure and solving a min-max Markov control problem. However, the uncertainty set (or risk envelope) for the min-max would change with the policy implemented, especially when stochastic policies are considered (see Remark \ref{rmk:difference-with-Markov-risk-measure}), which is different from the standard formulation of RMDPs, where the uncertainty sets is not related to the policy. \citep{osogami2012robustness} obtains stronger results of equivalence to RMDP yet only for certain type of risk measures such as Conditional Value at Risk (CVaR). In comparison, our proposed risk-senstive MDP can establish exact equivalence with (soft) RMDPs for various types of coherent and convex risk measures.

Apart from the works mentioned in the Introduction, there are also other works that discuss the relationship between risk and robustness under different settings or utility functions \cite{noorani2021riskrobust,noorani2022regularization,eysenbach2021maximum, derman2020distributional}. In particular, \citet{eysenbach2021maximum} do not establish exact equivalence, but rather an equivalence of a lower-bound of certain robust objective; \citet{noorani2021riskrobust,noorani2022regularization} considers exponential utility which is a different risk measure with the risk-sensitive MDP considered in the paper as well as the Markov risk measure. 

%We would also like to point out that our equivalence result is particularly similar to \citep{noorani2022regularization}. However, it turns out that the derivation in \citep{noorani2022regularization} has some technical flaws, which inspires us to correct the result and is one of the motivations for the development of the results in our paper.\Lina{is the flaws obvious or known to the community already? we need to be careful when we write sentences like this.} \cathy{Shall I put it into the appendix instead? And add a paragraph about why it is not correct?}

\paragraph{Policy Gradient for Risk-Sensitive/Robust Learning:} There are many previous works focusing on applying policy gradient-based algorithm for risk-sensitive/robust learning \cite{tamar2012policy, tamar2015policy, tamar2017sequential, huang2021convergence, yu2023global, jaimungal2022robust, coache2023reinforcement, wang2022convergence,wang2022policy,kumar2023policy,li2023policy}. However, most of the works lack theoretical guarantees on the global convergence of these gradient-based algorithms \cite{tamar2012policy, tamar2015policy, tamar2017sequential,jaimungal2022robust, coache2023reinforcement}. There are some recent studies focusing on the global convergence guarantees for the policy gradient theorem. In particular, the works by \citet{huang2021convergence} and \citet{yu2023global} are most related to our work. \citet{huang2021convergence} give negative results showing that Markov risk measures are not gradient-dominant. \citet{yu2023global} prove global convergence for one specific type of risk measure - expected conditional risk measures (ECRMs). There are also works that considers policy gradient for RMDP with different types of uncertainty sets, for example, \citet{wang2022policy} considers RMDPs with a particular R-contamination ambiguity set
\citet{kumar2023policy,li2023policy,wang2022convergence} considers policy gradient for RMDP with different uncertainty sets (e.g. s-rectangularity and non-rectangularity uncertainty set), which is out of the scope of discussion of this paper.

\paragraph{Offline learning for robust RL:} More and more attention is drawn to the offline learning of robust MDP. In the tabular setting, \citet{zhou2021finite}
examined the uncertainty set defined by the KL divergence for offline data with uniformly lower bounded
data visitation distribution. \citet{shi2022distributionally} and \citet{li2022settling} provide near-optimal sample complexity bound for offline RL with weaker data coverage assumptions. There are also works focusing on offline learning with large state space, for example, \citet{ma2022distributionally} considers offline RL with linear approximation to deal with large state space. Our offline algorithm is most related to the work by \citet{panaganti2022robust}, where they propose the fitted-Q iteration where the Q function as well as certain dual variables are approximated by possibly nonlinear functions such as neural networks.

%TODO: a sentence for panaganti et al.
%Our offline algorithm is most related to the work by \citet{panaganti2022robust} where they design a fitted-Q iteration for RMDP. Our result is different from \citep{panaganti2022robust} in the following aspect: first, we consider a different uncertainty set defined via KL-divergence while they consider the uncertainty set with total variation distance; second, compare with their algorithm, which requires joint learning of two different functions, namely the robust Q function and the dual function, our algorithm is easier to implement as it only requires learning the Z function; thirdly, the assumptions in \citep{panaganti2022robust} requires not only approximate completeness for the robust Q function, but also approximate dual realizability on the dual function, whereas our assumption is weaker as it does not require the dual realizability.

\paragraph{Sample complexity for robust RL:}
Apart from the offline setting, there are a number of works focusing on finite-sample performance
guarantees of robust RL algorithms under different data-generating mechanisms. For example, \citet{panaganti2022sample, yang2022toward,shi2023curious} developed sample complexities for a model-based robust RL algorithm with a variety
of uncertainty sets where the data are collected using a generative model. In the online learning setting, \citet{wang2021online} proposed a robust Q-learning
algorithm with an R-contamination uncertain set which achieves a similar bound as
its non-robust counterpart. \citet{badrinath2021robust} proposed a model-free algorithm with linear function approximation to cope with large state spaces.

% More and more attention are also draw to the offline learning of robust MDP. \cite{zhou2021finite}
% examined the uncertainty set defined by the KL divergence for offline data with uniformly lower bounded
% data visitation distribution. \cite{shi2022distributionally} and \cite{li2022settling} provide near-optimal sample complexity bound for offline RL with weaker data coverage assumptions.However, these works focus on the limited tabular settings. \cite{ma2022distributionally} considers offline RL with linear approximation to deal with large state space, yet doesn't consider other type of possibly nonlinear approxiamtion. Our offline algorithm is most related to the work by \cite{panaganti2022robust} 

% \paragraph{Relationship between risk and robustness:} \cite{osogami2012robustness,BAUERLE2022953,bauerle2022distributionally} and Erfaun's paper \cite{noorani2021riskrobust,noorani2022regularization} and Sergey Levine's paper \cite{eysenbach2021maximum} and a twice soft paper \cite{derman2021twice} \cite{derman2020distributional} (Most related \cite{osogami2012robustness,derman2020distributional})

%\paragraph{Exponential criteria} Another specific risk measure that is closely related to our work is the exponential criteria, which is considered in many recent risk-sensitive RL works, e.g. \cite{fei2020risk, fei2021exponential,fei2021risk, noorani2021risk,noorani2022risk}. The KL-soft setting

\paragraph{Other related works} In addition to advances in the learning for RMDPs, there are also works focusing on the planning and computational facets for these problems \cite{ho2022robust,behzadian2021fast,grand2022convex}. Furthermore, research efforts have extended beyond theoretical considerations (e.g. \cite{la2013actor,pinto2017robust,derman2018soft,mankowitz2019robust,xu2022trustworthy,zhang2020robust,ding2023seeing}) where robust RL algorithms are proposed to handle more complicated practical problems. 

A variety of different approaches have been proposed to model and address risk and robustness. Notably, in addition to Markov risk measures and robust MDPs, researchers have explored alternative methodologies, such as the use of exponential utility \cite{howard1972risk, fei2020risk, fei2021exponential, fei2021risk, noorani2021risk, noorani2022risk}, constraint MDPs \cite{wang2022robust, chow2017risk, hasanzadezonuzy2021learning, geibel2005risk}, distributional RL\cite{bellemare2017distributional, singh2020improving, ma2020dsac, stanko2019risk, sun2021exploring}, and robust control \cite{khalil1996robust, zhang2021policy} etc.  

We would also like to note that the risk-sensitive MDP with entropy risk measure (which we have proved to be equivalent to the KL-soft RMDP) is related but not identical to the exponential utility \cite{howard1972risk, fei2020risk, fei2021exponential, fei2021risk, noorani2021risk, noorani2022risk}. There are works that discuss the equivalence of optimizing the exponential utility and solving the KL-soft RMDP in the finite-horizon undiscounted-sum setting \cite{osogami2012robustness}. However, the result cannot be generalized to the infinite-horizon-discounted-sum setting which is considered in this paper. Under this setting, it is unclear whether exponential utility still obtains a similar interpretation in terms of robustness.

%Here we list some other works that consider different aspects of risk-sensitive/robust MDPs. Apart from Markov risk measures and RMDPs, there are also other approaches to model risk such as using the exponential criteria \cite{howard1972risk, fei2020risk, fei2021exponential,fei2021risk, noorani2021risk,noorani2022risk}, constraint MDP \cite{wang2022robust, chow2017risk, hasanzadezonuzy2021learning,geibel2005risk} and distributional RL \cite{bellemare2017distributional} \cite{singh2020improving,ma2020dsac,stanko2019risk,sun2021exploring}, etc. There are also works that consider the planning and computational aspect for RMDP \cite{ho2022robust,behzadian2021fast,grand2022convex} and also there are practical algorithms \cite{derman2018soft,la2013actor}

\section{Numerical Simulations}\label{sec:numerics}

In this section, we present simulation results that evaluate the exact policy gradient algorithm (see Section \ref{section:equivalence}) and the RFZI algorithm (see Section \ref{section:RFZI}) in the following environment.

\begin{wrapfigure}{r}{5cm}
    \centering\vspace*{-8pt}
        \tikzset{snode/.style = {draw=black, shape=circle, line width=1.0pt, inner sep=0pt, minimum width=15pt}}
    \begin{tikzpicture}
        \fill[fill=Yellow!40!white] (51.43:40pt) arc (51.43:-25.71:40pt) -- (-25.71:60pt) arc (-25.71:51.43:60pt) -- cycle;
        \draw[decorate, decoration={text along path, text={|\color{Goldenrod}\footnotesize|medium risk}, text align=center}] (38.57:63pt) arc (38.57:-12.86:63pt);
        
        \fill[fill=red!20!white] (-77.14:40pt) arc (-77.14:-154.29:40pt) -- (-154.29:60pt) arc (-154.29:-77.14:60pt) -- cycle;
        \draw[decorate, decoration={text along path, text={|\color{red}\footnotesize|high risk}, text align=center}] (-141.43:68pt) arc (-141.43:-90:68pt);
        
        \fill[fill=LimeGreen!40!white] (-180:40pt) arc (-180:-231.43:40pt) -- (-231.43:60pt) arc (-231.43:-180:60pt) -- cycle;
        \draw[decorate, decoration={text along path, text={|\color{Green}\footnotesize|low risk}, text align=center}] (-192.86:63pt) arc (-192.86:-218.57:63pt);

        \node[snode] (P1) at (90:50pt) {\footnotesize 0};
        \node[snode] (P2) at (64.29:50pt) {\footnotesize 0};
        \node[snode] (P3) at (38.57:50pt) {\footnotesize -1};
        \node[snode] (P4) at (12.86:50pt) {\footnotesize 2};
        \node[snode] (P5) at (-12.86:50pt) {\footnotesize -1};
        \node[snode] (P6) at (-38.57:50pt) {\footnotesize 0};
        \node[snode] (P7) at (-64.29:50pt) {\footnotesize 0};
        \node[snode] (P8) at (-90:50pt) {\footnotesize -10};
        \node[snode] (P9) at (-115.71:50pt) {\footnotesize 5};
        \node[snode] (P10) at (-141.43:50pt) {\footnotesize -10};
        \node[snode] (P11) at (-167.14:50pt) {\footnotesize 0};
        \node[snode] (P12) at (167.14:50pt) {\footnotesize $\frac{9}{10}$};
        \node[snode] (P13) at (141.43:50pt) {\footnotesize 1};
        \node[snode] (P14) at (115.71:50pt) {\footnotesize 0};

        \path (P1) edge (P2);
        \path (P2) edge (P3);
        \path (P3) edge (P4);
        \path (P4) edge (P5);
        \path (P5) edge (P6);
        \path (P6) edge (P7);
        \path (P7) edge (P8);
        \path (P8) edge (P9);
        \path (P9) edge (P10);
        \path (P10) edge (P11);
        \path (P11) edge (P12);
        \path (P12) edge (P13);
        \path (P13) edge (P14);
        \path (P14) edge (P1);
    \end{tikzpicture}
    \caption{An exemplary 14-state environment with high-, medium-, and low-risk zones.}\label{fig:small-toy-env}
    \vspace*{-16pt}
\end{wrapfigure}
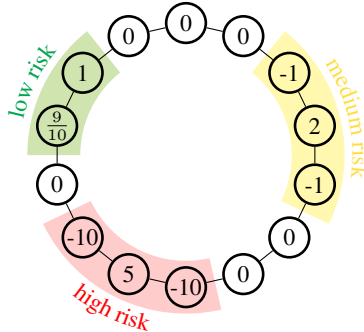

\paragraph{Environment setups: traveling on a cycle graph.} This is a environment with finite state and action space consisting of $n$ states $\cS = \bZ_n$, which can be conceptually regarded as arranged on a cycle. The initial distribution is a uniformly random distribution over $\cS$. At each state the agent is allowed to select one action from $\cA = \lbrace -1, 0, 1 \rbrace$ (corresponding to \texttt{left}, \texttt{stay} and \texttt{right}), and receives a reward $r(s)$ that only depends on $s$ ($r(\cdot)$ is called the hitting reward function). As the aliases suggest, an action $a$ at state $s$ is supposed to move the agent to $(s+a) \mod n$. The uncertainty in the environment appears in the form of stochastic transitions, and is characterized by the probability $\alpha \in [0, \frac{1}{2}]$ of missing the expected destination by one step; i.e., transition probability is
\begin{equation*}
    P(s' \mid s,a) = \begin{cases}
      \alpha & s'=(s+a \pm 1) \mod n \\
      1-2\alpha & s' = (s+a) \mod n \\
      0 & \text{otherwise}
    \end{cases}.
\end{equation*}
In this section, we focus on solving the KL-soft robust MDP problem \eqref{eq:KL-soft-MDP} in this environment with varying penalty magnitude $\beta$.

\paragraph{Metrics.} The performance of the algorithms may be evaluated by the following metrics:
\begin{itemize}
  \item \textit{Optimality gap} $\bE_{s_0 \sim \rho}[V^{\star}(s_0) - V^{\pi}(s_0)]$, where $\pi$ is the policy generated by the algorithm.
  \item \textit{Average test reward} over a few test episodes (20 by default).
  \item \textit{Robustness value} $\widehat{V}_{\pi}(\delta) := \inf_{P \in \cP_{\delta}} \bE^{\pi,P} \left[ \sum_{t=0}^{H} \gamma^t r(s_t, a_t) \;\middle|\; s_0 \sim \rho \right]$, where the model uncertainty set is selected as $\cP_{\delta} := \lbrace \widetilde{P} \mid \KL( \widetilde{P}_{s,a} || P_{s,a} ) \leq \delta,~ \forall s,a \rbrace$.
\end{itemize}
Note that the optimality gap and average test reward are usually plotted along the training trajectory, while the robustness value is usually plotted against the perturbation magnitude $\delta$.

\paragraph{Exact policy gradient.} We first examine the performance of the exact policy gradient algorithm under direct parameterization. Here we consider a 14-state environment as illustrated in Figure \ref{fig:small-toy-env}, where the rewards are marked in the nodes representing states. For this specific example, we can roughly classify the states into  high-, medium-, and low-risk zones, as suggested in the figure.

For clarity of exposition, for now we focus on three specific settings, i.e. $(\alpha, \beta) = (0.01, 0.1)$, $(0.01, 1.0), (0.15, 1.0)$ (more results can be found in Section \ref{sec:more-numeric-details}). The optimality gap curves under these settings are shown in Figure \ref{fig:simulation-PG-toy} below. It can be observed that the optimality gap decays to exactly 0 in all these settings, which justifies Theorem \ref{theorem:gradient-ascent} that guarantees convergence of the exact policy gradient algorithm under direct parameterization. We also point out that the loss curve is a little crooked because we use projected gradient ascent, so that the optimization dynamics is not smooth when the policy at each state is pushed to the boundary.

  \begin{figure}[ht]
    \centering\captionsetup{justification=centering}
    \begin{subfigure}[b]{0.32\linewidth}
      \centering
      \includegraphics[width=\linewidth]{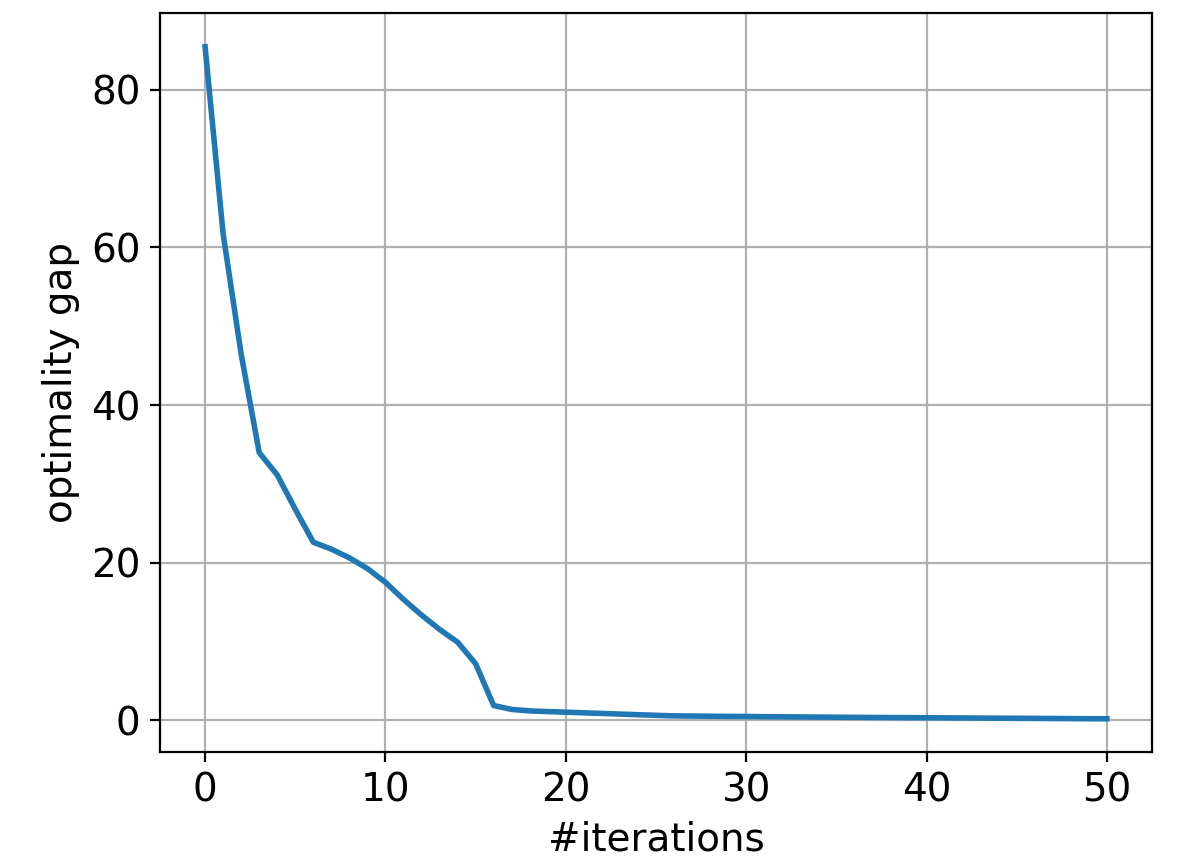}
      \subcaption{$\alpha = 0.01$, $\beta = 0.1$.}
    \end{subfigure}
    \begin{subfigure}[b]{0.32\linewidth}
      \centering
      \includegraphics[width=\linewidth]{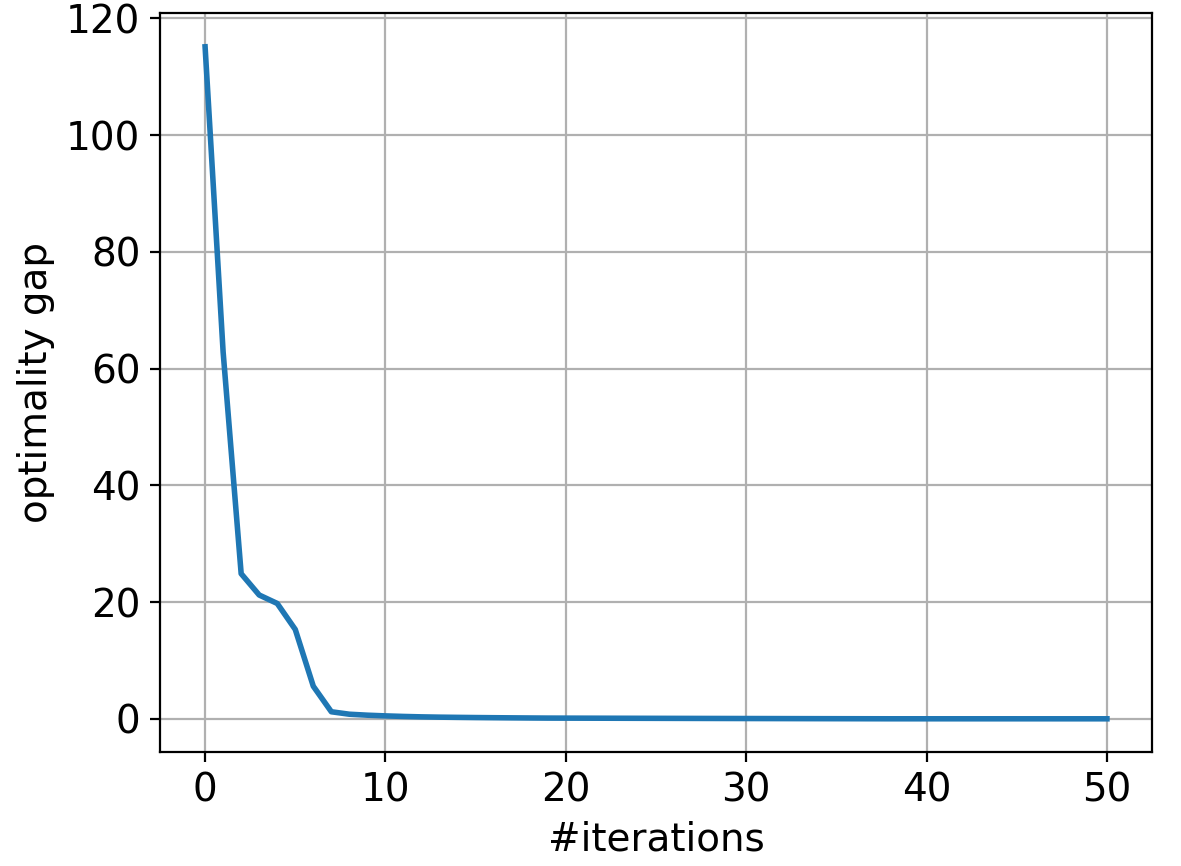}
      \subcaption{$\alpha = 0.01$, $\beta = 1.0$.}
    \end{subfigure}
    \begin{subfigure}[b]{0.32\linewidth}
      \centering
      \includegraphics[width=\linewidth]{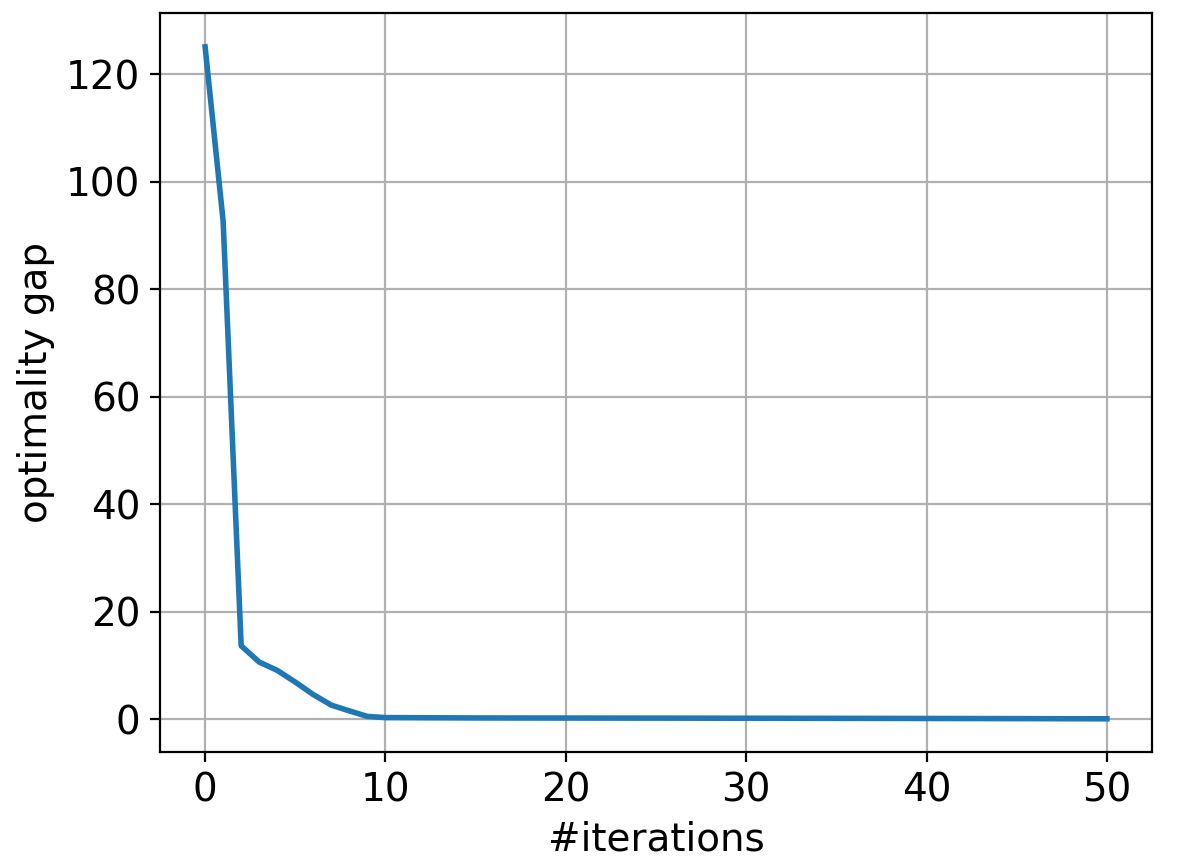}
      \subcaption{$\alpha = 0.15$, $\beta = 1.0$.}
    \end{subfigure}
    \caption{Optimality gap curves for the exact policy gradient algorithm in different settings.}\label{fig:simulation-PG-toy}
  \end{figure}

Further, we take a closer look at the learned policies in these settings. The policies are illustrated in Figure \ref{fig:PG-illustrated-policy}. To understand the role of $\beta$, we compare Figure \ref{fig:PG-illustrated-policy:a} with \ref{fig:PG-illustrated-policy:b} --- when the uncertainty in the model (represented by $\alpha$) is fixed and mild, the optimal policy for $\beta = 0.1$ is to greedily pursue the largest possible reward (i.e., staying at the state in high-risk zone with hitting reward 5); however, an agent with higher risk-sensitivity level $\beta = 1$ cares more about the potential losses caused by model uncertainty, and thus its optimal policy shifts to seeking safer options (i.e., moving to the medium-risk zone and staying there).

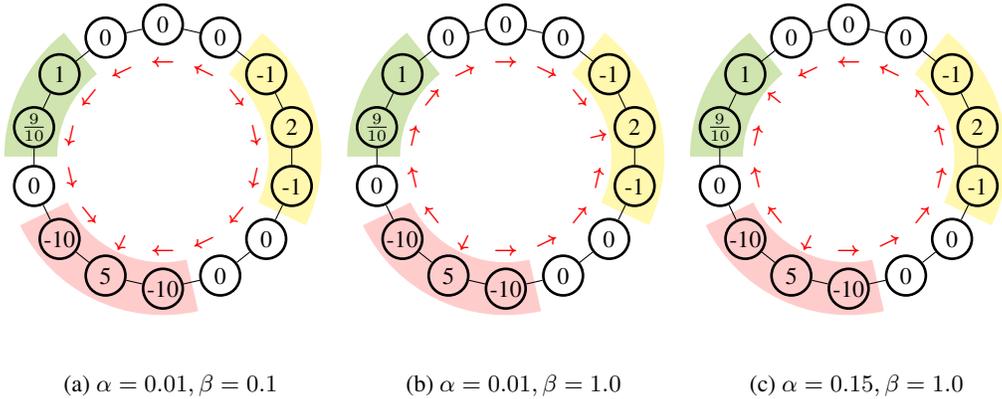
\begin{figure}[ht]
  \centering
  \tikzset{snode/.style = {draw=black, shape=circle, line width=1.0pt, inner sep=0pt, minimum width=15pt}}
  \begin{subfigure}[b]{0.32\linewidth}
      \centering
          \begin{tikzpicture}
        \fill[fill=Yellow!40!white] (51.43:40pt) arc (51.43:-25.71:40pt) -- (-25.71:60pt) arc (-25.71:51.43:60pt) -- cycle;
        \fill[fill=red!20!white] (-77.14:40pt) arc (-77.14:-154.29:40pt) -- (-154.29:60pt) arc (-154.29:-77.14:60pt) -- cycle;
        \fill[fill=LimeGreen!40!white] (-180:40pt) arc (-180:-231.43:40pt) -- (-231.43:60pt) arc (-231.43:-180:60pt) -- cycle;

        \node[snode] (P1) at (90:50pt) {\footnotesize 0};
        \node[snode] (P2) at (64.29:50pt) {\footnotesize 0};
        \node[snode] (P3) at (38.57:50pt) {\footnotesize -1};
        \node[snode] (P4) at (12.86:50pt) {\footnotesize 2};
        \node[snode] (P5) at (-12.86:50pt) {\footnotesize -1};
        \node[snode] (P6) at (-38.57:50pt) {\footnotesize 0};
        \node[snode] (P7) at (-64.29:50pt) {\footnotesize 0};
        \node[snode] (P8) at (-90:50pt) {\footnotesize -10};
        \node[snode] (P9) at (-115.71:50pt) {\footnotesize 5};
        \node[snode] (P10) at (-141.43:50pt) {\footnotesize -10};
        \node[snode] (P11) at (-167.14:50pt) {\footnotesize 0};
        \node[snode] (P12) at (167.14:50pt) {\footnotesize $\frac{9}{10}$};
        \node[snode] (P13) at (141.43:50pt) {\footnotesize 1};
        \node[snode] (P14) at (115.71:50pt) {\footnotesize 0};

        \path (P1) edge (P2);
        \path (P2) edge (P3);
        \path (P3) edge (P4);
        \path (P4) edge (P5);
        \path (P5) edge (P6);
        \path (P6) edge (P7);
        \path (P7) edge (P8);
        \path (P8) edge (P9);
        \path (P9) edge (P10);
        \path (P10) edge (P11);
        \path (P11) edge (P12);
        \path (P12) edge (P13);
        \path (P13) edge (P14);
        \path (P14) edge (P1);

        \node at (90:36pt) {\rotatebox{180}{\scriptsize\color{red} $\bm{\rightarrow}$}};
        \node at (64.29:36pt) {\rotatebox{154.29}{\scriptsize\color{red} $\bm{\rightarrow}$}};
        \node at (38.57:36pt) {\rotatebox{128.57}{\scriptsize\color{red} $\bm{\leftarrow}$}};
        \node at (12.86:36pt) {\rotatebox{102.86}{\scriptsize\color{red} $\bm{\leftarrow}$}};
        \node at (-12.86:36pt) {\rotatebox{77.14}{\scriptsize\color{red} $\bm{\leftarrow}$}};
        \node at (-38.57:36pt) {\rotatebox{51.43}{\scriptsize\color{red} $\bm{\leftarrow}$}};
        \node at (-64.29:36pt) {\rotatebox{25.71}{\scriptsize\color{red} $\bm{\leftarrow}$}};
        \node at (-90:36pt) {\rotatebox{0}{\scriptsize\color{red} $\bm{\leftarrow}$}};
        \node at (-115.71:36pt) {\rotatebox{-25.71}{\scriptsize\color{red} $\bm{\downarrow}$}};
        \node at (-141.43:36pt) {\rotatebox{-51.43}{\scriptsize\color{red} $\bm{\rightarrow}$}};
        \node at (-167.14:36pt) {\rotatebox{-77.14}{\scriptsize\color{red} $\bm{\rightarrow}$}};
        \node at (167.14:36pt) {\rotatebox{-102.86}{\scriptsize\color{red} $\bm{\rightarrow}$}};
        \node at (141.43:36pt) {\rotatebox{-128.57}{\scriptsize\color{red} $\bm{\rightarrow}$}};
        \node at (115.71:36pt) {\rotatebox{-154.29}{\scriptsize\color{red} $\bm{\rightarrow}$}};
      \end{tikzpicture}
      \caption{$\alpha = 0.01, \beta = 0.1$}\label{fig:PG-illustrated-policy:a}
  \end{subfigure}
  \begin{subfigure}[b]{0.32\linewidth}
      \centering
          \begin{tikzpicture}
        \fill[fill=Yellow!40!white] (51.43:40pt) arc (51.43:-25.71:40pt) -- (-25.71:60pt) arc (-25.71:51.43:60pt) -- cycle;
        \fill[fill=red!20!white] (-77.14:40pt) arc (-77.14:-154.29:40pt) -- (-154.29:60pt) arc (-154.29:-77.14:60pt) -- cycle;
        \fill[fill=LimeGreen!40!white] (-180:40pt) arc (-180:-231.43:40pt) -- (-231.43:60pt) arc (-231.43:-180:60pt) -- cycle;

        \node[snode] (P1) at (90:50pt) {\footnotesize 0};
        \node[snode] (P2) at (64.29:50pt) {\footnotesize 0};
        \node[snode] (P3) at (38.57:50pt) {\footnotesize -1};
        \node[snode] (P4) at (12.86:50pt) {\footnotesize 2};
        \node[snode] (P5) at (-12.86:50pt) {\footnotesize -1};
        \node[snode] (P6) at (-38.57:50pt) {\footnotesize 0};
        \node[snode] (P7) at (-64.29:50pt) {\footnotesize 0};
        \node[snode] (P8) at (-90:50pt) {\footnotesize -10};
        \node[snode] (P9) at (-115.71:50pt) {\footnotesize 5};
        \node[snode] (P10) at (-141.43:50pt) {\footnotesize -10};
        \node[snode] (P11) at (-167.14:50pt) {\footnotesize 0};
        \node[snode] (P12) at (167.14:50pt) {\footnotesize $\frac{9}{10}$};
        \node[snode] (P13) at (141.43:50pt) {\footnotesize 1};
        \node[snode] (P14) at (115.71:50pt) {\footnotesize 0};

        \path (P1) edge (P2);
        \path (P2) edge (P3);
        \path (P3) edge (P4);
        \path (P4) edge (P5);
        \path (P5) edge (P6);
        \path (P6) edge (P7);
        \path (P7) edge (P8);
        \path (P8) edge (P9);
        \path (P9) edge (P10);
        \path (P10) edge (P11);
        \path (P11) edge (P12);
        \path (P12) edge (P13);
        \path (P13) edge (P14);
        \path (P14) edge (P1);

        \node at (90:36pt) {\rotatebox{180}{\scriptsize\color{red} $\bm{\leftarrow}$}};
        \node at (64.29:36pt) {\rotatebox{154.29}{\scriptsize\color{red} $\bm{\leftarrow}$}};
        \node at (38.57:36pt) {\rotatebox{128.57}{\scriptsize\color{red} $\bm{\leftarrow}$}};
        \node at (12.86:36pt) {\rotatebox{102.86}{\scriptsize\color{red} $\bm{\downarrow}$}};
        \node at (-12.86:36pt) {\rotatebox{77.14}{\scriptsize\color{red} $\bm{\rightarrow}$}};
        \node at (-38.57:36pt) {\rotatebox{51.43}{\scriptsize\color{red} $\bm{\rightarrow}$}};
        \node at (-64.29:36pt) {\rotatebox{25.71}{\scriptsize\color{red} $\bm{\rightarrow}$}};
        \node at (-90:36pt) {\rotatebox{0}{\scriptsize\color{red} $\bm{\rightarrow}$}};
        \node at (-115.71:36pt) {\rotatebox{-25.71}{\scriptsize\color{red} $\bm{\downarrow}$}};
        \node at (-141.43:36pt) {\rotatebox{-51.43}{\scriptsize\color{red} $\bm{\leftarrow}$}};
        \node at (-167.14:36pt) {\rotatebox{-77.14}{\scriptsize\color{red} $\bm{\leftarrow}$}};
        \node at (167.14:36pt) {\rotatebox{-102.86}{\scriptsize\color{red} $\bm{\leftarrow}$}};
        \node at (141.43:36pt) {\rotatebox{-128.57}{\scriptsize\color{red} $\bm{\leftarrow}$}};
        \node at (115.71:36pt) {\rotatebox{-154.29}{\scriptsize\color{red} $\bm{\leftarrow}$}};
      \end{tikzpicture}
      \caption{$\alpha = 0.01, \beta = 1.0$}\label{fig:PG-illustrated-policy:b}
  \end{subfigure}
  \begin{subfigure}[b]{0.32\linewidth}
      \centering
          \begin{tikzpicture}
        \fill[fill=Yellow!40!white] (51.43:40pt) arc (51.43:-25.71:40pt) -- (-25.71:60pt) arc (-25.71:51.43:60pt) -- cycle;
        \fill[fill=red!20!white] (-77.14:40pt) arc (-77.14:-154.29:40pt) -- (-154.29:60pt) arc (-154.29:-77.14:60pt) -- cycle;
        \fill[fill=LimeGreen!40!white] (-180:40pt) arc (-180:-231.43:40pt) -- (-231.43:60pt) arc (-231.43:-180:60pt) -- cycle;

        \node[snode] (P1) at (90:50pt) {\footnotesize 0};
        \node[snode] (P2) at (64.29:50pt) {\footnotesize 0};
        \node[snode] (P3) at (38.57:50pt) {\footnotesize -1};
        \node[snode] (P4) at (12.86:50pt) {\footnotesize 2};
        \node[snode] (P5) at (-12.86:50pt) {\footnotesize -1};
        \node[snode] (P6) at (-38.57:50pt) {\footnotesize 0};
        \node[snode] (P7) at (-64.29:50pt) {\footnotesize 0};
        \node[snode] (P8) at (-90:50pt) {\footnotesize -10};
        \node[snode] (P9) at (-115.71:50pt) {\footnotesize 5};
        \node[snode] (P10) at (-141.43:50pt) {\footnotesize -10};
        \node[snode] (P11) at (-167.14:50pt) {\footnotesize 0};
        \node[snode] (P12) at (167.14:50pt) {\footnotesize $\frac{9}{10}$};
        \node[snode] (P13) at (141.43:50pt) {\footnotesize 1};
        \node[snode] (P14) at (115.71:50pt) {\footnotesize 0};

        \path (P1) edge (P2);
        \path (P2) edge (P3);
        \path (P3) edge (P4);
        \path (P4) edge (P5);
        \path (P5) edge (P6);
        \path (P6) edge (P7);
        \path (P7) edge (P8);
        \path (P8) edge (P9);
        \path (P9) edge (P10);
        \path (P10) edge (P11);
        \path (P11) edge (P12);
        \path (P12) edge (P13);
        \path (P13) edge (P14);
        \path (P14) edge (P1);

        \node at (90:36pt) {\rotatebox{180}{\scriptsize\color{red} $\bm{\rightarrow}$}};
        \node at (64.29:36pt) {\rotatebox{154.29}{\scriptsize\color{red} $\bm{\rightarrow}$}};
        \node at (38.57:36pt) {\rotatebox{128.57}{\scriptsize\color{red} $\bm{\rightarrow}$}};
        \node at (12.86:36pt) {\rotatebox{102.86}{\scriptsize\color{red} $\bm{\rightarrow}$}};
        \node at (-12.86:36pt) {\rotatebox{77.14}{\scriptsize\color{red} $\bm{\rightarrow}$}};
        \node at (-38.57:36pt) {\rotatebox{51.43}{\scriptsize\color{red} $\bm{\rightarrow}$}};
        \node at (-64.29:36pt) {\rotatebox{25.71}{\scriptsize\color{red} $\bm{\rightarrow}$}};
        \node at (-90:36pt) {\rotatebox{0}{\scriptsize\color{red} $\bm{\rightarrow}$}};
        \node at (-115.71:36pt) {\rotatebox{-25.71}{\scriptsize\color{red} $\bm{\downarrow}$}};
        \node at (-141.43:36pt) {\rotatebox{-51.43}{\scriptsize\color{red} $\bm{\leftarrow}$}};
        \node at (-167.14:36pt) {\rotatebox{-77.14}{\scriptsize\color{red} $\bm{\leftarrow}$}};
        \node at (167.14:36pt) {\rotatebox{-102.86}{\scriptsize\color{red} $\bm{\leftarrow}$}};
        \node at (141.43:36pt) {\rotatebox{-128.57}{\scriptsize\color{red} $\bm{\downarrow}$}};
        \node at (115.71:36pt) {\rotatebox{-154.29}{\scriptsize\color{red} $\bm{\rightarrow}$}};
      \end{tikzpicture}
      \caption{$\alpha = 0.15, \beta = 1.0$}\label{fig:PG-illustrated-policy:c}
  \end{subfigure}
  \caption{Illustrated policies learned by the exact policy gradient algorithm in different settings.}\label{fig:PG-illustrated-policy}
\end{figure}

Similarly, to understand the role of $\alpha$, we compare Figure \ref{fig:PG-illustrated-policy:b} with \ref{fig:PG-illustrated-policy:c} --- for an agent with fixed moderate penalty magnitude $\beta$, when the noise in the model is small ($\alpha = 0.01$), it is still optimal for the agent to stay in the medium-risk zone for a balance of risks and rewards; nevertheless, if the agent is put in a noisier environment ($\alpha = 0.15$), then its optimal policy would be directly moving to the low-risk zone to avoid potential risks.

To further examine the robustness of the risk-sensitive policies generated by the exact policy gradient algorithm, we calculate their robustness values with respect to different perturbation magnitudes $\delta$, and plot them for comparison in Figure \ref{fig:simulation-PG-full-robust}. Here the \textit{risk-neutral} policy refers to the optimal policy of the standard risk-neutral MDP, {while the \textit{robust baseline} policy refers to the optimal policy of the RMDP with KL-rectangular ambiguity set $\mathcal{P}_{\delta}$ (as defined above)}. It can be observed that, when $\delta$ is large, the risk-sensitive policies outperform the risk-neutral policy in both settings. {Meanwhile, our algorithm generally exhibits a comparable level of robustness as compared to the robust baseline that directly optimizes over RMDPs; sometimes it is even more robust than the baseline, especially when the actual ambiguity set is significantly larger than the one assumed in training (see Figure \ref{fig:simulation-PG-full-robust:a} for the curve where $\delta \gg 0.3$)}. Moreover, policies generated from higher penalty magnitude $\beta$ tend to have lower robustness values when $\delta$ is small, but gradually become more robust as $\delta$ increases.
  
  \begin{figure}[ht]
    \centering\captionsetup{justification=centering}
    \begin{subfigure}[b]{0.4\linewidth}
      \centering
      \includegraphics[width=\linewidth]{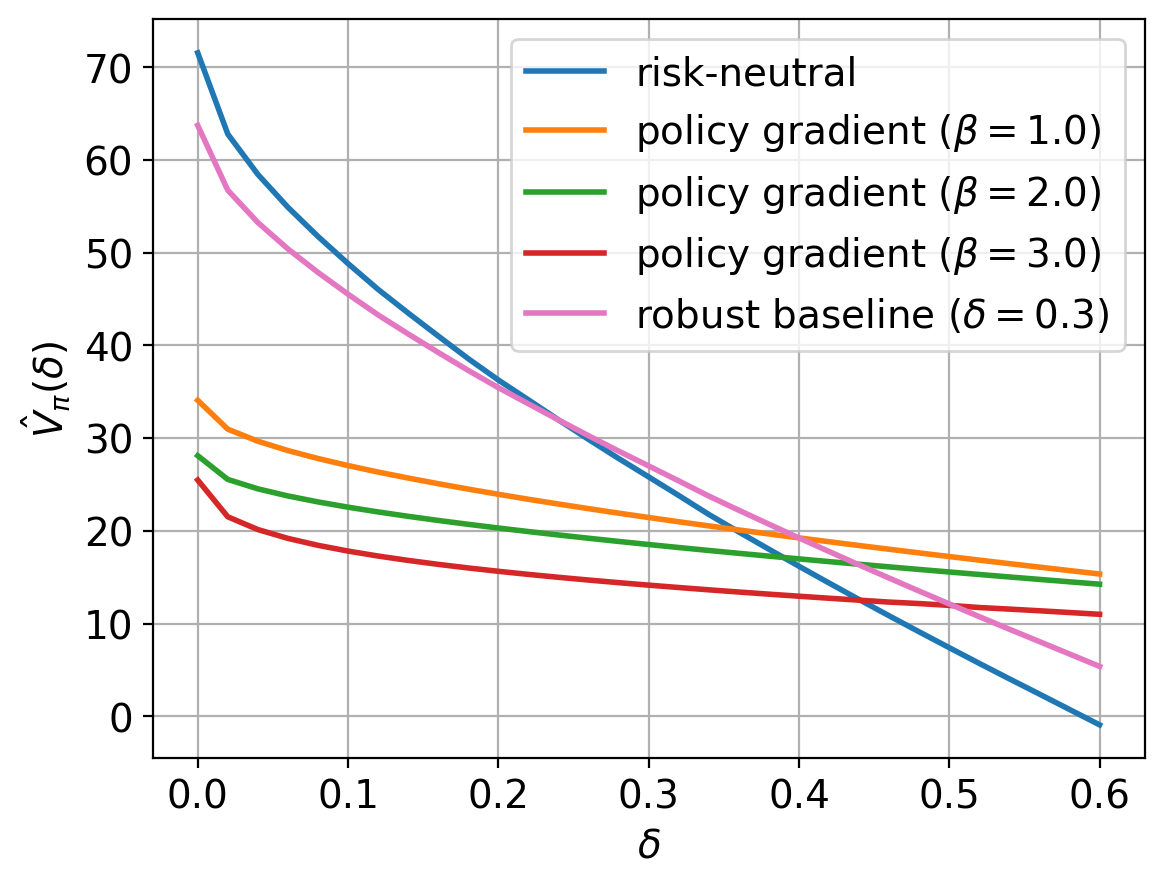}
      \subcaption{$\alpha = 0.01$.}\label{fig:simulation-PG-full-robust:a}
    \end{subfigure}\hspace{10pt}
    \begin{subfigure}[b]{0.4\linewidth}
      \centering
      \includegraphics[width=\linewidth]{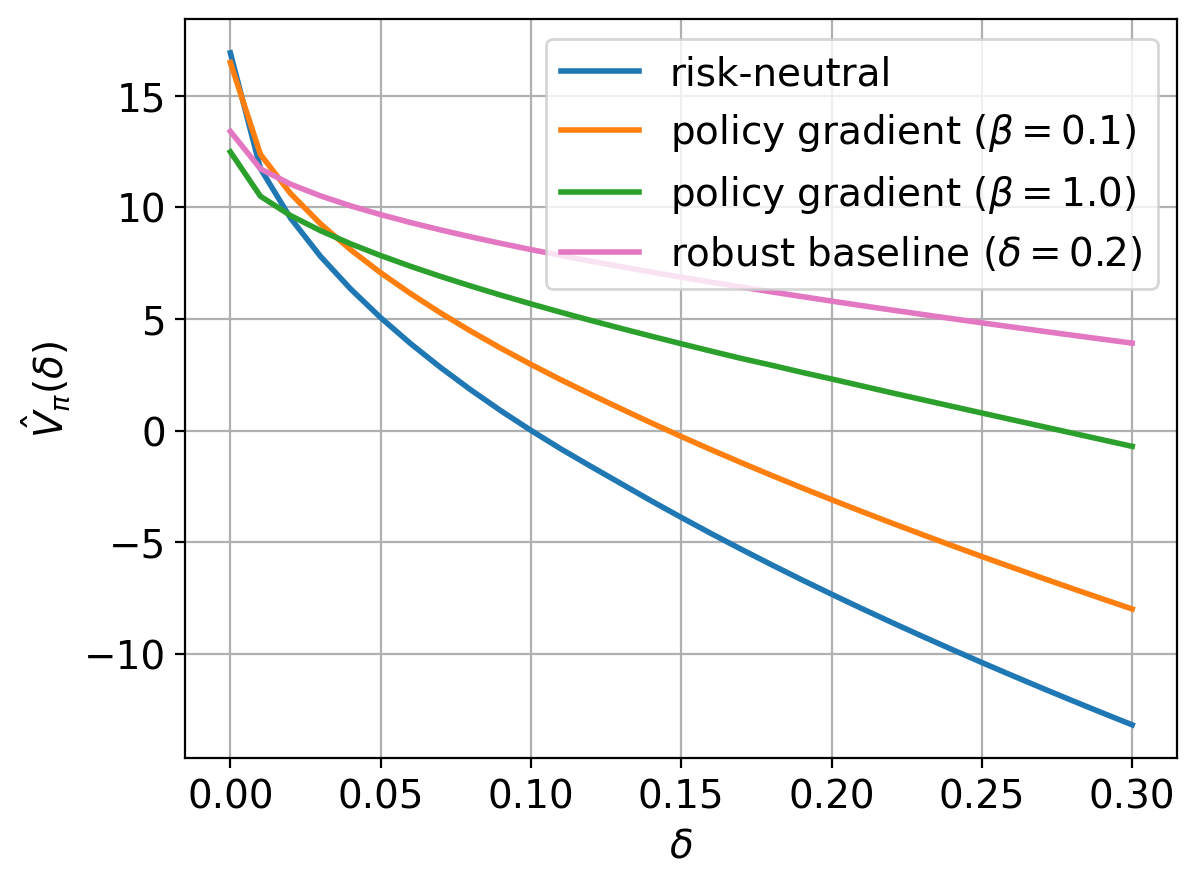}
      \subcaption{$\alpha = 0.15$.}
    \end{subfigure}
    \caption{Robustness values of the generated policies with respect to different $\delta$.\\${}^{\dagger}${\small \textit{Risk-neutral} policy refers to the optimal policy of the risk-neutral MDP.\\{\textit{Robust baseline} refers to the optimal policy of the RMDP with KL-rectangular ambiguity set $\mathcal{P}_{\delta}$.}}}\label{fig:simulation-PG-full-robust}
  \end{figure}

The above discussion reveals that risk-sensitive agents in face of the transition uncertainty do learn to avoid those states that could bring small instant rewards at the risk of potential future losses, and their risk-averse tendency increases when $\beta$ is set larger.
{These numerical evidence, in turn, further motivates and justifies our focus on KL-soft RMDPs. On the one hand, learning in the context of KL-soft RMDPs are generally more computationally tractable by relatively straightforward algorithm designs, without introducing complicated optimization techniques or additional assumptions. On the other hand, the optimal policy learned in the KL-soft RMDP context does exhibit robustness in the presence of model uncertainty. These two observations, in combination, show that the research into KL-soft RMDPs may offer an alternative, analytically tractable, and potentially more accessible approach to robust reinforcement learning while maintaining comparable robust behavior.
}

\paragraph{RFZI.} Now we proceed to examine the performance of the RFZI algorithm. The algorithm is tested in an 100-state environment. The hitting reward design of this environment is conceptually similar to the exemplary 14-state environment (see our code for more details). For practical implementation, the function family is selected as a 3-layer neural network, which may take any proper state-action representation as input (details deferred to Section \ref{sec:more-numeric-details}). Note that it is crucial to select good representations for efficient reinforcement learning, as the state itself might not provide sufficient information for learning (see e.g. \cite{du2019good} for more discussions). Further, since the minimizer in each iteration cannot be exactly calculated, instead we simply perform a batch of stochastic gradient descent updates for approximation, where for each update we sample a subset from the offline dataset. The practical algorithm is shown in Algorithm \ref{alg:RFQI-practical} below.

\begin{algorithm}[ht]
  \caption{The practical RFZI algorithm}\label{alg:RFQI-practical}
  \begin{algorithmic}[1]
    \State \textbf{Input:} Offline dataset $\cD = (s_i, a_i, r_i, s_i')_{i=1}^N$, function family $\cF = \lbrace Z(\cdot; \theta) \mid \theta \rbrace$, learning rate $\eta$, update rate $\tau$, number of batches $T_{\mathrm{batch}}$, batch size $N_{\mathrm{batch}}$.
    \State \textbf{Initialize:} $\theta_{\mathrm{current}}$, $\theta_{\mathrm{target}}$.
    \For{$k = 0, \dots, K-1$}
      \For{$t = 1, 2, \ldots, T_{\mathrm{batch}}$}
        \State Sample a batch of transitions $\lbrace (s_i, a_i , r_i, s'_i) \mid i \in [N_{\mathrm{batch}}] \rbrace$ from the dataset $\cD$.
        \State Perform gradient descent $\theta_{\mathrm{current}} \gets \theta_{\mathrm{current}} - \eta \nabla \widehat{\mathcal{L}}(\theta_{\mathrm{current}})$, where
        \begin{equation*}
          \widehat{\mathcal{L}}(\theta_{\mathrm{current}}) := \hcL(Z(\cdot; \theta_{\mathrm{current}}), Z(\cdot; \theta_{\mathrm{target}})).
        \end{equation*}
      \EndFor
      \State Update: $\theta_{\mathrm{target}} \gets (1-\tau) \theta_{\mathrm{target}} + \tau \theta_{\mathrm{current}}$, $\theta_{\mathrm{current}} \gets \theta_{\mathrm{target}}$.
    \EndFor 
    \State \textbf{Output:} $\pi_K = \argmax_a [r(s,a) - \gamma\beta^{-1}\log Z(s,a; \theta_{\mathrm{target}})]$.
  \end{algorithmic}
\end{algorithm}

  \begin{figure}[ht]
    \centering\captionsetup{justification=centering}
    \begin{subfigure}{\linewidth}
      \centering
      \includegraphics[width=0.32\linewidth]{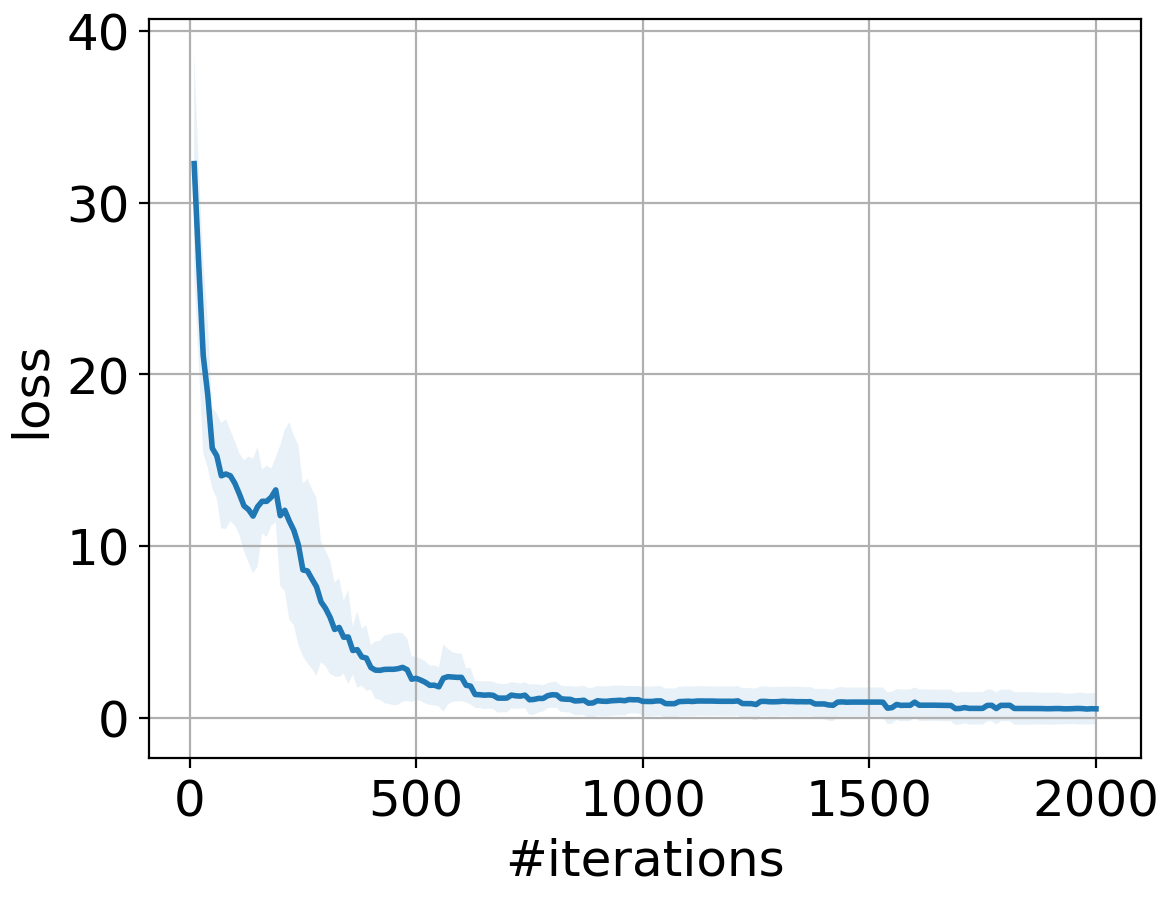}
      \includegraphics[width=0.32\linewidth]{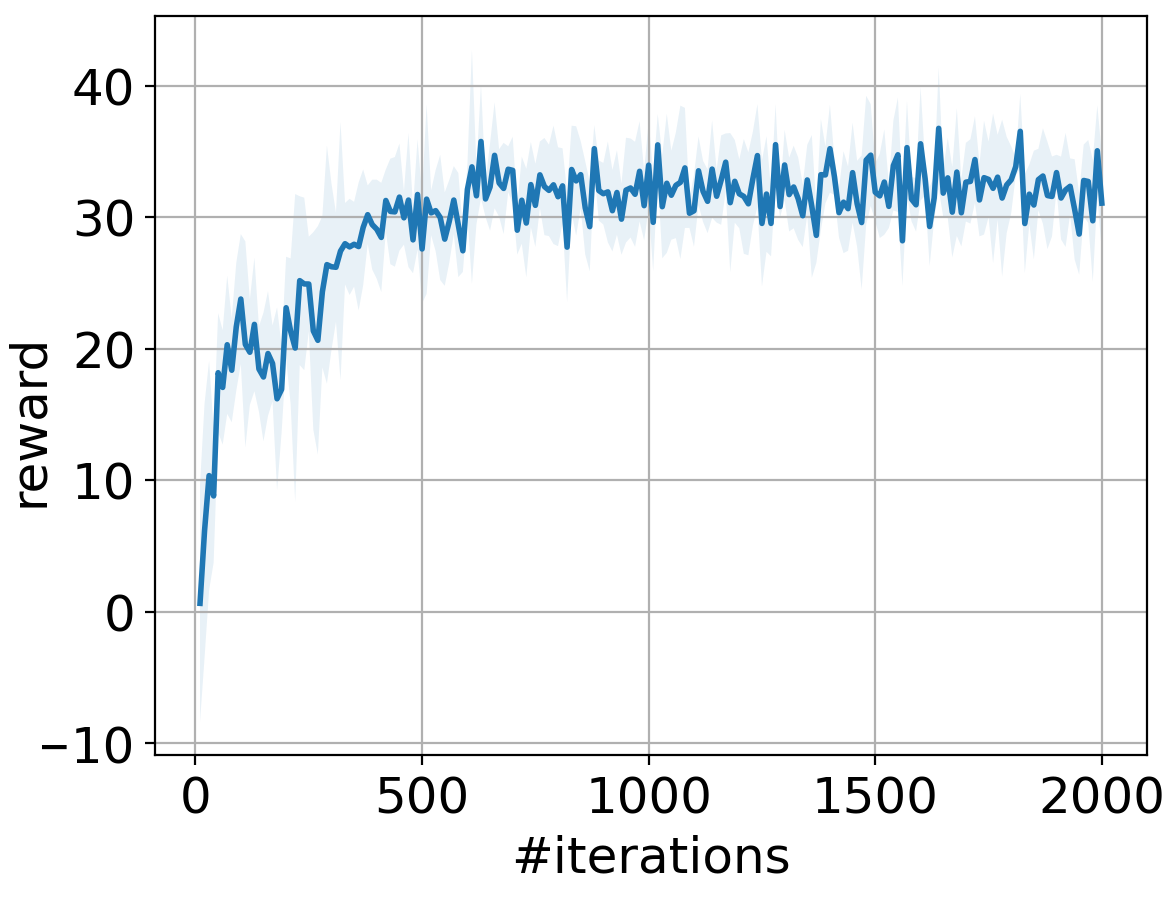}
      \includegraphics[width=0.32\linewidth]{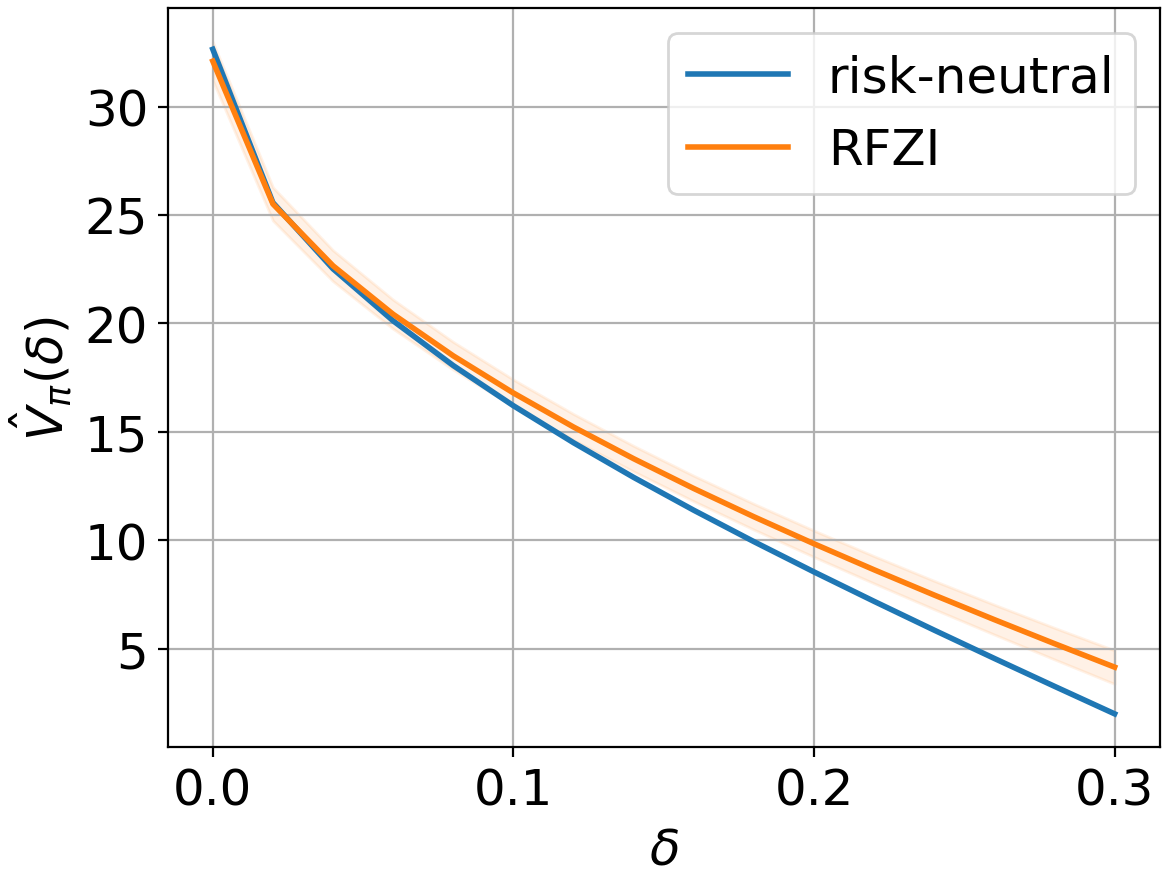}
      \subcaption{$\beta = 0.01$.}
    \end{subfigure}
    
    \begin{subfigure}[b]{\linewidth}
      \centering
      \includegraphics[width=0.32\linewidth]{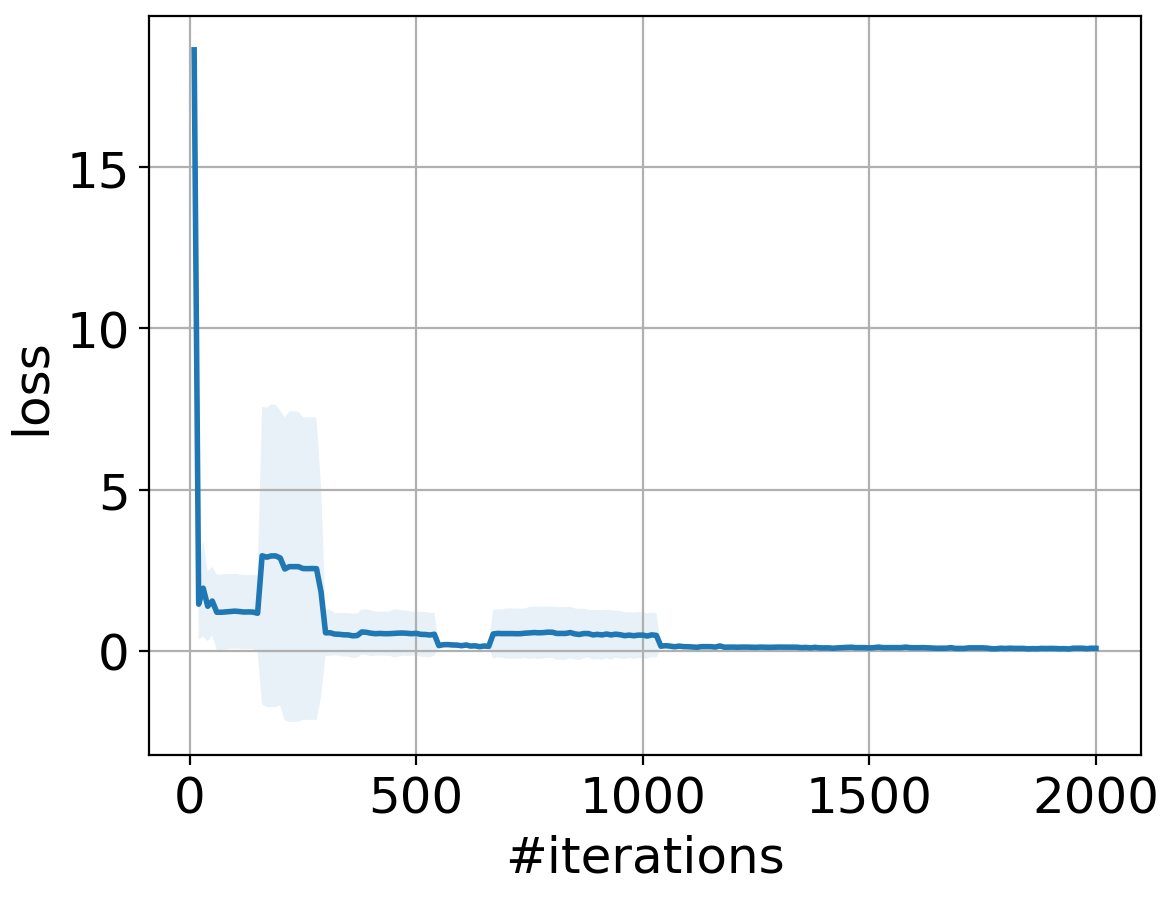}
      \includegraphics[width=0.32\linewidth]{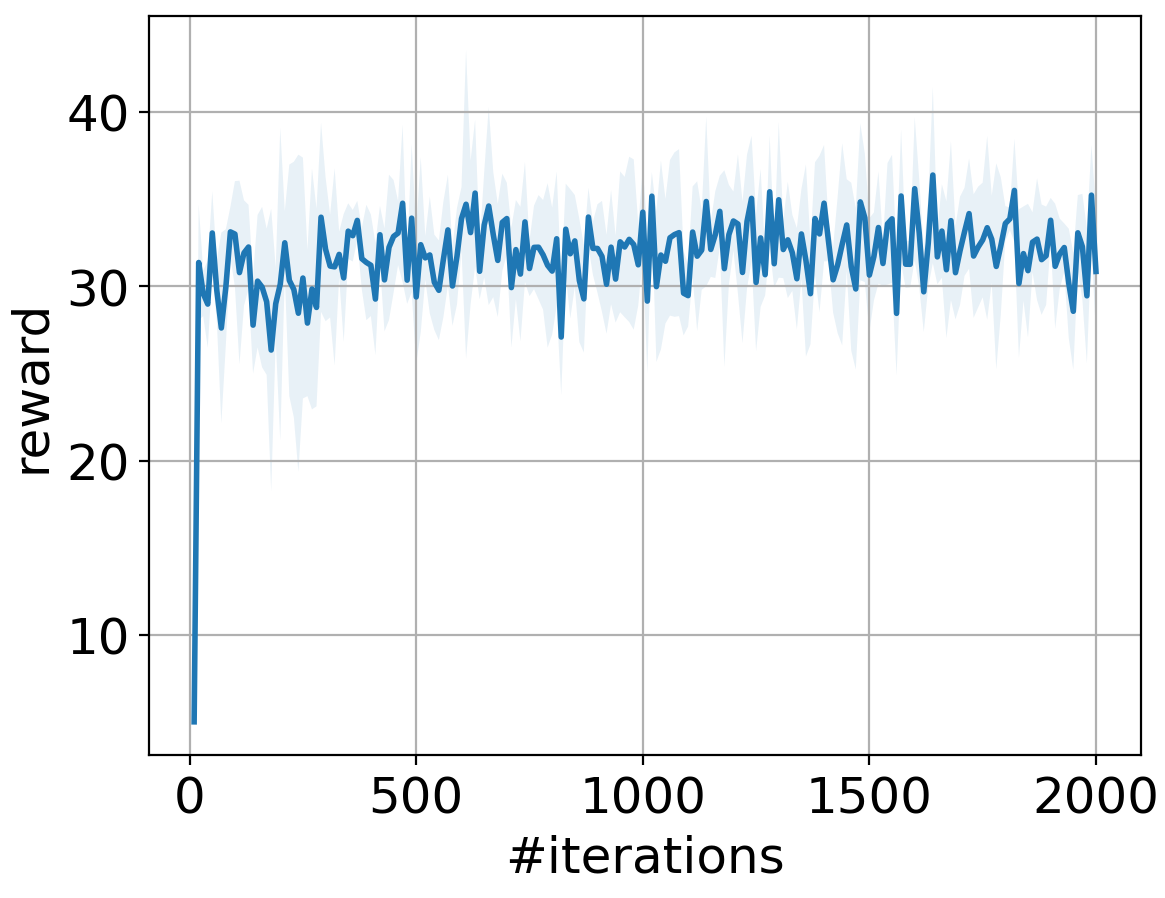}
      \includegraphics[width=0.32\linewidth]{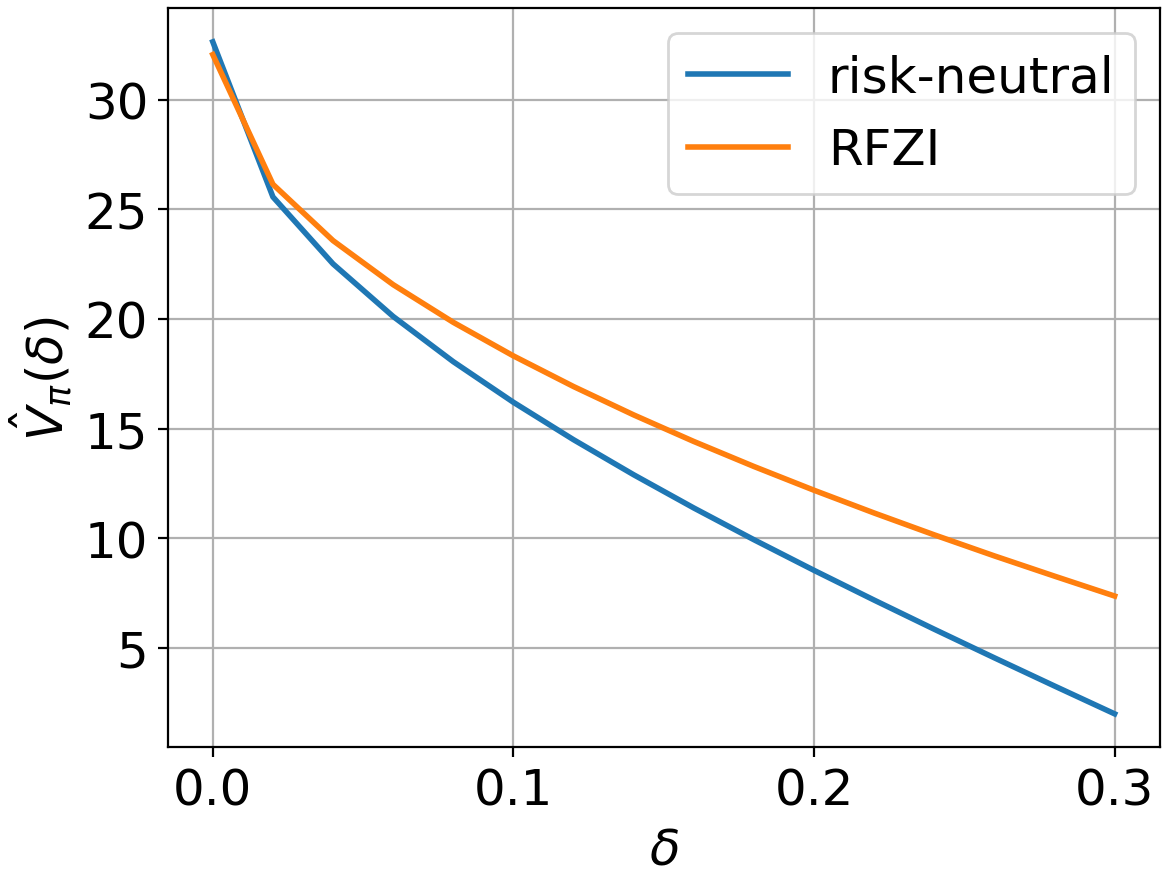}
      \subcaption{$\beta = 0.1$.}
    \end{subfigure}
    \caption{Simulation results for practical RFZI in the 100-state environment. (\textbf{Left:} optimality gap; \textbf{Middle:} average test reward; \textbf{Right:} robustness value.)}\label{fig:RFZI-multi-zone}
  \end{figure}

Simulation results for different penalty magnitudes are shown below in Figure \ref{fig:RFZI-multi-zone}. It can be observed that in both cases the optimality gap decays to close to 0 over time, and the average test reward also converges to oscillating around a stable value. To further examine the robustness of the policies, we compare their robustness values against the risk-neutral policy (i.e., the optimal policy of the standard risk-neutral MDP). The robustness value curve suggests that the performance of our RFZI policy is more robust than the risk-neutral policy in face of model uncertainty, and the advantage increases with larger penalty magnitude $\beta$.

However, we would also like to point out some limitations of the practical RFZI algorithm. Firstly, the training dynamics becomes unstable with larger $\beta$, reflected by slower convergence of the gradient descent updates in each iteration. Additionally, the training of the network is sensitive to other hyperparameters including learning rate, number of batches and batch size, which have to be carefully tuned for satisfactory performance. It remains future work to design algorithms that are more robust and more stable with regard to the choice of the hyperparameters.%These two limitations in combination make it hard to obtain successful training results when $\beta$ is very large.%This is an issue that has to be settled so that the algorithm can bring practical values.

\subsection{More Numerical Details}\label{sec:more-numeric-details}

Code for reproducing the simulation results can be found at \url{https://github.com/huyangsh/risk-sensitive-RL_ICRL-2024}.

% \paragraph{Notations.} To simplify the notations, in this environment we shall always regard the hitting reward $r$ as a vector in $\bR^{n}$, and any deterministic policy $\pi$ as a vector in $\cA^{n}$.

\paragraph{Exact policy gradient.} Here we show the optimality gap (Figure \ref{fig:simulation-PG-full-loss}) and the policies (Table \ref{tab:simulation-PG-full-policy}) of a full range of experiments as specified therein. It can be verified that all converged policies are exactly the optimal policies obtained by solving the Bellman optimality equation for $Q^{\star}$. A closer look at the policies reveals a pattern that is similar to what we have observed before --- agents learn to move to lower-risk zones when the uncertainty in the environment (i.e., $\alpha$) is higher or when its penalty magnitude (i.e., $\beta$) is higher. 

  \begin{table}[H]
    \centering
    \caption{Policies found by the exact policy gradient algorithm.}\label{tab:simulation-PG-full-policy}
    
    \begin{small}
    \begin{tabular}{c|ccc}
        \specialrule{1.0pt}{0pt}{0pt}
        $\bm{\alpha}$ & $\bm{\beta}$ & \textbf{\#steps ($\bm{\eta = 0.1}$)} & \textbf{policy}$^{\bm{\dagger}}$ \\
        \hline
        \multirow{4}{*}{$0.01$} & $0.1$ & 226 & [-1, -1, 1, 1, 1, 1, 1, 1, 0, -1, -1, -1, -1, -1] \\
        & $1.0$ & 67 & [1, 1, 1, 0, -1, -1, -1, -1, 0, 1, 1, 1, 1, 1] \\
        & $2.0$ & 76 & [-1, -1, 1, 0, -1, -1, -1, -1, 0, 1, 1, 1, 0, -1] \\
        & $3.0$ & 192 & [-1, -1, -1, 0, -1, -1, -1, -1, 0, 1, 1, 1, 0, -1] \\
        \hline
        \multirow{3}{*}{$0.15$} & $0.01$ & 114 & [1, 1, 1, 0, -1, -1, -1, -1, 0, 1, 1, 1, 0, 1] \\
        & $0.1$ & 109 & [1, 1, 1, 0, -1, -1, -1, -1, 0, 1, 1, 1, 0, -1] \\
        & $1.0$ & 234 & [-1, -1, -1, -1, -1, -1, -1, -1, 0, 1, 1, 1, 0, -1] \\
        \specialrule{1.0pt}{0pt}{0pt}
    \end{tabular}
    
    ${}^{\dagger}$ Deterministic policies are represented by a vector in $\cA^{n}$, where an entry\\of the vector represents the action taken at the corresponding state.
    \end{small}
  \end{table}

  \begin{figure}[ht]
    \centering\captionsetup{justification=centering}
    \begin{subfigure}[b]{0.4\linewidth}
      \centering
      \includegraphics[width=\linewidth]{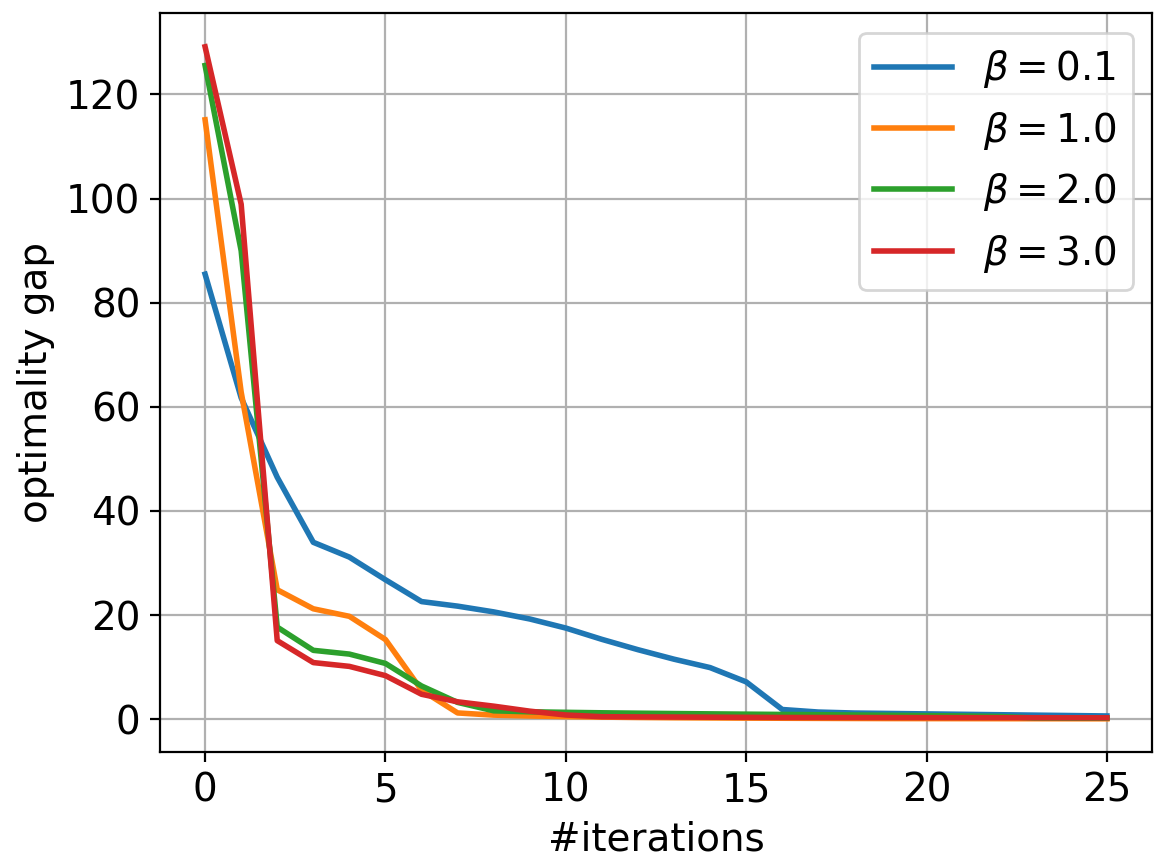}
      \subcaption{$\alpha = 0.01$.}
    \end{subfigure}\hspace{10pt}
    \begin{subfigure}[b]{0.4\linewidth}
      \centering
      \includegraphics[width=\linewidth]{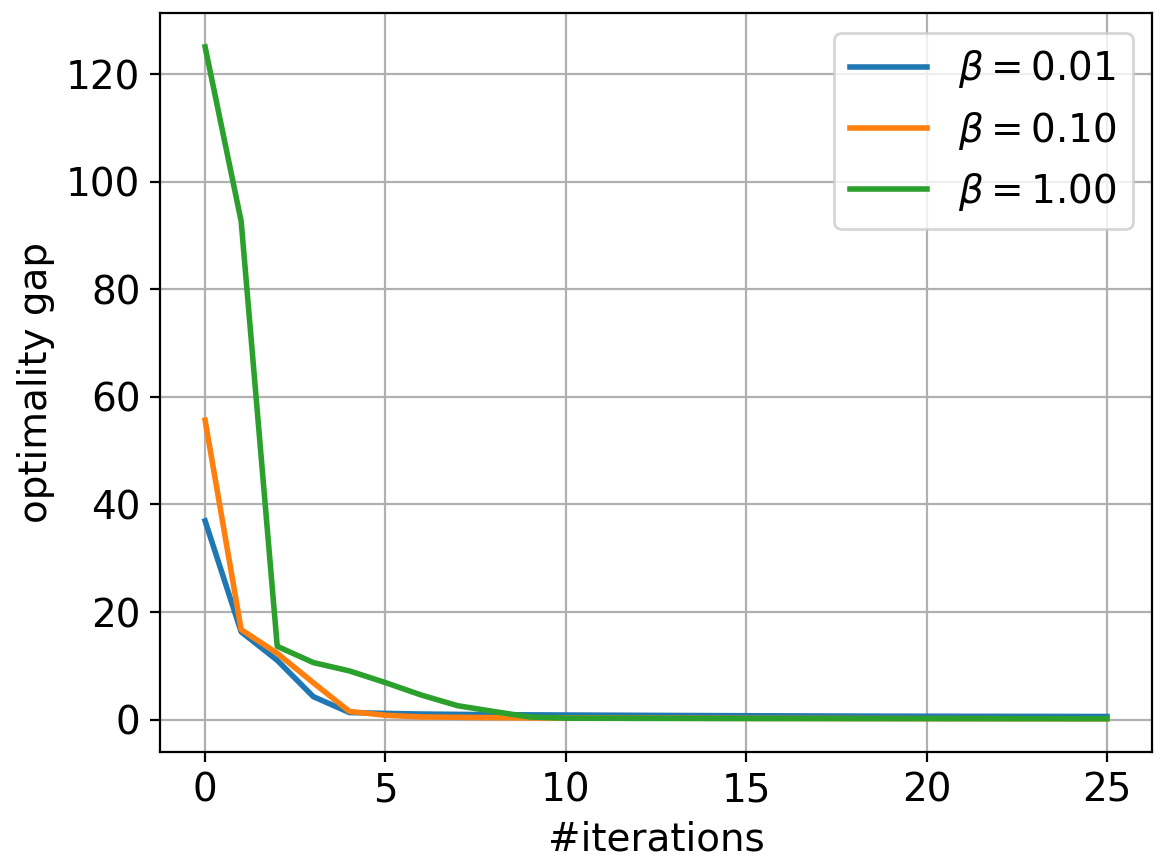}
      \subcaption{$\alpha = 0.15$.}
    \end{subfigure}
    \caption{Optimality gap curves for the exact policy gradient algorithm in different settings.}\label{fig:simulation-PG-full-loss}
  \end{figure}

\paragraph{RFZI.} In the implementation, we use sinusoidal embedding of states, i.e. 
\begin{equation*}
  \phi(s) = \left[ \sin \frac{2\pi}{N}, \sin \frac{4\pi}{N}, \ldots, \sin \frac{2N\pi}{N}, \cos \frac{2\pi}{N}, \cos \frac{4\pi}{N}, \ldots, \cos \frac{2N\pi}{N} \right],
\end{equation*}
which is similar to the embedding used in \cite{ren2023stochastic}. The Z-functions are approximated by a 3-layer network with a 256-dimensional first hidden layer and a 32-dimensional second hidden layer (both fully-connected and activated by ReLU). The output of the network is normalized by a sigmoid function to clamp the output in $(0,1)$ (in accordance with Assumption \ref{assump:positivity}).

\paragraph{Training details.} Training is performed on a workstation equipped with a 32-core CPU (Intel\textsuperscript{\textregistered} Xeon Platinum 8358, 2.60 GHz) and a NVIDIA\textsuperscript{\textregistered} A100 GPU. The average training time for a typical RFZI training scheme of 2000 iterations and 500 batches is about 5 hours.

\section{Relationship and Difference with Markov risk measures}\label{apdx:relationship-to-Markov-risk-measures}

In this section, we compare our definition of risk-sensitive MDP with the Markov risk measure. {Intuitively speaking, the Markov risk measure also takes the risk generated by the randomness of the policy into account whereas our definition treats the randomness of the policy in a risk-neutral manner and only considers risk from the uncertainty of the transition probability. This intuitively explains why the two notions are equivalent for deterministic policies but not for stochastic policies. }For clearness, we compare with the definition considered in \cite{huang2021convergence,tamar2015policy,tamar2017sequential}, where the reward $r(s)$ is only dependent on the state $s$ but not on action $a$.
The Markov risk measure for policy $\pi$ is defined as :
\begin{equation*}
    \Vmrm^\pi(s_0) = r(s_0) - \gamma\sigma\left(P^\pi_{s_0}, r(s_1) -\gamma\sigma\left(P^\pi_{s_1},r(s_2) - \gamma\sigma\left(P^\pi_{s_2}, r(s_3) - \cdots\right)\right)\right),
\end{equation*}
where $P_{s}^\pi$ is the transition probability defined by $P_{s}^\pi(s') = \sum_{a}\pi(a|s)P_{s,a}(s')$. Thus, if we define the Markov-risk-measure-Bellman operator $\Tmrm^\pi: \bR^{|\cS|\to\bR^{|\cS|}}$as: 
\begin{align*}
    [\Tmrm^\pi V](s):= r(s) - \gamma\sigma(P_s^\pi, V) = r(s) -\gamma\sigma(\sum_{a}\pi(a|s)P_{s,a}, V)
\end{align*}
then the Markov risk measure is the fixed point of the Bellman operator, i.e. $\Vmrm^\pi = \Tmrm^\pi \Vmrm^\pi$.

In contrast, the value function $\tV^\pi$ of the risk-sensitive MDP \eqref{eq:markov-risk-measure-pi} is the fixed point of the following risk-sensitive Bellman operator:
\begin{align*}
    [\cT^\pi V](s)&:= r(s) - \gamma\sum_a \pi(a|s)\sigma(\Psa, V)
\end{align*}

Note that the risk-sensitive Bellman operator $\cT^\pi$ is linear with respect to the policy $\pi$, whereas $\Tmrm$ can be potentially nonlinear w.r.t. $\pi$ as $\sigma$ is generally a nonlinear function. Thus for stochastic policies, $\Vmrm^\pi$ and $\tV^\pi$ are not equivalent. However, it is not hard to verify that when $\pi$ is a deterministic policy, $\cT^\pi$ and $\Tmrm^\pi$ are the same
\begin{align*}
    [\Tmrm^\pi V](s) = r(s) -\gamma\sigma(P_{s,\pi(s)}, V) = [\cT^\pi V](s).
\end{align*}
Thus, for deterministic policies, the value function $\tV^\pi$ and the Markov risk measure $\Vmrm^\pi$ are the same. {Additionally, when the risk-measure $\sigma$ is mixture quasiconcave (c.f. \cite{Delage2016DicesionMakingUU}), it can be shown that the optimal policy $\pi$ for the Markov risk measure can also be chosen as a deterministic policy; thus under this case the Markov risk measure and the risk-sensitive MDP obtain the same optimal value, i.e. $\Vmrm^\star = \tV^\star$.} {However, we would also like to emphasize that when adding policy regularization or that the risk measure $\sigma$ is not mixture quasiconcave (e.g. mean semi-deviation), the optimal policy might no longer be a deterministic policy. In this setting, the optimal policy and optimal value of the Markov risk measure and the risk-sensitive MDP might not be the same.}

%Additionally, when the risk-measure $\sigma$ is mixture quasiconcave (c.f. Delage et al.), it can be shown that the optimal policy $\pi$ for the Markov risk measure can also be chosen as a deterministic policy; thus under this case the Markov risk measure and the risk-sensitive MDP obtain the same optimal value, i.e. $\Vmrm^\star = \tV^\star$. However, we would also like to emphasize that when adding policy regularization or that the risk measure $\sigma$ is not a mixture quasiconcave (e.g. mean semi-deviation), the optimal policy might no longer be a deterministic policy. In this setting, the optimal policy and optimal value of the Markov risk measure and the risk-sensitive MDP may not be the same

% For stochastic policy, it is not equivalent
% For deterministic policy, equivalent
% Have the same optimal policy and optimal value

\section{Proof of Lemma \ref{lemma:Bellman-eq-existence-uniqueness}}\label{apdx:proof-of-Bellman-eq-existence}
Given a Markovian policy $\pi$ and a function on the state space $\cS$, define the Bellman operator as:
\begin{align}
    [\cT^\pi V](s)&:= \sum_a \pi(a|s)\left(r(s,a) - \gamma \sigma(\Psa, V) \right) \label{eq:finite-bellman-2},
\end{align}
Further, define the optimal Bellman operator $\cT^\star$ as:
\begin{align}
    [\cT^\star V](s)&:=\max_a \left(r(s,a) - \gamma \sigma(\Psa, V) \right). \label{eq:finite-bellman-star}
\end{align}

\begin{lemma}\label{lemma:Bellman-V-contraction}
    \begin{align*}
        \|\cT^\pi(V'-V)\|_\infty\le \gamma\|V'-V\|_\infty,~~\|\cT^\star(V'-V)\|_\infty \le \gamma\|V'-V\|_\infty.
    \end{align*}
\begin{proof}
    \begin{align*}
        [\cT^\pi(V'-V)](s) &= \sum_a \pi(a|s)\left(r(s,a) - \gamma \sigma(\Psa, V') \right) - \sum_a \pi(a|s)\left(r(s,a) - \gamma \sigma(\Psa, V) \right)\\
        &= \gamma \left(\sigma(\Psa, V) -\sigma(\Psa, V') \right)\\
        &\le \gamma \left(\sigPsa(V' - \|V'-V\|_\infty) -\sigma(\Psa, V') \right)\quad \textup{(monotonicity)}\\
        &= \gamma \|V'-V\|_\infty \qquad \textup{(translation invariance)}.
    \end{align*}
    Using the same analysis we can also get,
    \begin{align*}
        [\cT^\pi(V-V')](s)&\le \gamma \|V'-V\|_\infty\\
        \Longrightarrow~~ \|\cT^\pi(V'-V)\|_\infty &\le \gamma \|V'-V\|_\infty.
    \end{align*}
    Similarly, for $\cT^\star$,
\begin{align*}
    [\cT^\star(V'-V)](s) &= \max_a \left(r(s,a) - \gamma \sigma(\Psa, V') \right) - \max_a \left(r(s,a) - \gamma \sigma(\Psa, V) \right)\\
        &= \gamma \max_a\left(\sigma(\Psa, V) -\sigma(\Psa, V') \right) \quad \textup{(}\max_x f(x) - \max g(x) \le \max_x (f(x)-g(x))\textup{)}\\
        &\le \gamma \max_a\left(\sigPsa(V' - \|V'-V\|_\infty) -\sigma(\Psa, V') \right)\quad \textup{(monotonicity)}\\
        &= \gamma \|V'-V\|_\infty \qquad \textup{(translation invariance)}.\\
        \Longrightarrow~~ [\cT^\star(V-V')](s) &\le \gamma \|V'-V\|_\infty.\\
        \Longrightarrow~~ \|\cT^\star(V'-V)\|_\infty &\le \gamma \|V'-V\|_\infty.
\end{align*}
\end{proof}
\end{lemma}

\begin{proof}[Proof of Lemma \ref{lemma:Bellman-eq-existence-uniqueness}]
Lemma \ref{lemma:Bellman-eq-existence-uniqueness} is an immediate corollary of Lemma \ref{lemma:Bellman-V-contraction}. Note that $\tV^\pi$ and $\tV^\star$ in \eqref{eq:markov-risk-measure-pi} and \eqref{eq:markov-risk-measure-star} is the fixed point solution of
\begin{align*}
   \tV^\pi = \cT^\pi \tV^\pi, \quad \tV^\star = \cT^\star \tV^\star.
\end{align*}
From Lemma \ref{lemma:bellman-contraction} and the contraction mapping theorem \cite{istratescu1981fixed}, the fixed point solution exists and is unique, which completes the proof.
\end{proof}
\section{Proof of Theorem \ref{theorem:dual-representation}}
%\textcolor{blue}{TODO: Add a Lemma about the dual representation theorem}

% From the dual representation theorem, we can get
% \begin{equation}\label{eq:strong-duality}
%     \sigPsa(V) = \min_{\hPsa} \bE_{s'\sim \hPsa}V(s') + D(\hPsa, \Psa)
% \end{equation}

\subsection{Finite horizon discounting case}
We first define the value functions and Bellman operators for the finite horizon case. For any policy $\pi$ (doesn't necessarily need to be stationary or Markovian), define the value function as:
\begin{align}
    \oV_{0:h}^\pi(s):=\min_{\{\widehat{P}_t\}_{t=0}^h}\bE_{s_t,a_t\sim \pi,\widehat{P}}\sum_{t=0}^h \gamma^t\left(r(s_t,a_t)+\gamma D(\hP{}_{t;s_t,a_t},P_{s_t,a_t})\right).\label{eq:def-V-pi-finite-horizon}
\end{align}
Define $\oV_{0:h}^\star(s)$ as
\begin{align}
    \oV_{0:h}^\star(s) := \max_\pi\min_{\{\widehat{P}_t\}_{t=0}^h}\bE_{s_t,a_t\sim \pi,\widehat{P}}\sum_{t=0}^h \gamma^t\left(r(s_t,a_t)+\gamma D(\hP{}_{t;s_t,a_t},P_{s_t,a_t})\right).\label{eq:def-V-star-finite-horizon}
\end{align}

\begin{lemma}\label{lemma:Bellman-eq-finite-horizon-star}
$\oV_{0:h}^\star$ is given by:
\begin{equation*}
    \oV_{0:h+1}^\star = \cT^\star \oV_{0:h}^\star, ~~(\oV_{0:-1}^\star:= 0)
\end{equation*}
where $\cT^\star$ is defined as in \eqref{eq:finite-bellman-star}.
Further, for all state $s$, the max-min solution of \eqref{eq:def-V-star-finite-horizon} is given by a same set of policies and probability transitions:
\begin{align*}
    \pi_{t|h}^\star(s) &= \argmax_a \left(r(s,a) - \gamma \sigma(\Psa,\oV_{0:h-t-1}^\star)\right),\\
    \hP\star_{t|h;s,a} &= \argmin_{\widehat{P}} D(\widehat{P}, \Psa) + \bE_{s'\sim \widehat{P}}\oV^\star_{0:h-t-1}(s').
\end{align*}
\end{lemma}

\begin{proof}
    We prove by induction. The statements are trivial for $h=0$. Assume that the statements are true for $0\le t\le h$, then for $h+1$, from the definition of $\oV_h^\star$'s we have that
    \begin{align*}
        &\oV_{0:h+1}^\star(s_0) = \max_{\pi_0}\max_{\{\pi_t\}_{t=1}^{h+1}}\min_{\widehat{P}_0}\min_{\{\widehat{P}_t\}} \bE_{\pi, \widehat{P}}\sum_{t=0}^{h+1} \gamma^t\left(r(s_t,a_t)+\gamma D(\hP{}_{t;s_t,a_t},P_{s_t,a_t})\right)\\
        &\le  \max_{\pi_0}\min_{\widehat{P}_0}\max_{\{\pi_t\}_{t=1}^{h+1}}\min_{\{\widehat{P}_t\}} \bE_{\pi, \widehat{P}}\sum_{t=0}^{h+1} \gamma^t\left(r(s_t,a_t)+\gamma D(\hP{}_{t;s_t,a_t},P_{s_t,a_t})\right) ~~(\max_x\min_y f(x,y)\le \min_y\max_x f(x,y))\\
        &= \max_{\pi_0}\min_{\widehat{P}_0}\bE_{a_0, s_1\sim \pi_0, \widehat{P}_0}\! \left[r(s_0,a_0)\!+\!\gamma D(\hP{}_{0;s_0,a_0},P_{s_0,a_0}) \!+\! \left[\!\max_{\{\pi_t\}_{t=1}^{h+1}}\min_{\{\widehat{P}_t\}} \bE_{\pi, \widehat{P}}\sum_{t=1}^{h+1} \!\gamma^t\!\left(\!r(s_t,a_t)+\gamma D(\hP{}_{t;s_t,a_t},P_{s_t,a_t})\!\right)\right]\right]\\
         &= \max_{\pi_0}\min_{\widehat{P}_0}\bE_{a_0, s_1\sim \pi_0, \widehat{P}_0} \left[r(s_0,a_0)+\gamma D(\hP{}_{0;s_0,a_0},P_{s_0,a_0}) + \gamma \oV_{0:h}^\star(s_1)\right].%\\
        % &= \max_{\pi_0}\bE_{a_0\sim \pi_0} r(s_0,a_0)+\min_{\widehat P_0} \beta^{-1}\KL(\widehat{P}_{s_0,a_0}||P_{s_0,a_0}) + \gamma\sum_{s_1}\widehat{P}_{s_0,a_0}(s_1) V_h^\star(s_1)\\
        % &= \max_{a}r(s_0,a) - \beta^{-1}\log \bE_{s'\sim P_{s_0,a}} e^{-\beta\gamma V_{h}(s')}.
    \end{align*}
    Further, from the statement that $\pi_{t|h}^\star, \hP\star_{t|h;s,a}$ solves the min-max problem \eqref{eq:def-V-star-finite-horizon}, we have
\begin{align*}
    &\quad \oV_{0:h+1}^\star(s_0)= \max_{\pi_0}\max_{\{\pi_t\}_{t=1}^{h+1}}\min_{\widehat{P}_0}\min_{\{\widehat{P}_t\}} \bE_{\pi, \widehat{P}}\sum_{t=0}^{h+1} \gamma^t\left(r(s_t,a_t)+\gamma D(\hP{}_{t;s_t,a_t},P_{s_t,a_t})\right)\\
    &\ge \max_{\pi_0}\min_{\widehat{P}_0}\min_{\{\widehat{P}_t\}} \bE_{\pi_0, \{\pi_{t|h+1}^\star\}_{t=1}^{h+1}, \widehat{P}}\sum_{t=0}^{h+1} \gamma^t\left(r(s_t,a_t)+\gamma D(\hP{}_{t;s_t,a_t},P_{s_t,a_t})\right)\\
    &= \max_{\pi_0}\min_{\widehat{P}_0}\bE_{a_0, s_1\sim \pi_0, \widehat{P}_0} \Bigg[r(s_0,a_0)+\gamma D(\hP{}_{0;s_0,a_0},P_{s_0,a_0}) \\
    &\quad + \left. \left[\min_{\{\widehat{P}_t\}_{t=1}^{h+1}} \bE_{\{\pi_{t|h+1}^\star\}_{t=1}^{h+1}, \widehat{P}}\sum_{t=1}^{h+1} \gamma^t\left(r(s_t,a_t)+\gamma D(\hP{}_{t;s_t,a_t},P_{s_t,a_t})\right)\right]\right]\\
     &= \max_{\pi_0}\min_{\widehat{P}_0}\bE_{a_0, s_1\sim \pi_0, \widehat{P}_0}\left[ r(s_0,a_0)\!+\!\gamma D(\hP{}_{0;s_0,a_0},P_{s_0,a_0}) \!+\! \bE_{\{\pi_{t|h+1}^\star\}_{t=1}^{h+1}, \hP\star_{t|h+1}}\!\sum_{t=1}^{h+1}\! \gamma^t\!\left(\!r(s_t,a_t)\!+\!\gamma D(\hP{}_{t;s_t,a_t},P_{s_t,a_t})\!\right)\right]\\
     &= \max_{\pi_0}\min_{\widehat{P}_0}\bE_{a_0, s_1\sim \pi_0, \widehat{P}_0} \left[r(s_0,a_0)+\gamma D(\hP{}_{0;s_0,a_0},P_{s_0,a_0}) + \gamma \oV_{0:h}^\star(s_1)\right]~~(\pi^\star_{t+1|h+1} = \pi^\star_{t|h}, \hP\star_{t+1|h+1} = \hP\star_{t|h}).%\\
        %&= \max_{a}r(s,a) - \beta^{-1}\log \bE_{s'\sim P_{s,a}} e^{-\beta\gamma V_{h}(s')}.
\end{align*}
 Thus, we may conclude that 
 \begin{align*}
     \oV^\star_{0:h+1}(s_0) &= \max_{a_0}\min_{\widehat{P}_0} r(s_0,a_0)+\gamma\left(D(\hP{}_{0;s_0,a_0},P_{s_0,a_0}) + \bE_{s_1\sim \widehat{P}_0} \oV_{0:h}^\star(s_1)\right)\\
     &=\max_{a_0} r(s_0,a_0) - \gamma \sigma(P_{s_0,a_0},\oV^\star_{0:h}) ~~\quad \textup{(dual representation theorem)}\\
     &= \cT^\star \oV_{0:h}^\star,
 \end{align*}
 and that the max-min policies and probability transitions can be taken as:
\begin{align*}
    \pi_{0|h+1}^\star(s) &= \argmax_{a} r(s,a) - \gamma \sigma(\Psa,\oV^\star_{0:h})\\
    \pi_{t|h+1}^\star(s) = \pi_{t-1|h}^\star(s) &= \argmax_{a} r(s,a) - \gamma \sigma(\Psa,\oV^\star_{0:h-t}), \quad t\ge 1\\
    \hP\star_{0|h+1;s,a} &= \argmin_{\widehat{P}} D(\widehat{P}, \Psa) + \bE_{s'\sim \widehat{P}}\oV^\star_{0:h}(s')\\
    \hP\star_{t|h+1;s,a} &= \hP\star_{t-1|h;s,a} = \argmin_{\widehat{P}} D(\widehat{P}, \Psa) + \bE_{s'\sim \widehat{P}}\oV^\star_{0:h-t}(s'),\quad t\ge 1.
\end{align*}
The above arguments complete the proof.
\end{proof}

\begin{lemma}\label{lemma:Bellman-eq-finite-horizon-pi}
    For any stationary and Markovian policy $\pi$, we have that
    \begin{equation*}
        \oV_{0:h+1}^\pi = \cT^\pi \oV_{0:h}^\pi,
    \end{equation*}
    where $\cT^\pi$ is defined as in \eqref{eq:finite-bellman-2}.
    Further, for all state $s$, minimal solution of \eqref{eq:def-V-pi-finite-horizon} is given by the same set of probability transitions:
\begin{align*}
    \hP\pi_{t|h;s,a} = \argmin_{\widehat{P}} D(\widehat{P}, \Psa) + \bE_{s'\sim \widehat{P}}\oV^\pi_{0:h-t-1}(s').
\end{align*}
\end{lemma}

\begin{proof}
    Our proof is largely similar to Lemma \ref{lemma:Bellman-eq-finite-horizon-star} and is again by induction. The statements are trivial for $h=0$. Assume that the statements are true for $0\le t\le h$, then for $h+1$, from the definition of $V_h^\pi$'s we have that
    \begin{align*}
        &\oV_{h+1}^{\pi} (s_0) = \min_{\widehat{P}_0}\min_{\{\widehat{P}_t\}} \bE_{\pi, \widehat{P}}\sum_{t=0}^{h+1} \gamma^t\left(r(s_t,a_t)+\gamma D(\hP{}_{t;s_t,a_t},P_{s_t,a_t})\right)\\
        &= \min_{\widehat{P}_0} \min_{\{\widehat{P}_t\}} \bE_{\pi_0, \widehat{P}_0} \left[ r(s_0,a_0)+\gamma D(\hP{}_{0;s_0,a_0},P_{s_0,a_0}) + \bE_{\pi_{1:h+1}, \widehat{P}_{1:h+1}} \sum_{t=1}^{h+1} \gamma^t\left(r(s_t,a_t)+\gamma D(\hP{}_{t;s_t,a_t},P_{s_t,a_t})\right) \right] \\
        &= \min_{\widehat{P}_0} \bE_{\pi_0, \widehat{P}_0} \left[ r(s_0,a_0)+\gamma D(\hP{}_{0;s_0,a_0},P_{s_0,a_0}) + \min_{\{\widehat{P}_t\}} \bE_{\pi_{1:h+1}, \widehat{P}_{1:h+1}} \sum_{t=1}^{h+1} \gamma^t\left(r(s_t,a_t)+\gamma D(\hP{}_{t;s_t,a_t},P_{s_t,a_t})\right) \right]\\ %\quad(independence)\\
        &= \min_{\widehat{P}_0}\bE_{a_0, s_1\sim \pi_0, \widehat{P}_0} \left(r(s_0,a_0)+\gamma D(\hP{}_{0;s_0,a_0},P_{s_0,a_0}) + \gamma \oV_{0:h}^{\pi}(s_1)\right)\\
        &= \bE_{a_0 \sim \pi_0} \left( r(s_0,a_0)+\gamma\min_{\widehat P_0}\left(D(\hP{}_{0;s_0,a_0},P_{s_0,a_0}) + \bE_{s_1\sim\widehat{P}_{0;s_0,a_0}} \oV_{0:h}^\pi(s_1)\right) \right)\\
        &= \bE_{a_0 \sim \pi_0} \left( r(s_0,a_0) -\gamma \sigma(P_{s_0,a_0},\oV_{0:h}^\pi)\right).
    \end{align*}
    Here the last step follows from dual representation theorem. Further,  the minimal probability transitions can be taken as:
\begin{align*}
    \hP\pi_{0|h+1;s,a} &= \argmin_{\widehat{P}} D(\widehat{P}, \Psa) + \bE_{s'\sim \widehat{P}}\oV^\pi_{0:h}(s')\\
    \hP\pi_{t|h+1;s,a} &= \hP\pi_{t-1|h;s,a} = \argmin_{\widehat{P}} D(\widehat{P}, \Psa) + \bE_{s'\sim \widehat{P}}\oV^\pi_{0:h-t}(s'),\quad t\ge 1.
\end{align*}
The above arguments complete the proof.
\end{proof}

\subsection{Proof of Theorem \ref{theorem:dual-representation} (infinite horizon case)}
\begin{proof}
We first verify that for $\oV^\star, \oV^\pi$ defined in \eqref{eq:min-max-value-function-star}, \eqref{eq:min-max-value-function-pi},
\begin{align*}
    \lim_{h\to+\infty}\oV_{0:h}^\star =\oV^\star, ~~ \lim_{h\to+\infty}\oV_{0:h}^\pi = \oV^\pi.
\end{align*}
{
From the definition of $\oV^\pi$, we have that
\begin{align*}
    \oV^\pi(s) &= \inf_{\{\widehat P_t\}_{t\ge0}}\bE_{s_t,a_t\sim \pi,\widehat P} \left[\sum_{t=0}^{+\infty} \gamma^t \left(r(s_t,a_t) + \gamma D(\widehat{P}_{t;s_t,a_t}, P_{s_t, a_t}) \right)\Big|s_0 = s\right]\\
    &\le  \inf_{\{\widehat P_t\}_{t=0}^h}\bE_{s_t,a_t\sim \pi,\widehat P} \left[\sum_{t=0}^{h} \gamma^t \left(r(s_t,a_t) + \gamma D(\widehat{P}_{t;s_t,a_t}, P_{s_t, a_t}) \right)\Big|s_0 = s\right]\\
    &\qquad\qquad + \bE_{s_t,a_t\sim \pi, P} \left[\sum_{t=h}^{+\infty} \gamma^t \left(r(s_t,a_t) + \gamma D(P_{s_t,a_t}, P_{s_t, a_t}) \right)\Big|s_0 = s\right]\\
    &= \oV_{0:h}^\pi(s) + \bE_{s_t,a_t\sim \pi, P} \left[\sum_{t=h}^{+\infty} \gamma^t \left(r(s_t,a_t)\right)\Big|s_0 = s\right]\\
    &\le \oV_{0:h}^\pi(s) + \frac{\gamma^h}{1-\gamma},\\
    \oV^\pi(s) &= \inf_{\{\widehat P_t\}_{t\ge0}}\bE_{s_t,a_t\sim \pi,\widehat P} \left[\sum_{t=0}^{+\infty} \gamma^t \left(r(s_t,a_t) + \gamma D(\widehat{P}_{t;s_t,a_t}, P_{s_t, a_t}) \right)\Big|s_0 = s\right]\\
    &\ge \inf_{\{\widehat P_t\}_{t\ge0}}\bE_{s_t,a_t\sim \pi,\widehat P} \left[\sum_{t=0}^{h} \gamma^t \left(r(s_t,a_t) + \gamma D(\widehat{P}_{t;s_t,a_t}, P_{s_t, a_t}) \right)\Big|s_0 = s\right]\\
    & = \inf_{\{\widehat P_t\}_{t=0}^h}\bE_{s_t,a_t\sim \pi,\widehat P} \left[\sum_{t=0}^{h} \gamma^t \left(r(s_t,a_t) + \gamma D(\widehat{P}_{t;s_t,a_t}, P_{s_t, a_t}) \right)\Big|s_0 = s\right]\\
    & = \oV_{0:h}^\pi(s)\\
    \Longrightarrow ~~&|\oV^\pi(s) - \oV_{0:h}^\pi(s)| \le \frac{\gamma^h}{1-\gamma}~~\Longrightarrow~~\lim_{h\to+\infty}\oV_{0:h}^\pi = \oV^\pi.
\end{align*}
And similarly, for $\oV^\star$, we have
\begin{align*}
    \oV^\star(s) &= \sup_{\pi}\! \inf_{\{\widehat P_t\}_{t\ge0}}\!\!\bE_{s_t,a_t\sim \pi,\widehat P} \left[\sum_{t=0}^{+\infty} \gamma^t \left(r(s_t,a_t) + \gamma D(\widehat{P}_{t;s_t,a_t}, P_{s_t, a_t}) \right)\Big|s_0 = s\right].\\
    &\le \sup_{\pi} \left[\inf_{\{\widehat P_t\}_{t=0}^h}\bE_{s_t,a_t\sim \pi,\widehat P} \left[\sum_{t=0}^{h} \gamma^t \left(r(s_t,a_t) + \gamma D(\widehat{P}_{t;s_t,a_t}, P_{s_t, a_t}) \right)\Big|s_0 = s\right]\right.\\
    &\left.\qquad\qquad + \bE_{s_t, a_t\sim \pi, P} \left[\sum_{t=h}^{+\infty} \gamma^t \left(r(s_t,a_t) + \gamma D(P_{s_t,a_t}, P_{s_t, a_t}) \right)\Big|s_0 = s\right]\right]\\
    & \le \sup_{\pi}\inf_{\{\widehat P_t\}_{t=0}^h}\bE_{s_t,a_t\sim \pi,\widehat P} \left[\sum_{t=0}^{h} \gamma^t \left(r(s_t,a_t) + \gamma D(\widehat{P}_{t;s_t,a_t}, P_{s_t, a_t}) \right)\Big|s_0 = s\right]\\
    &\qquad\qquad +\sup_{\pi}\bE_{s_t, a_t\sim \pi, P} \left[\sum_{t=h}^{+\infty} \gamma^t \left(r(s_t,a_t) + \gamma D(P_{s_t,a_t}, P_{s_t, a_t}) \right)\Big|s_0 = s\right]\\
    &= \oV^\pi_{0:h}(s) + \sup_{\pi}\bE_{s_t, a_t\sim \pi, P} \left[\sum_{t=h}^{+\infty} \gamma^t \left(r(s_t,a_t) + \gamma D(P_{s_t,a_t}, P_{s_t, a_t}) \right)\Big|s_0 = s\right]\\
    &\le \oV^\pi_{0:h}(s) + \frac{\gamma^h}{1-\gamma},\\
    \oV^\star(s) &= \sup_{\pi}\! \inf_{\{\widehat P_t\}_{t\ge0}}\!\!\bE_{s_t,a_t\sim \pi,\widehat P} \left[\sum_{t=0}^{+\infty} \gamma^t \left(r(s_t,a_t) + \gamma D(\widehat{P}_{t;s_t,a_t}, P_{s_t, a_t}) \right)\Big|s_0 = s\right].\\
    &\ge \sup_{\pi} \left[\inf_{\{\widehat P_t\}_{t=0}^h}\bE_{s_t,a_t\sim \pi,\widehat P} \left[\sum_{t=0}^{h} \gamma^t \left(r(s_t,a_t) + \gamma D(\widehat{P}_{t;s_t,a_t}, P_{s_t, a_t}) \right)\Big|s_0 = s\right]\right]\\
    &= V_{0:h}^\star(s)\\
     \Longrightarrow ~~&|\oV^\star(s) - \oV_{0:h}^\star(s)| \le \frac{\gamma^h}{1-\gamma}~~\Longrightarrow~~\lim_{h\to+\infty}\oV_{0:h}^\star = \oV^\star.
\end{align*}

Then from Lemma \ref{lemma:Bellman-eq-finite-horizon-star} and Lemma \ref{lemma:Bellman-eq-finite-horizon-pi} we have that
\begin{align*}
     \oV_{0:h+1}^\star = \cT^\star \oV_{0:h}^\star, ~~ \oV_{0:h+1}^\pi = \cT^\pi \oV_{0:h}^\pi.
\end{align*}
Since $\cT^\star, \cT^\pi$ is a continuous mapping, taking the limit on both sides of the equations we get$$\oV^\star = \cT^\star\oV^\star, \oV^\pi = \cT^\pi \oV^\pi,$$
}
%Thus from the contraction mapping theorem \cite{istratescu1981fixed} we know that $\oV^\star, \oV^\pi$ defined in \eqref{eq:min-max-value-function-star}, \eqref{eq:min-max-value-function-pi} satisfies the fixed point equations
%$$\oV^\star = \cT^\star\oV^\star, \oV^\pi = \cT^\pi \oV^\pi,$$
i.e., \eqref{eq:markov-risk-measure-pi} and \eqref{eq:min-max-value-function-pi} obtains the same solution $\oV^\pi$ and \eqref{eq:markov-risk-measure-star} and \eqref{eq:min-max-value-function-star} obtains the same solution $\oV^\star$. 

Next, we will show that the claim that the minimal solution of \eqref{eq:min-max-value-function-pi} is given by \eqref{eq:hP-pi}. Since $\oV^\pi$ satisfies $\oV^\pi = \cT^\pi \oV^\pi$, we have
\begin{align*}
    \oV^\pi(s_0) &= \bE_{a_0\sim\pi(\cdot|s_0)} \left(r(s_0, a_0) + \gamma \min_{\hP{}_0} D(\hP{}_{0;s_0,a_0}, P_{s_0,a_0}) + \bE_{s_1\sim \hP{}_{0;s_0,a_0}} \oV^\pi(s_1)\right)\\
    & = \bE_{a_0\sim\pi(\cdot|s_0),s_1\sim \hP{\pi}_{s_0,a_0}} \left(r(s_0, a_0) + \gamma D(\hP{\pi}_{s_0,a_0},P_{s_0,a_0}) +\gamma \oV^\pi(s_1)\right).
\end{align*}
Apply this equation iteratively, we get
\begin{align*}
    \oV^\pi(s_0) &=\bE_{a_t\sim\pi(\cdot|s_t),s_{t+1}\sim \hP{\pi}_{s_t,a_t}} \left(r(s_0, a_0) +  \gamma D(\hP{\pi}_{s_0,a_0},P_{s_0,a_0}) + \gamma r(s_1,a_1) + \gamma^2  D(\hP{\pi}_{s_1,a_1},P_{s_1,a_1}) + \gamma^2 \oV^\pi(s_2)\right)\\
    &=...\\
    &= \bE_{a_t\sim\pi(\cdot|s_t),s_{t+1}\sim \hP{\pi}_{s_t,a_t}}  \sum_{t=0}^{+\infty} \gamma^t \left(r(s_t,a_t) + \gamma D(\hP{\pi}_{s_t,a_t}, P_{s_t, a_t}) \right),
\end{align*}
which concludes that the minimal solution is given by $\hP{\pi}$ defined in \eqref{eq:hP-pi}.

For $\oV^\star$. We aim to show that $\oV^\star = \oV^{\pi^\star}$. From the definition of $\pi^\star$ and the fact that $\oV^\star = \cT^\star\oV^\star$,
\begin{align*}
   \oV^\star(s_0) &= \max_{a_0}r(s_0, a_0) - \gamma\sigma(P_{s_0,a_0},\oV^\star)\\  
    &\stackrel{\eqref{eq:pi-star}}{=}\bE_{a_0\sim\pi^\star(\cdot|s_0)} \left(r(s_0, a_0) - \gamma \sigma(P_{s_0,a_0},\oV^\star)\right)\\
    % &=\max_{a_0} r(s_0, a_0) + \gamma \min_{\widehat P_0}\left(D(\hP{}_{0;s_0,a_0}, P_{s_0,a_0}) + \bE_{s_1\sim \hP{}_{0;s_0,a_0}}\oV^\star(s_1)\right)\\
    &=\bE_{a_0\sim\pi^\star(\cdot|s_0)} \left(r(s_0, a_0) + \gamma \min_{\widehat P_0}\left(D(\hP{}_{0;s_0,a_0}, P_{s_0,a_0}) + \bE_{s_1\sim \hP{}_{0;s_0,a_0}}\oV^\star(s_1)\right)\right) \\
    &=\min_{\hP{}_{0}} \bE_{a_0\sim\pi^\star(\cdot|s_0)} \left(r(s_0, a_0) + \gamma \left(D(\hP{}_{0;s_0,a_0}, P_{s_0,a_0}) + \bE_{s_1\sim \hP{}_{0;s_0,a_0}}\oV^\star(s_1)\right)\right).
\end{align*}
Apply the above equation iteratively we get
\begin{align*}
    \oV^\star(s_0) &= \min_{\{\hP{}_t\}_{t=1}^{+\infty}}\bE_{a_t\sim\pi^\star(\cdot|s_t), s_{t+1}\sim\hP{}_{t;s_t,a_t}} \sum_{t=0}^{+\infty} \gamma^t \left(r(s_t,a_t) + \gamma D(\hP{\pi}_{s_t,a_t}, P_{s_t, a_t}) \right)\\
     & = \oV^{\pi^\star}(s_0),
\end{align*}
which implies that the optimal value function can be obtained by the stationary policy $\pi^\star$. Thus, the minimal transition probability is given by
\begin{align*}
    \hP\star_{s,a} = \hP{\pi^\star}_{s,a} = \argmin_{\widehat{P}} D(\widehat{P}, \Psa) + \bE_{s'\sim \widehat{P}}\oV^\star(s'),
\end{align*}
which completes the proof.
\end{proof}

%%
% Experiments: Maybe start with grid-world, change the path difficulties https://movingai.com/benchmarks/street/index.html

% Short term:
% Stick with discrete space (grid): toy example or grid world

\section{Proof of Theorem \ref{theorem:policy-gradient}}
\begin{proof}[Proof of Theorem \ref{theorem:policy-gradient}]
    We first prove the differentiability of $V^\theta$ with respect to $\theta$ by the implicit function theorem \cite{loomis1968advanced}. We define the $|\cS|$-dimentional multivariate function $F( \theta, V)$ as follows:
    \begin{equation*}
        [F(\theta, V)](s) = V(s) -  \sum_a \pi_\theta(a|s)(r(s,a) - \gamma \sigma(\Psa, V)).
    \end{equation*}
    From the definition of the value function for Markov risk measures, $V^\theta$ is given by the following implicit function:
    \begin{equation*}
        F(\theta, V^\theta) = 0.
    \end{equation*}
    Thus, from the implicit function theorem, to prove the differentiability of $V^\theta$ with respect to $\theta$, it suffices to prove that the Jacobian matrix 
    \begin{equation*}
        J_{F,V}(\theta, V^\theta) = \left[\frac{\partial F_s}{\partial V_{s'}}\Big|_{V = V^\theta}
        \right]_{s, s'\in \cS}
    \end{equation*}
    is invertible. Here $F_s$ denotes the $s$-th entry of $F$ and $V_{s'}$ the $s'$-th entry of $V$. From Lemma \ref{lemma:Jacobian-calculation}, 
    \begin{align*}
    J_{F,V}(\theta, V^\theta) = I - \gamma \hP\theta_{\cS},
\end{align*}
    where $\hP\theta_{\cS}$ is a stochastic matrix. Thus, $\|\gamma \hP\theta_{\cS}\|_\infty\le \gamma < 1$, which implies that $ I - \gamma \hP\theta_{\cS}$ is invertible, thus from the implicit function theorem $V^\theta$ is differentiable w.r.t. $\theta$.

    Given the differentiability, what is left is to calculate the gradient. We can further use implicit function theorem to compute the gradient, yet another easier way is through the following algebraic manipulation:
    \begin{align*}
       V^\theta(s_0) &= \sum_{a_0} \pi_\theta(a_0|s_0)\left(r(s_0,a_0) - \gamma \sigma(P_{s_0,a_0},V^\theta) \right)\\
       \Longrightarrow~~ \nabla_\theta  V^\theta(s_0) &= \sum_{a_0}\nabla_\theta\pi_\theta(a_0|s_0)\left(r(s_0,a_0) - \gamma \sigma(P_{s_0,a_0},V^\theta) \right) \\
       &\quad - \gamma \sum_{a_0}\pi_\theta(a_0|s_0)\sum_{s_1}\frac{\partial \sigma(P_{s_0,a_0},\cdot)}{\partial V_{s_1}}\Big|_{V = V^\theta}\nabla_\theta V^\theta(s_1)\\
       &= \sum_{a_0}\pi_\theta(a_0|s_0)Q^\theta(s_0,a_0)\nabla_\theta\log\pi_\theta(a_0|s_0)\\
       &\quad +\gamma \sum_{a_0}\pi_\theta(a_0|s_0)\sum_{s_1}\hP\theta_{s_0,a_0}(s_1)\nabla_\theta V^\theta(s_1)~~~~\textup{(Lemma \ref{lemma:Jacobian-calculation})}\\
       &= \bE_{a_0\sim\pi_\theta(\cdot|s_0)}Q^\theta(s_0,a_0)\nabla_\theta\log\pi_\theta(a_0|s_0) + \gamma \bE_{s_1\sim\hP\theta_{s_0,a_0}}\nabla_\theta V^\theta(s_1)
    \end{align*}
    Applying the above equation iteratively we get:
\begin{align*}
   &\quad V^\theta(s_0) =  \bE_{a_0\sim\pi_\theta(\cdot|s_0)}Q^\theta(s_0,a_0)\nabla_\theta\log\pi_\theta(a_0|s_0) + \gamma \bE_{s_1\sim\hP\theta_{s_0,a_0}}\nabla_\theta V^\theta(s_1)\\
   &= \hspace{-15pt}\mathop{\bE}_{a_t\sim\pi_\theta(\cdot|s_t), s_{t+1}\sim\hP\theta_{s_{t},a_{t}}, t = 0,1}\hspace{-15pt}Q^\theta(s_0,a_0)\nabla_\theta\log\pi_\theta(a_0|s_0) + \gamma Q^\theta(s_1,a_1)\nabla_\theta\log\pi_\theta(a_1|s_1) + \gamma^2\nabla_\theta V^\theta(s_2)\\
   &=\cdots\\
   &= \bE_{a_t\sim \pi_\theta(\cdot|s_t), s_{t+1}\sim \hP\theta_{s_t,a_t}} \sum_{t=1}^{+\infty}\gamma^t Q^\theta(s_t,a_t)\nabla_\theta \log \pi_\theta(a_t|s_t),
\end{align*}
which completes the proof.
\end{proof}

\begin{lemma}\label{lemma:Jacobian-calculation}
\begin{align*}
    \frac{\partial \sigma(\Psa, \cdot)}{\partial V_{s'}}\Big|_{V = V^\theta} = -\hP\theta_{s,a}(s'),
\end{align*}
which implies that
\begin{align*}
    J_{F,V}(\theta, V^\theta) = I - \gamma \hP\theta_{\cS},
\end{align*}
where $\hP\theta_{\cS}$ is a stochastic matrix given by
\begin{align*}
    [\hP\theta_{\cS}]_{s,s'} = \hP\theta_{\cS}(s'|s) = \sum_{a}\pi_\theta(a|s)\hP\theta_{s,a}(s').
\end{align*}
\begin{proof}
From the definition of $\hP\theta_{s,a}(s')$ \eqref{eq:hP-pi} and the dual representation theorem (Theorem \ref{thm:convex-risk-dual-representation}) we have
\begin{align*}
    &\sigma(\Psa, V^\theta) = -\left(\min_{\widehat P}\bE_{s'\sim\widehat{P}} V^\theta(s')+ D(\widehat P, P_{s,a})\right) = -\bE_{s'\sim\hP\theta_{s,a}}V^\theta(s') -D(\hP\theta_{s,a}, P_{s,a})\\
    \Longrightarrow~~ &D(\hP\theta_{s,a}, P_{s,a}) = -\sigma(\Psa, V^\theta)- \bE_{s'\sim\hP\theta_{s,a}}V^\theta(s').
\end{align*}
From the definition of $D$:
\begin{align*}
    D(\hP\theta_{s,a}, \Psa) = \sup_V -\sigma(\Psa, V) - \bE_{s'\sim \hP\theta_{s,a}}V(s')\\
    \Longrightarrow ~~ V^\theta = \argmax_V -\sigma(\Psa, V) - \bE_{s'\sim \hP\theta_{s,a}}V(s'),
\end{align*}
thus
\begin{align*}
    \nabla_V \left(\sigma(\Psa, V) + \bE_{s'\sim \hP\theta_{s,a}}V(s')\right)\Big|_{V = V^\theta} = 0\\
    \Longrightarrow ~~ \frac{\partial\sigma(\Psa,\cdot)}{\partial V_{s'}}\Big|_{V = V^\theta} = -\frac{\partial \bE_{s''\sim \hP\theta_{s,a}}V(s'')}{\partial V_{s'}}\Big|_{V = V^\theta}=-\hP\theta_{s,a}(s').
\end{align*}
Then we have
\begin{align*}
    \frac{\partial F_s}{\partial V_{s'}}\Big|_{V = V^\theta} &= \frac{\partial \left(V(s) - \sum_a \pi_\theta(a|s)(r(s,a) - \gamma \sigma(\Psa, V))\right)}{\partial V_{s'}}\\
    &=\mathbf{1}\{s' = s\} + \gamma\sum_{a}\pi_\theta(a|s) \frac{\partial \sigma(\Psa,\cdot)}{\partial V_{s'}}\Big|_{V = V^\theta}\\
    &= \mathbf{1}\{s' = s\} - \gamma\sum_{a}\pi_\theta(a|s)\hP\theta_{s,a}(s')\\
    &= \mathbf{1}\{s' = s\} - \gamma[\hP\theta_{\cS}]_{s,s'},
\end{align*}
which completes the proof.
\end{proof}
\end{lemma}

\section{Proof of Lemma \ref{lemma:gradient-domination} and Theorem \ref{theorem:gradient-ascent}}
{ Before proving Lemma \ref{lemma:gradient-domination} and Theorem \ref{theorem:gradient-ascent}, we first introduce the performance difference lemma for soft RMDPs, which will play an important role in the following proofs. The lemma adopts from the performance difference lemma for risk-neutral MDPs (c.f. \cite{kakade2002approximate}) }
{
\begin{lemma}[Performance Difference Lemma for soft RMDPs]\label{lemma:performance-difference-soft-MDP}
Given stationary policies $\pi', \pi$, we have that
    \begin{align*}
       & \bE_{s_t,a_t\sim \pi',\hP\pi}\sum_{t=0}^{+\infty} \gamma^t\left(r(s_t,a_t)+\gamma D(\hP{\pi}_{t;s_t,a_t},P_{s_t,a_t})\right)- \bE_{s_t,a_t\sim \pi,\hP\pi}\sum_{t=0}^{+\infty} \gamma^t\left(r(s_t,a_t)+\gamma D(\hP{\pi}_{t;s_t,a_t},P_{s_t,a_t})\right)\notag\\
   &= \frac{1}{1-\gamma}\sum_{s,a} d^{\pi^\star,\hP\pi}(s) (\pi'(a|s) - \pi(a|s))Q^\pi(s,a)
    \end{align*}
\begin{proof}
 For notational simplicity in this proof we also define the value function $V^{\pi,\hP{}}$ and Q-function $Q^{\pi,\hP{}}$ for a given policy $\pi$ under a given probability transition $\hP{}$ as follows: 
 \begin{align*}
     V^{\pi,\hP{}}(s)&:= \bE_{s_t,a_t\sim \pi,\hP{}}\sum_{t=0}^{+\infty} \left[\gamma^t\left(r(s_t,a_t)+\gamma D(\hP{}_{t;s_t,a_t},P_{s_t,a_t})\right)|s_0 = s\right],\\
     Q^{\pi,\hP{}}(s)&:= \bE_{s_t,a_t\sim \pi,\hP{}}\sum_{t=0}^{+\infty} \left[\gamma^t\left(r(s_t,a_t)+\gamma D(\hP{}_{t;s_t,a_t},P_{s_t,a_t})\right)|s_0 = s, a_0 = a\right]
 \end{align*}
 Also from Theorem \ref{theorem:dual-representation} we know that $V^{\pi,\hP\pi} = V^\pi, Q^{\pi,\hP\pi} = Q^\pi$.
 Then we only need to show that
 \begin{align*}
     V^{\pi',\hP\pi}(s) - V^{\pi,\hP\pi}(s) = \frac{1}{1-\gamma}\sum_{s,a} d^{\pi^\star,\hP\pi}(s) (\pi'(a|s) - \pi(a|s))Q^\pi(s,a).
 \end{align*}
The left hand side of the equation can be decomposed as
    \begin{align*}
       & \bE_{s_t,a_t\sim \pi',\hP\pi}\sum_{t=0}^{+\infty} \gamma^t\left(r(s_t,a_t)+\gamma D(\hP{\pi}_{t;s_t,a_t},P_{s_t,a_t})\right)- \bE_{s_t,a_t\sim \pi,\hP\pi}\sum_{t=0}^{+\infty} \gamma^t\left(r(s_t,a_t)+\gamma D(\hP{\pi}_{t;s_t,a_t},P_{s_t,a_t})\right)\\
       &=\underbrace{ \bE_{s_t,\!a_t\!\sim\! \pi',\!\hP\pi}\!\!\sum_{t=0}^{+\infty} \!\gamma^t\!\left(r(s_t,a_t)\!+\!\gamma D(\hP{\pi}_{t;s_t,a_t},P_{s_t,a_t})\right)\!-\! \bE_{a_0\sim \pi', s_t,a_t\sim \pi,\hP\pi, t\ge 1}\sum_{t=0}^{+\infty}\! \gamma^t\left(r(s_t,a_t)\!+\!\gamma D(\hP{\pi}_{t;s_t,a_t},P_{s_t,a_t})\right)}_{\textup{Part A}}\\
       &+ \underbrace{\bE_{a_0\sim \pi', s_t\!,a_t\sim \pi,\hP\pi, t\ge 1}\sum_{t=0}^{+\infty} \gamma^t\left(r(s_t,a_t)\!+\!\gamma D(\hP{\pi}_{t;s_t,a_t},P_{s_t,a_t})\right) \!-\! \bE_{s_t,a_t\sim \pi,\hP\pi}\sum_{t=0}^{+\infty} \gamma^t\left(r(s_t,a_t)\!+\!\gamma D(\hP{\pi}_{t;s_t,a_t},P_{s_t,a_t})\right)}_{\textup{Part B}}.
    \end{align*}
    Note that 
    \begin{align*}
    &\textup{Part A} \\
    &=\bE_{a_0\sim \pi'}\left(\bE_{s_t,a_t\sim \pi',\hP\pi, t\ge 1}\sum_{t=0}^{+\infty} \gamma^t\left(r(s_t,a_t)+\gamma D(\hP{\pi}_{t;s_t,a_t},P_{s_t,a_t})\right)\right.\\&\left.\qquad\qquad- \bE_{s_t,a_t\sim \pi,\hP\pi, t\ge1}\sum_{t=0}^{+\infty} \gamma^t\left(r(s_t,a_t)+\gamma D(\hP{\pi}_{t;s_t,a_t},P_{s_t,a_t})\right)\right)\\
    & = \gamma\bE_{a_0\sim\pi', s_1\sim\hP\pi} (V^{\pi',\hP\pi}(s_1) - V^{\pi',\hP\pi}(s_1))
    \end{align*}
    and
    \begin{align*}
        \textup{Part B} &=\bE_{a_0\sim \pi'} Q^{\pi,\hP\pi}(s_0, a_0) - \bE_{a_0\sim \pi} Q^{\pi,\hP\pi}(s_0, a_0) \\
        &= \sum_{a_0}(\pi'(a_0|s_0) -\pi(a_0|s_0))Q^{\pi,\hP\pi}(s_0, a_0).
    \end{align*}
    Thus we get
    \begin{align*}
       &\quad  V^{\pi',\hP\pi}(s_0) - V^{\pi,\hP\pi}(s_0)  \\&=\gamma \bE_{a_0\sim\pi', s_1\sim\hP\pi} (V^{\pi',\hP\pi}(s_1) - V^{\pi',\hP\pi}(s_1)) + \sum_{a_0}(\pi'(a_0|s_0) -\pi(a_0|s_0))Q^{\pi,\hP\pi}(s_0, a_0).
    \end{align*}
    Applying this equality iteratively we get
    \begin{align*}
        V^{\pi',\hP\pi}(s_0) - V^{\pi,\hP\pi}(s_0) 
        &= \sum_{t=0}^\infty \gamma^t\bE_{a_\tau, s_\tau \sim \pi',\hP\pi} \sum_{a_t}(\pi'(a_t|s_t) -\pi(a_t|s_t))Q^{\pi,\hP\pi}(s_t, a_t)\\
        & = \frac{1}{1-\gamma}\sum_{s,a} d^{\pi^\star,\hP\pi}(s) (\pi'(a|s) - \pi(a|s))Q^\pi(s,a),
    \end{align*}
    which completes the proof.
\end{proof}
\end{lemma}
}
\begin{proof}[Proof of Lemma \ref{lemma:gradient-domination}]
From Theorem \ref{theorem:dual-representation}, we have
\begin{align*}
    \bE_{s_0\sim\rho}V^\star(s_0) &= \min_{\{\widehat{P}_t\}_{t=0}^h}\bE_{s_t,a_t\sim \pi^\star,\widehat{P}^\theta}\sum_{t=0}^{+\infty} \gamma^t\left(r(s_t,a_t)+\gamma D(\hP{\theta}_{t;s_t,a_t},P_{s_t,a_t})\right)\\
    &\le \bE_{s_t,a_t\sim \pi^\star,\hP\theta}\sum_{t=0}^{+\infty} \gamma^t\left(r(s_t,a_t)+\gamma D(\hP{\theta}_{t;s_t,a_t},P_{s_t,a_t})\right)\\
    \bE_{s_0\sim\rho}V^\theta(s_0) &= \bE_{s_t,a_t\sim \pi_\theta,\hP\theta}\sum_{t=0}^{+\infty} \gamma^t\left(r(s_t,a_t)+\gamma D(\hP{\theta}_{t;s_t,a_t},P_{s_t,a_t})\right)\\.
\end{align*}
%\yang{I think all $\hP{}_{t;s_t,a_t}$ except for the first one should be $\hP{\theta}_{t;s_t,a_t}$. The same below.} 
Thus
\begin{align}
   &\quad \bE_{s_0\sim\rho}V^\star(s_0) -  V^\theta(s_0)\notag\\
   &\le \bE_{s_t,a_t\sim \pi^\star,\hP\theta}\sum_{t=0}^{+\infty} \gamma^t\left(r(s_t,a_t)+\gamma D(\hP{\theta}_{t;s_t,a_t},P_{s_t,a_t})\right) \notag\\
   &\qquad - \bE_{s_t,a_t\sim \pi_\theta,\hP\theta}\sum_{t=0}^{+\infty} \gamma^t\left(r(s_t,a_t)+\gamma D(\hP{\theta}_{t;s_t,a_t},P_{s_t,a_t})\right)\notag\\
   &= \frac{1}{1-\gamma}\sum_{s,a} d^{\pi^\star,\hP\theta}(s) (\pi^\star(a|s) - \pi_\theta(a|s))Q^\theta(s,a)~~~\textup{\small (by Lemma \ref{lemma:performance-difference-soft-MDP})}\notag\\
   &\le \frac{1}{1-\gamma}\sum_{s,a} d^{\pi^\star,\hP\theta}(s) \max_{\overline\pi}(\overline \pi(a|s) - \pi_\theta(a|s))Q^\theta(s,a)\label{eq:optimality-gap-ineq}\\
   &\le \frac{1}{1-\gamma}\left\|\frac{d^{\pi^\star,\hP\theta}}{d^{\pi_\theta,\hP\theta}}\right\|_\infty \sum_{s,a} d^{\pi_\theta,\hP\theta}(s) \max_{\overline\pi}(\overline \pi(a|s) - \pi_\theta(a|s))Q^\theta(s,a)\notag\\
   &= \left\|\frac{d^{\pi^\star,\hP\theta}}{d^{\pi_\theta,\hP\theta}}\right\|_\infty \max_{\overline\pi} \left<\overline\pi - \pi_\theta, \bE_{s_0\sim\rho}\nabla_\theta V^\theta(s_0)\right> \notag
\end{align}
\end{proof}

\begin{proof}[Proof of Theorem \ref{theorem:gradient-ascent}] For notational simplicity, we define use $\pi_s:= \pi(\cdot|s)$ to denote the $|\cA|$-dimentional probability distribution. We also use the abbreviation $Q^{(k)}, \hP{(k)}$ to denote $Q^{\pi^{(k)}}, \hP{\pi^{(k)}}$. We also define the following variable that will be useful throughout the proof.
\begin{align*}
    G_\eta^{(k)}:= \frac{1}{\eta}\left(\pi_s^{(k+1)} - \pi_s^{(k)}\right)
\end{align*}

Similar to the proof of Lemma \ref{lemma:gradient-domination}, we have

\begin{align*}
    \bE_{s_0\sim\rho}V^{(k)}(s_0) &= \min_{\{\widehat{P}_t\}_{t=0}^h}\bE_{s_t,a_t\sim \pi^{(k)},\widehat{P}}\sum_{t=0}^{+\infty} \gamma^t\left(r(s_t,a_t)+\gamma D(\hP{}_{t;s_t,a_t},P_{s_t,a_t})\right)\\
    &\le \bE_{s_t,a_t\sim \pi^{(k)},\hP{(k+1)}}\sum_{t=0}^{+\infty} \gamma^t\left(r(s_t,a_t)+\gamma D(\hP{(k+1)}_{t;s_t,a_t},P_{s_t,a_t})\right)\\
    \bE_{s_0\sim\rho}V^{(k+1)}(s_0) &= \bE_{s_t,a_t\sim \pi^{(k+1)},\hP{(k+1)}}\sum_{t=0}^{+\infty} \gamma^t\left(r(s_t,a_t)+\gamma D(\hP{(k+1)}_{t;s_t,a_t},P_{s_t,a_t})\right).
\end{align*}
Thus
    \begin{align*}
        &\bE_{s_0\sim\rho} V^{(k+1)}(s_0) - V^{(k)}(s_0) \ge \bE_{s_t,a_t\sim \pi^{(k+1)},\hP{{(k+1)}}} \sum_{t=0}^{+\infty} \gamma^t \left(r(s_t,a_t) + \gamma D(\widehat{P}^{(k+1)}_{t;s_t,a_t}, P_{s_t, a_t}) \right)\\
    &\qquad\qquad\qquad\qquad\qquad\qquad-\bE_{s_t,a_t\sim \pi^{(k)},\hP{{(k+1)}}} \sum_{t=0}^{+\infty} \gamma^t \left(r(s_t,a_t) + \gamma D(\widehat{P}^{(k+1)}_{t;s_t,a_t}, P_{s_t, a_t}) \right)\\
    &= \frac{1}{1-\gamma}\sum_{s,a} d^{\pi^{(k)}, \hP{(k+1)}}(s) (\pi^{(k+1)}(a|s) - \pi^{(k)}(a|s))Q^{(k+1)}(s,a)~~~\textup{\small (by Lemma \ref{lemma:performance-difference-soft-MDP})}\\
    & =  \underbrace{\frac{1}{1-\gamma}\sum_{s,a} d^{\pi^{(k)}, \hP{(k+1)}}(s) (\pi^{(k+1)}(a|s) - \pi^{(k)}(a|s))Q^{(k)}(s,a)}_{\textup{Part I}}\\
    &~~ +  \underbrace{\frac{1}{1-\gamma}\sum_{s,a} d^{\pi^{(k)}, \hP{(k+1)}}(s) (\pi^{(k+1)}(a|s) - \pi^{(k)}(a|s))(Q^{(k+1)}(s,a)-Q^{(k)}(s,a))}_{\textup{Part II}}.
    \end{align*}
    From Lemma \ref{lemma:projection-property-1}, we have
    \begin{align*}
         &\sum_a (\pi^{(k+1)}(a|s) - \pi^{(k)}(a|s))Q^{(k)}(s,a) = \<\pi^{(k+1)}_s - \pi^{(k)}_s, Q^{(k)}(s,\cdot)\>\\
         &= \left\<\projA\left(\pi^{(k)}_s + \eta\frac{1}{1-\gamma} d^{\pi^{(k)}, \hP{(k)}}(s)Q^{(k)}(s,\cdot)\right), Q^{(k)}(s,\cdot)\right\>\\
         &\ge \frac{1-\gamma}{\eta d^{\pi^{(k)}, \hP{(k)}}(s)}\|\pi_s^{(k+1)} - \pi_s^{(k)}\|_2^2.
    \end{align*}
%\textcolor{}{TODO: define $\pi_s, G_\eta^{(k)}$}
Thus
\begin{align*}
    \textup{Part I} &\ge \frac{1}{\eta}\sum_{s} \frac{d^{\pi^{(k)}, \hP{(k+1)}}(s)}{d^{\pi^{(k)}, \hP{(k)}}(s)}\|\pi_s^{(k+1)} - \pi_s^{(k)}\|_2^2\\
    &\ge \frac{1}{\eta M}\|\pi^{(k+1)} - \pi^{(k)}\|_2^2\\
    &= \frac{\eta}{ M}\|G_\eta^{(k)}\|_2^2.
\end{align*}
\begin{rmk}\label{rmk:proof-dependency-on-M}
    Note that for standard MDP, the corresponding Part I can be bounded by ${\eta}\|G_\eta^{(k)}\|_2^2$ instead of $\frac{\eta}{ M}\|G_\eta^{(k)}\|_2^2$. This is the key reason why the dependency on $M$ is worse for our setting.
\end{rmk}

For Part II, we can bound it using Lemma \ref{lemma:Q-smoothness}
\begin{align*}
    |\textup{Part II}|&\le \frac{1}{1-\gamma}\sum_{s,a} d^{\pi^{(k)}, \hP{(k+1)}}(s) \left|\pi^{(k+1)}(a|s) - \pi^{(k)}(a|s)\right| \left|Q^{(k+1)}(s,a)-Q^{(k)}(s,a)\right|\\
    & \le \frac{1}{(1-\gamma)^3}\sum_{s,a} d^{\pi^{(k)}, \hP{(k+1)}}(s) \left|\pi^{(k+1)}(a|s) - \pi^{(k)}(a|s)\right| \max_s\|\pi^{(k+1)}_s - \pi^{(k)}_s\|_1\\
    &\le \frac{1}{(1-\gamma)^3} \left(\max_s\|\pi^{(k+1)}_s - \pi^{(k)}_s\|_1\right)^2\\
    & \le  \frac{|\cA|}{(1-\gamma)^3} \|\pi^{(k+1)} - \pi^{(k)}\|_2^2\\
    &= \frac{\eta^2|\cA|}{(1-\gamma)^3}\|G_\eta^{(k)}\|_2^2.
\end{align*}
Combining the bounds on Part I and II we get, for $\eta = \frac{(1-\gamma)^3}{2|\cA|M}$
\begin{align*}
    \bE_{s_0\sim\rho} V^{(k+1)}(s_0) - V^{(k)}(s_0) \ge \left(\frac{\eta}{ M} - \frac{\eta^2|\cA|}{(1-\gamma)^3}\right)\|G_\eta^{(k)}\|_2^2
    = \frac{(1-\gamma)^3}{4|\cA|M} \|G_\eta^{(k)}\|_2^2
\end{align*}
Using the telescoping technique we get
\begin{align}
   \sum_{k=0}^{K-1}\|G_\eta^{(k)}\|_2^2&\le \frac{4|\cA|M^2}{(1-\gamma)^3}\sum_{k=0}^{K-1}\bE_{s_0\sim\rho} (V^{(k+1)}(s_0) - V^{(k)}(s_0)) \notag\\
  & {\le  \frac{4|\cA|M^2}{(1-\gamma)^3} \left(V^{(K)}(s_0) - V^{(0)}(s_0)\right)\le \frac{4|\cA|M^2}{(1-\gamma)^4}}\label{eq:telescoping}
\end{align}
{where the last inequality uses the fact that $ 0\!\le\! V^\pi(s) \!\le\! \bE_{s_t,a_t\sim \pi, P} [\sum_{t=0}^{+\infty} \gamma^t r(s_t,a_t) \Big|s_0 \!\!=\!\! s] \!\le\! \frac{1}{1-\gamma}$.}
\textit{Claim:}
\begin{align*}
    \bE_{s_0\sim\rho}V^\star(s_0) - V^{(k+1)}(s) \le \left(M + \frac{\eta|\cA|}{(1-\gamma)^3}\right)\|G_\eta^{(k)}\|_2
\end{align*}
\begin{proof}[Proof of Claim]
    From the definition of projection, for any $\pi_s'\in\Delta^{|\cA|}$,
    \begin{align*}
        \<\pi^{(k)}_s + \eta\frac{1}{1-\gamma} d^{\pi^{(k)}, \hP{(k)}}(s)Q^{(k)}(s,\cdot) &- \pi_s^{(k+1)}, \pi_s' - \pi_s^{(k+1)}\>\le 0\\
        \Longrightarrow ~~ \<Q^{(k)}(s,\cdot), \pi_s' - \pi_s^{(k+1)}\>&\le \frac{1-\gamma}{\eta d^{\pi^{(k)}, \hP{(k)}}(s)} \<\pi_s^{(k+1)} - \pi_s^{(k)}, \pi_s' - \pi^{(k+1)}\>\\
        &\le \frac{(1-\gamma)M}{\eta }\|\pi_s^{(k+1)} - \pi_s^{(k)}\|_2\|\pi_s' - \pi^{(k+1)}\|_2\\
        &\le  (1-\gamma){M} \|G_\eta^{(k)}\|_2\|\pi_s' - \pi^{(k+1)}\|_2.
    \end{align*}
    Thus
    \begin{align*}
       &\<Q^{(k+1)}(s,\cdot), \pi_s' - \pi_s^{(k+1)}\>\le \<Q^{(k)}(s,\cdot), \pi_s' - \pi_s^{(k+1)}\> + \<Q^{(k+1)}(s,\cdot)-Q^{(k)}(s,\cdot), \pi_s' - \pi_s^{(k+1)}\>\\
       &\le  (1-\gamma){M}  \|G_\eta^{(k)}\|_2\|\pi_s' - \pi^{(k+1)}\|_2 + \|Q^{(k+1)}(s,\cdot)-Q^{(k)}(s,\cdot)\|_\infty\|\pi_s' - \pi_s^{(k+1)}\|_1\\
       &\le  (1-\gamma){M}  \|G_\eta^{(k)}\|_2\|\pi_s' - \pi^{(k+1)}\|_2 + \sqrt{|\cA|}\|Q^{(k+1)}(s,\cdot)-Q^{(k)}(s,\cdot)\|_\infty\|\pi_s' - \pi_s^{(k+1)}\|_2.
    \end{align*}
    Further, from Lemma \ref{lemma:Q-smoothness}, 
    \begin{align*}
        \|Q^{(k+1)}(s,\cdot)-Q^{(k)}(s,\cdot)\|_\infty &\le \frac{1}{(1-\gamma)^2}\max_{\overline s\in \cS}\|\pi_{\overline s}^{(k+1)} - \pi_{\overline s}^{(k)}\|_1\le  \frac{\sqrt{|\cA|}}{(1-\gamma)^2}\|\pi^{(k+1)} - \pi^{(k)}\|_2\\
        &= \frac{\eta \sqrt{|\cA|}}{(1-\gamma)^2}\|G_\eta^{(k)}\|_2.
    \end{align*}
    Thus
    \begin{align*}
        \<Q^{(k+1)}(s,\cdot), \pi_s' - \pi_s^{(k+1)}\>&\le (1-\gamma){M}  \|G_\eta^{(k)}\|_2\|\pi_s' - \pi^{(k+1)}\|_2 +  \frac{\eta {|\cA|}}{(1-\gamma)^2}\|G_\eta^{(k)}\|_2\|\pi_s' - \pi_s^{(k+1)}\|_2\\
        &= \left( (1-\gamma){M}  + \frac{\eta {|\cA|}}{(1-\gamma)^2}\right)\|G_\eta^{(k)}\|_2\|\pi_s' - \pi^{(k+1)}\|_2
    \end{align*}
    From the proof of Lemma \ref{lemma:gradient-domination} (inequality \eqref{eq:optimality-gap-ineq}), we have
    \begin{align*}
        \bE_{s_0\sim\rho}V^\star(s_0) - V^{(k+1)}(s) &\le \frac{1}{1-\gamma}\sum_{s.a}d^{\pi^\star,\hP{(k+1)}}(s)\max_{\pi'}(\pi'(a|s) - \pi^{k+1}(a|)s)Q^{(k+1)}(s,a)\\
        & = \frac{1}{1-\gamma}\sum_{s.a}d^{\pi^\star,\hP{(k+1)}}(s) \max_{\pi'}\<Q^{(k+1)}(s,\cdot), \pi_s' - \pi_s^{(k+1)}\>\\
        &\le \frac{1}{1-\gamma}\max_s \max_{\pi'}\<Q^{(k+1)}(s,\cdot), \pi_s' - \pi_s^{(k+1)}\>\\
        &\le \frac{1}{1-\gamma }\left( (1-\gamma){M}  + \frac{\eta {|\cA|}}{(1-\gamma)^2}\right)\|G_\eta^{(k)}\|_2\|\pi_s' - \pi^{(k+1)}\|_2\\
        &= \left(M + \frac{\eta|A|}{(1-\gamma)^3}\right)\|G_\eta^{(k)}\|_2,
    \end{align*}
    which completes the proof of the claim.
\end{proof}
Substitute the claim into \eqref{eq:telescoping} we get
\begin{align*}
    \sum_{k=1}^K \left(\bE_{s_0\sim\rho}V^\star(s_0) - V^{(k)}(s_0)\right)^2 \le \frac{16|\cA| M^4}{(1-\gamma)^4},
\end{align*}
and thus
\begin{align*}
   \min_{1\le k\le K} \bE_{s_0\sim\rho}V^\star(s_0) - V^{(k)}(s_0) \le \sqrt{\frac{16|\cA| M^4}{(1-\gamma)^4 K}}.
\end{align*}
By setting 
\begin{equation*}
    K \ge \frac{16|\cA| M^4}{(1-\gamma)^4 \epsilon^2},
\end{equation*}
we get
\begin{align*}
   \min_{1\le k\le K} \bE_{s_0\sim\rho}V^\star(s_0) - V^{(k)}(s_0) \le \epsilon.
\end{align*}
\end{proof}
\section{Proof of Theorem \ref{theorem:sample-complexity}}\label{apdx:RFQI}
\subsection{Proof Sketches}
Before providing the full proof, in this section we first give a brief proof sketch of Theorem \ref{theorem:sample-complexity}. The lemmas in the proof sketch is proved in the following sections. We define the following auxiliary variables: $$Q_k(s,a) := r(s,a) - \gamma\beta^{-1}\log Z_k(s,a), \quad \pi_k(s) := \argmax_a Q_k(s,a).$$
The proof of Theorem \ref{theorem:sample-complexity} can be decoupled into the following four steps:
\paragraph{Step 1: Decomposition of the performance difference.} In this step, we first decompose the performance difference $\bE_{s_0\sim \rho}(V^{\star} - V^{\pi_k})(s_0)$ in terms of $[Q^\star - Q_k]$ and $[Q_k - Q^\star]$.
\begin{lemma}\label{lemma:performance-difference}
\begin{align*}
&\bE_{s_0\sim \rho}(V^{\star} - V^{\pi_k})(s_0) \le\\ &\bE_{s_0\sim\rho, s_{\tau+1} \sim \hP{\pi_k}_{s_\tau,a_\tau}, a_\tau\sim \pi_k(s_\tau)} \sum_{t=0}^{+\infty}\gamma^t ([Q^\star - Q_k](s_t,\pi^\star(s_t))) + ([Q_k - Q^\star](s_t,\pi_k(s_t))).
   %\bE_{s_0\sim\rho} V^{\pi^\star}(s_0) - V^{\pi_k}(s_0) \le \frac{1}{1-\gamma}\|Q^\star - Q_k\|_{1, d^{\pi_k,\hP{\pi_k}}\odot\pi^\star} + \|Q^\star - Q_k\|_{1, d^{\pi_k,\hP{\pi_k}}\odot\pi_k},
\end{align*}
where $\hP{\pi_k}$ is defined as $\hP{\pi_k}(s'|s,a)\propto P(s'|s,a)\exp(-\gamma V^{\pi_k}(s'))$.
\end{lemma}
\paragraph{Step 2: Bound $[Q^\star - Q_k]$ and $[Q_k - Q^\star]$.} Given Lemma \ref{lemma:performance-difference}, the next step is to further upper-bound $[Q^\star - Q_k]$ and $[Q_k - Q^\star]$. We have the following lemma:
\begin{lemma}\label{lemma:bellman-contraction}
    For any distribution $\nu$ that is admissible,
    \begin{align*}
        \bE_{s,a\sim\nu}(Q_k - Q^\star)(s,a),~~ \bE_{s,a\sim\nu}(Q^\star - Q_k)(s,a) \le \frac{\gamma^k}{1-\gamma} + C\sum_{m=1}^k\gamma^{k-m}\|Q_m - \cT_Q Q_{m-1}\|_{1,\mu},
    \end{align*}
    where $C$ is defined in Assumption \ref{assump:concentrability}.
\end{lemma}
The upper bound consists of two parts. The first part is caused by the contraction mapping of Bellman operator, and the second part captures the error of replacing the Bellman operation $\cT_Q Q_{m-1}$ with its approximation $Q_m$.
\paragraph{Step 3: Bound $\|Q_k - \cT_Q Q_{k-1}\|_{1,\mu} $.} According to Lemma \ref{lemma:bellman-contraction}, to bound  $[Q^\star - Q_k]$ and $[Q_k - Q^\star]$, we need to have a better understanding of $\|Q_k - \cT_Q Q_{k-1}\|_{1,\mu} $.
\begin{lemma}\label{lemma:bellman-error-decomposition}
\begin{align*}
    \|Q_k - \cT_Q Q_{k-1}\|_{1,\mu}\le \gamma\beta^{-1}e^{\frac{\beta}{1-\gamma}} \left(\|(\hatTzf -\Tzf) Z_{k-1}\|_{1,\mu} + \|(\Tzf - \Tz)Z_{k-1}\|_{1,\mu}\right).    
\end{align*}
\end{lemma}
Lemma \ref{lemma:bellman-error-decomposition} suggests that the error $ \|Q_k - \cT_Q Q_{k-1}\|_{1,\mu}$ can be bounded by two parts, the first part $\|(\Tzf - \Tz)Z_{k-1}\|_{1,\mu}$ is the error caused by function approximation, i.e., replacing the Bellman operator with the project Bellman operator. The second part is the error of replacing the projected Bellman operator with the empirical projected Bellman operator.
\paragraph{Step 4: Bound $\|(\hatTzf -\Tzf) Z_{k-1}\|_{1,\mu} $ and $\|(\Tzf - \Tz)Z_{k-1}\|_{1,\mu}$.}
The last step closes the proof by bounding $\|(\hatTzf -\Tzf) Z_{k-1}\|_{1,\mu} $ and $\|(\Tzf - \Tz)Z_{k-1}\|_{1,\mu}$ which shows up on the right hand side of Lemma \ref{lemma:bellman-error-decomposition}. 

\begin{lemma}\label{lemma:empirical-loss-estimate}
\begin{equation*}
    \|(\Tzf - \Tz)Z_{k-1}\|_{2,\mu}\le \epsilon_c.
\end{equation*}
With probability at least $1-\delta$,
    \begin{align*}
        \|(\hatTzf -\Tzf) Z_{k-1}\|_{2,\mu}&\le 4\sqrt{\frac{2\log(|\cF|)}{N}} + 5\sqrt{\frac{2\log(8/\delta)}{N}}
    \end{align*}
\end{lemma}

Combining the four steps finishes the proof.

\subsection{Proof of Lemmas in Proof Sketches}
In this section we define the following operator $\Tvq:\bR^{|\cS|}\to \bR^{|\cS|\times|\cA|}$ for notational simplicity: 
\begin{align}\label{eq:Tvq}
    [\Tvq V](s,a) = r(s,a) - \gamma\beta^{-1}\log \bE_{s'\sim P(\cdot|s,a)} e^{-\beta V(s')}
\end{align}

\begin{proof}[Proof of Lemma \ref{lemma:performance-difference}]
Define $V_k(s):= \max_a Q_k(s,a)$, then
    \begin{align*}
        &\quad (V^{\star} - V^{\pi_k})(s_0)  = (V^{\star} - V_k)(s_0) + (V_k - V^{\pi_k})(s_0)\\
        &= (Q^\star(s_0,\pi^\star(s_0)) - Q_k(s_0,\pi_k(s_0)))  +  (Q_k(s_0,\pi_k(s_0)) - Q^{\pi_k}(s_0,\pi_k(s_0)))\\
        &\le (Q^\star(s_0,\pi^\star(s_0)) - Q_k(s_0,\pi^\star(s_0))) + (Q_k(s_0,\pi_k(s_0)) - Q^\star(s_0,\pi_k(s_0))) + (Q^\star(s_0,\pi_k(s_0)) - Q^{\pi_k}(s_0,\pi_k(s_0)))\\
        &\le ([Q^\star - Q_k](s_0,\pi^\star(s_0))) + ([Q_k - Q^\star](s_0,\pi_k(s_0))) + [\Tvq(V^\star - V^{\pi_k})](s_0,  \pi_k(s_0))\\
        &\stackrel{\textup{Lemma \ref{lemma:Bellman-operator-contraction}}}{\le} \gamma\bE_{s_1\sim\hP{\pi_k}_{s_0,a_0}, a_0 \sim \pi_k(s_0)}{} [V^\star - V^{\pi_k}](s_1)+ ([Q^\star - Q_k](s_0,\pi^\star(s_0))) + ([Q_k - Q^\star](s_0,\pi_k(s_0)))
    \end{align*}
Apply the above inequality recursively, and we get:
\begin{align*}
    \bE_{s_0\sim \rho}(V^{\star} - V^{\pi_k})(s_0) &\le \bE_{s_0\sim\rho, s_{\tau+1} \sim \hP{\pi_k}_{s_\tau,a_\tau}, a_\tau\sim \pi_k(s_\tau)} \sum_{t=0}^{+\infty}\gamma^t ([Q^\star - Q_k](s_t,\pi^\star(s_t))) + ([Q_k - Q^\star](s_t,\pi_k(s_t))).
\end{align*}
\end{proof}

    \begin{proof}[Proof of Lemma \ref{lemma:bellman-contraction}]
    We use proof by induction. The inequality naturally holds for $k=0$. Assume that it holds for $k-1$, then
    \begin{align*}
        \bE_{s,a\sim\nu}(Q_k - Q^\star)(s,a) &= \bE_{s,a\sim\nu}[Q_k - \cT_Q Q_{k-1}](s,a) + \bE_{s,a\sim\nu}[\cT_Q(Q_{k-1} - Q^\star)](s,a)\\
        &\stackrel{\textup{Lemma \ref{lemma:Bellman-operator-contraction}}}{\le} \bE_{s,a\sim\nu}[Q_k - \cT_Q Q_{k-1}](s,a) + \gamma \bE_{s,a\sim\nu}\bE_{s'\sim \hP{\star}_{s,a}} \max_{a'}(Q_{k-1} - Q^\star)(s',a')\\
        & = \bE_{s,a\sim\nu}[Q_k - \cT_Q Q_{k-1}](s,a) + \gamma \bE_{s',a'\sim\nu'} (Q_{k-1} - Q^\star)(s',a'),
    \end{align*}
    where $\nu'$ is the marginal distribution on $(s',a')$ given the joint distribution on $(s,a,s',a')$ by $s,a\sim \nu, s'\sim \hP{\star}_{s,a}, a' = \argmax_{a'} (Q_{k-1} - Q^\star)(s',a')$. Since $\nu$ is admissible, $\nu'$ is also admissible. Thus from the induction assumption, we have
    \begin{align*}
        \bE_{s,a\sim\nu}(Q_k - Q^\star)(s,a) &\le \bE_{s,a\sim\nu}[Q_k - \cT_Q Q_{k-1}](s,a) + \gamma \left(\frac{\gamma^{k-1}}{1-\gamma} + C\sum_{m=1}^{k-1}\gamma^{k-1-m}\|Q_m - \cT_Q Q_{m-1}\|_{1,\mu}\right)\\
        &\le \frac{\gamma^k}{1-\gamma} + C\sum_{m=1}^k\gamma^{k-m}\|Q_m - \cT_Q Q_{m-1}\|_{1,\mu}
    \end{align*}
    \end{proof}

\begin{proof}[Proof of Lemma \ref{lemma:bellman-error-decomposition}] From the definition of $Q_k$, we have
    \begin{align*}
        Q_k(s,a) &= r(s,a) - \gamma \beta^{-1}\log Z_k(s, a)\\
        \cT_Q Q_{k-1}(s,a) &= r(s,a) -\gamma \beta^{-1}\log \bE_{s'\sim P(\cdot|s,a)} e^{-\beta \max_{a'}Q_{k-1}(s',a')}\\
        &=r(s,a) - \gamma\beta^{-1}\log \bE_{s'\sim P(\cdot|s,a)} e^{-\beta\max_{a'}\left(r(s',a') - \gamma\beta^{-1}\log Z_{k-1}(s', a')\right)}\\
        &= r(s,a) - \gamma \beta^{-1}\log([\Tz Z_{k-1}](s,a)).\\
    \Longrightarrow~~ \|Q_k - \cT_Q Q_{k-1}\|_{1,\mu}&= \gamma\beta^{-1}\|\log Z_k - \log \Tz Z_{k-1}\|_{1,\mu}\\
    &\stackrel{\textup{Lemma \ref{lemma:log}}}{\le} \frac{\gamma\beta^{-1}}{\min_{s,a}\min\{Z_k(s,a), \Tz Z_
    {k-1}(s,a)\}}\|Z_k - \Tz Z_{k-1}\|_{1,\mu}\\
    &\le \gamma\beta^{-1}e^{\frac{\beta}{1-\gamma}}\|\hatTzf Z_{k-1}  - \Tz Z_{k-1}\|_{1,\mu}\\
    &\le \gamma\beta^{-1}e^{\frac{\beta}{1-\gamma}} \left(\|(\hatTzf -\Tzf) Z_{k-1}\|_{1,\mu} + \|(\Tzf - \Tz)Z_{k-1}\|_{1,\mu}\right).
    \end{align*}
\end{proof}

    \begin{proof}[Proof of Lemma \ref{lemma:empirical-loss-estimate}]
    The first inequality can be obtained by the approximate completeness assumption (Assumption \ref{assump:approx-completeness}).
\begin{align*}
     &\sup_{Z\in \cF}\inf_{Z'\in\cF}\|Z' - \Tz Z\|_{2,\mu} \le \epsilon_c\\
     \Longrightarrow~~& \inf_{Z'\in\cF}\|Z' - \Tz Z_{k-1}\|_{2,\mu} \le \epsilon_c\\
     \Longrightarrow~~&\textup{LHS} = \|\hatTzf Z_{k-1} - \Tz Z_{k-1}\|_{2,\mu}\le \epsilon_c
\end{align*}
       The second inequality can be obtained by applying Lemma 3, eq(12) in \cite{panaganti2022robust}. Here we can set $l(Z, (s,a,s')) = \left( Z(s,a) - \exp{\left(-\beta\gamma\max_{a'} (r(s',a') - \beta^{-1}\log Z_{k-1}(s',a'))\right)}\right)^2$. Then it is not hard to verify that $l(Z, (s,a,s')) \le 1$, $l(Z, (s,a,s'))$ is $2$-Lipschitz in $Z$ and that $|Z(s,a)|\le 1$, thus we can set $c_1 = 1, c_2 = 1, c_3 = 2$ in eq(12) in \cite{panaganti2022robust} and obtain that with probability at least $1-\delta$,
       \begin{align*}
       &\bE_{s,a\sim\mu} \Etrans  \left[l(\hatTzf Z_{k-1}, (s,a,s'))  -   l(\Tzf Z_{k-1}, (s,a,s'))\right] \le 4\sqrt{\frac{2\log(|\cF|)}{N}} + 5\sqrt{\frac{2\log(8/\delta)}{N}}\\
       \Longrightarrow ~~ &\cL(\hatTzf Z_{k-1}, Z_{k-1}) -  \cL(\Tzf Z_{k-1}, Z_{k-1}) \\
       &= \bE_{s,a\sim\mu}\left( \hatTzf Z_{k-1}(s,a) - \Etrans e^{-\beta\gamma\max_{a'} (r(s',a') - \beta^{-1}\log Z(s',a'))}\right)^2 \\
       &= \|(\hatTzf -\Tzf) Z_{k-1}\|_{2,\mu} \le4\sqrt{\frac{2\log(|\cF|)}{N}} + 5\sqrt{\frac{2\log(8/\delta)}{N}}
       \end{align*}
%     \textcolor{blue}{TODO: briefly read about the Rademacher complexity, Lemma 3 in \cite{panaganti2022robust} looks correct, but need to understand the lemma better.}
    \end{proof}

\subsection{Proof of Theorem \ref{theorem:sample-complexity}}
\begin{proof}[Proof of Theorem \ref{theorem:sample-complexity}]
\begin{align*}
    &\quad \bE_{s_0\sim \rho}(V^{\star} - V^{\pi_K})(s_0)\\ &\stackrel{\textup{Lemma~\ref{lemma:performance-difference}}}{=} \bE_{s_0\sim\rho, s_{\tau+1} \sim \hP{\pi_K}_{s_\tau,a_\tau}, a_\tau\sim \pi_K(s_\tau)} \sum_{t=0}^{+\infty}\gamma^t ([Q^\star - Q_K](s_t,\pi^\star(s_t))) + ([Q_K - Q^\star](s_t,\pi_K(s_t)))\\
    &\stackrel{\textup{Lemma~\ref{lemma:bellman-contraction}}}{\le} 2\sum_{t=0}^{+\infty}\gamma^t \left(\frac{\gamma^K}{1-\gamma} + C\sum_{k=1}^K\gamma^{K-k}\|Q_k - \cT_Q Q_{k-1}\|_{1,\mu}\right)\\
    &= \frac{2\gamma^K}{(1-\gamma)^2} + \frac{2C}{1-\gamma}\sum_{k=1}^K\gamma^{K-k}\|Q_k - \cT_Q Q_{k-1}\|_{1,\mu}\\
    &\stackrel{\textup{Lemma~\ref{lemma:bellman-error-decomposition}}}{\le}\frac{2\gamma^K}{(1-\gamma)^2} + \frac{2C}{1-\gamma}\gamma\beta^{-1}e^{\frac{\beta}{1-\gamma}}\sum_{k=1}^K\gamma^{K-k} \left(\|(\hatTzf -\Tzf) Z_{k-1}\|_{1,\mu} + \|(\Tzf - \Tz)Z_{k-1}\|_{1,\mu}\right)\\
    &\le \frac{2\gamma^K}{(1-\gamma)^2} + \frac{2C}{1-\gamma}\gamma\beta^{-1}e^{\frac{\beta}{1-\gamma}}\sum_{k=1}^K\gamma^{K-k} \left(\|(\hatTzf -\Tzf) Z_{k-1}\|_{2,\mu} + \|(\Tzf - \Tz)Z_{k-1}\|_{2,\mu}\right)\\
    &\stackrel{\textup{Lemma~\ref{lemma:empirical-loss-estimate}}}{\le} \frac{2\gamma^K}{(1-\gamma)^2} + \frac{2C}{1-\gamma}\gamma\beta^{-1}e^{\frac{\beta}{1-\gamma}}\sum_{k=1}^K\gamma^{K-k} \left(4\sqrt{\frac{2\log(|\cF|)}{N}} + 5\sqrt{\frac{2\log(8/\delta)}{N}} +\epsilon_c\right)\\
    &\le \frac{2\gamma^K}{(1-\gamma)^2} + \gamma\beta^{-1}e^{\frac{\beta}{1-\gamma}}\frac{2C}{(1-\gamma)^2}\left(4\sqrt{\frac{2\log(|\cF|)}{N}} + 5\sqrt{\frac{2\log(8/\delta)}{N}} +\epsilon_c\right)
\end{align*}

\end{proof}

\section{Auxiliaries}

\begin{lemma}\label{lemma:Q-smoothness}
    \begin{equation*}
        |Q^{\pi'}(s_0, a_0) - Q^{\pi}(s_0,a_0)|\le \frac{1}{(1-\gamma)^2}\max_s \|\pi'_s - \pi_s\|_1,
    \end{equation*}
    where $\pi_s$ denotes $\pi(\cdot|s) \in \Delta^{|\cS|}$ and $ \|\pi'_s - \pi_s\|_1 = \sum_a |\pi'(a|s) - \pi(a|s)|$.
\begin{proof}
\begin{align*}
    Q^{\pi'}(s_0, a_0) &= \min_{\{\widehat P_t\}_{t\ge1}}\bE_{s_t,a_t\sim \pi',\widehat P} \sum_{t=0}^{+\infty} \gamma^t \left(r(s_t,a_t) + \gamma d(\widehat{P}_{t;s_t,a_t}, P_{s_t, a_t}) \right)\\
    &\le \bE_{s_t,a_t\sim\pi',\hP\pi} \sum_{t=0}^{+\infty} \gamma^t \left(r(s_t,a_t) + \gamma d(\widehat{P}_{t;s_t,a_t}, P_{s_t, a_t}) \right).\\
    Q^\pi(s_0, a_0) &= \bE_{s_t,a_t\sim\pi,\hP\pi}\sum_{t=0}^{+\infty} \gamma^t \left(r(s_t,a_t) + \gamma d(\widehat{P}_{t;s_t,a_t}, P_{s_t, a_t}) \right)
\end{align*}
Thus
\begin{align*}
    & Q^{\pi'}(s_0, a_0) - Q^{\pi}(s_0,a_0) \le\bE_{s_t,a_t\sim\pi',\hP\pi} \sum_{t=0}^{+\infty} \gamma^t \left(r(s_t,a_t) + \gamma d(\widehat{P}_{t;s_t,a_t}, P_{s_t, a_t}) \right)\\
    &\qquad\qquad\qquad\qquad\qquad\quad-\bE_{s_t,a_t\sim\pi,\hP\pi} \sum_{t=0}^{+\infty} \gamma^t \left(r(s_t,a_t) + \gamma d(\widehat{P}_{t;s_t,a_t}, P_{s_t, a_t}) \right)\\
    &= \bE_{s_t\sim \pi',\widehat P} \sum_{t=0}^{+\infty} \gamma^t \sum_a (\pi'(a|s) - \pi(a|s)) Q^\pi(s,a) \quad \textup{(by Lemma \ref{lemma:performance-difference-soft-MDP})}\\
    &\le \frac{1}{1-\gamma}\bE_{s_t\sim \pi',\widehat P} \sum_{t=0}^{+\infty} \gamma^t \sum_a |\pi'(a|s) - \pi(a|s)|\\
    &\le \frac{1}{(1-\gamma)^2} \max_s\|\pi_s' - \pi_s\|_1
\end{align*}
\end{proof}
\end{lemma}

\begin{lemma}\label{lemma:projection-property-1}
    For any convex set $\cX\subset \bR^n$ and $x\in \cX, f\in \bR^n, \eta > 0$,
    \begin{equation*}
        \< \projX(x+\eta f), f \> \ge \frac{1}{\eta}\|\projX(x+\eta f) - x\|_2^2. 
    \end{equation*}
    \begin{proof}
        From the definition of projection, for any $y\in \cX$,
        \begin{align*}
        \<y - \projX(x+\eta f), x+\eta f - \projX(x + \eta f)\>\le 0,
        \end{align*}
    Set $y = x$ we get:
    \begin{align*}
          \|\projX(x+\eta f) - x\|_2^2 \le \< \projX(x+\eta f), \eta f \>,
    \end{align*}
    which completes the proof.
    \end{proof}
\end{lemma}

% \begin{lemma}\label{lemma:projection-property-2}
%     \textcolor{blue}{TBD}
% \end{lemma}

\begin{lemma}\label{lemma:Bellman-operator-contraction}
The operators $\Tq, \Tvq$ defined in \eqref{eq:Tq}, \eqref{eq:Tvq} satisfies
\begin{align*}
    [\Tvq(\oV - V)](s,a) &\le \gamma \hEtrans{} (\oV-V)(s')\\
    [\Tq(\oQ - Q)](s,a) &\le \gamma\hEtrans{} \max_{a'} (\oQ - Q)(s',a'),
\end{align*}
where $\widehat{P}_{s,a}$ is defined as:
\begin{align*}
    \widehat P(s'|s,a) \propto P(s'|s,a) \exp(-\beta V(s')) ~~\textup{ or }~~ \widehat P(s'|s,a) \propto P(s'|s,a) \exp(-\beta \max_{a' }Q(s',a')).
\end{align*}
\begin{proof}
Let $\sigma(\Psa, V):= \beta^{-1}\log\bE_{s'\sim\Psa} e^{-\beta V(s')}$, then from Example \ref{example:entropy-risk-measure} and the dual representation theorem \ref{thm:convex-risk-dual-representation}
\begin{align*}
   \sigma(\Psa, V):= \sup_{\widetilde P_{s,a}} -\bE_{s'\sim\widetilde P_{s,a}} V(s') - \beta^{-1}\KL(\widetilde P_{s,a}||P_{s,a}). 
\end{align*}
\begin{align}
    &\quad [\Tvq(\oV - V)] = -\gamma \left(\sigma(\Psa, \oV) - \sigma(\Psa, V)\right)\label{eq:proof-1}\\
    &=\gamma\left(\inf_{\widetilde P_{s,a}} \left(\bE_{s'\sim\widetilde P_{s,a}} \oV(s') + \beta^{-1}\KL(\widetilde P_{s,a}||P_{s,a})\right)-\inf_{\widetilde P_{s,a}}\left(\bE_{s'\sim\widetilde P_{s,a}} V(s') + \beta^{-1}\KL(\widetilde P_{s,a}||P_{s,a})\right) \right)\notag\\
    &= \gamma\left(\inf_{\widetilde P_{s,a}} \left(\bE_{s'\sim\widetilde P_{s,a}} \oV(s') + \beta^{-1}\KL(\widetilde P_{s,a}||P_{s,a})\right)-\left(\bE_{s'\sim\widehat P_{s,a}} V(s') + \beta^{-1}\KL(\widehat P_{s,a}||P_{s,a})\right) \right)\notag\\
    &\le \gamma \left(\left(\bE_{s'\sim\widehat P_{s,a}} \oV(s') + \beta^{-1}\KL(\widehat P_{s,a}||P_{s,a})\right)-\left(\bE_{s'\sim\widehat P_{s,a}} V(s') + \beta^{-1}\KL(\widehat P_{s,a}||P_{s,a})\right) \right)\notag\\
    &= \gamma \bE_{s'\sim\widehat P_{s,a}} (\oV-V)(s')\label{eq:proof-2}
\end{align}

Similarly,
\begin{align*}
   [\Tq (\oQ - Q)](s,a) &= -\gamma\beta^{-1}\log \bE_{s'\sim P(\cdot|s,a)} e^{-\beta\max_{a'}\oQ(s',a')} + \gamma\beta^{-1}\log \bE_{s'\sim P(\cdot|s,a)} e^{-\beta\max_{a'}Q(s',a')}\\
  &= -\gamma\left(\sigma(\Psa, \max_a Q(\cdot, a)) - \sigma(\Psa, \max_a \oQ(\cdot, a))\right)
  % &\le\gamma \left\|\max_a Q(\cdot, a)) - \max_a \oQ(\cdot, a))\right\|_\infty\\
  % &\le \gamma\|\|
\end{align*}
Using the same inequality from \eqref{eq:proof-1} to \eqref{eq:proof-2} we get
\begin{align*}
    -\gamma\left(\sigma(\Psa, \max_a Q(\cdot, a)) - \sigma(\Psa, \max_a \oQ(\cdot, a))\right) &\le \bE_{s'\sim\widehat P_{s,a}}\max_a \left(\oQ(s', a) - \max_a \oQ(s', a)\right)\\
    &\le \bE_{s'\sim\widehat P_{s,a}}\max_a \left(\oQ(s', a) - \oQ(s', a)\right),
\end{align*}
which completes the proof.
\end{proof}
\end{lemma}

\begin{lemma}\label{lemma:log}
\begin{equation*}
    |\log x - \log y| \le \frac{1}{\min\{x,y\}}|x-y|
\end{equation*}
\begin{proof}
    Without loss of generality, we assume $x\ge y$, then
    \begin{align*}
        |\log x - \log y| = \log x - \log y = \log(1+\frac{x-y}{y})\\
        \le \frac{x-y}{y} = \frac{1}{y}(x-y) = \frac{1}{\min\{x,y\}}|x-y|,
    \end{align*}
which completes the proof.
\end{proof}
\end{lemma}

\end{document}